\documentclass[11pt]{amsart}
\usepackage{amsmath,amssymb,enumerate,etoolbox,mathrsfs,xcolor,url,hyperref}
\usepackage[normalem]{ulem}

\usepackage[top=30mm,right=30mm,bottom=30mm,left=20mm]{geometry}

\newtheorem{theorem}{Theorem}[section]
\newtheorem{thm}[theorem]{Theorem}

\newtheorem{theorema}{Theorem}

\newtheorem{prop}[theorem]{Proposition}
\newtheorem*{proposition*}{Proposition}
\newtheorem{lem}[theorem]{Lemma}
\newtheorem{cor}[theorem]{Corollary}
\newtheorem{exa}[theorem]{Example}
\newcommand{\F}{\mathbb{F}}
\newcommand{\Fqb}{\overline{\mathbb{F}}_q}

\newcommand{\Z}{\mathbb{Z}}
\newcommand{\CB}{\mathbf{C}}
\newcommand{\ZB}{\mathbf{Z}}
\newcommand{\EB}{\mathbf{E}}

\newcommand{\tG}{\tilde G}
\newcommand{\SL}{{\rm SL}}
\newcommand{\GL}{{\rm GL}}
\newcommand{\SU}{{\rm SU}}
\newcommand{\GU}{{\rm GU}}
\newcommand{\Sp}{{\rm Sp}}
\newcommand{\SO}{{\rm SO}}
\newcommand{\Spin}{{\rm Spin}}
\newcommand{\GO}{{\rm GO}}

\newcommand{\Ker}{{\rm Ker}}
\newcommand{\supp}{{\mathsf {supp}}}
\newcommand{\Irr}{{\rm Irr}}

\newcommand{\Hom}{{\rm Hom}}
\newcommand{\End}{{\rm End}}

\newcommand{\Span}{{\rm Span}}

\newcommand{\Spec}{{\rm Spec}\,}
\newcommand{\Prob}{\mathbf{P}}
\newcommand{\PSp}{\mathrm{PSp}}
\newcommand{\PSL}{\mathrm{PSL}}
\newcommand{\Cl}{{\rm Cl}}
\newcommand{\bv}{\mathbf{v}}
\newcommand{\bw}{\mathbf{w}}

\newcommand{\cG}{\mathcal{G}}

\newcommand{\cS}{\mathcal{S}}
\newcommand{\cX}{\mathcal{X}}
\newcommand{\cM}{\mathcal{M}}
\newcommand{\sU}{\mathsf{U}}

\newcommand{\sX}{\mathsf{X}}
\newcommand{\sQ}{\mathsf{Q}}
\newcommand{\uG}{\underline{G}}
\newcommand{\uS}{\underline{S}}
\newcommand{\uX}{\underline{X}}
\newcommand{\uW}{\underline{W}}
\newcommand{\Stab}{\mathrm{Stab}}
\newcommand{\uC}{\underline{C}}
\newcommand{\SSS}{\mathsf{S}}
\newcommand{\eps}{\varepsilon}
\newcommand{\diag}{\mathrm{diag}}
\newcommand{\codim}{\mathrm{codim}}
\newcommand{\ssp}{\mathrm{ss}}
\newcommand{\cl}{\mathfrak{l}}
\newcommand{\tw}[1]{{}^#1\!}
\newcommand{\edit}[1]{{\color{black} #1}}

\newcommand{\edt}[1]{{\color{black} #1}}
\numberwithin{equation}{section}

\begin{document}

\title{Uniform Character Bounds for Finite Classical Groups}

\author{Michael Larsen}
\email{mjlarsen@indiana.edu}
\address{Department of Mathematics\\
    Indiana University \\
    Bloomington, IN 47405\\
    U.S.A.}

\author{Pham Huu Tiep}
\email{pht19@math.rutgers.edu}
\address{Department of Mathematics\\
    Rutgers University \\
    Piscataway, NJ 08854-8019 \\
    U.S.A.}
\thanks{The first author was partially supported by the NSF 
grant DMS-2001349. The second author gratefully acknowledges the support of the NSF (grants
DMS-1840702 and DMS-2200850), the Joshua Barlaz Chair in Mathematics, and the Charles Simonyi Endowment at the 
Institute for Advanced Study (Princeton).}
\thanks{The authors are grateful to Martin Liebeck and Persi Diaconis for helpful comments on the paper.}
\thanks{\edit{The authors are grateful to the referee for careful reading and many comments and suggestions that helped greatly improve
the exposition of the paper.}} 

\begin{abstract}
For every finite quasisimple group of Lie type $G$, every irreducible character $\chi$ of $G$, and every element $g$ of $G$, we give an exponential upper bound for the
character ratio $|\chi(g)|/\chi(1)$ with exponent linear in $\log_{|G|} |g^G|$, or, equivalently, in the ratio of the support of $g$ to the rank of $G$.
We give several applications, including a proof of Thompson's conjecture for all sufficiently large simple symplectic groups, orthogonal groups in characteristic $2$, and some other infinite families of orthogonal and unitary groups.\end{abstract}

\maketitle

\tableofcontents

\section{Introduction}
Let $G$ be a \edit{finite group}, $g$ an element of $G$, and $\chi$ an irreducible character of $G$.
As $\chi(g)$ is a sum of $\chi(1)$ roots of unity, $|\chi(g)|\le \chi(1)$,
and when $g$ lies in the center of $G$, no better upper bound is possible.
The goal of this paper is to provide a good bound for $|\chi(g)|$ in terms of $\chi(1)$ and $|g^G|$ over
the whole range of character degrees and conjugacy class  sizes, \edit{which applies to all finite quasisimple groups $G$ of Lie type,
(i.e. $G=[G,G]$ and $G/\ZB(G)$ is a finite simple group of Lie type)}. Our main result is the following:

\begin{theorema}
\label{main}
There exists an absolute constant $c>0$ such that for all finite quasisimple groups $G$ of Lie type, irreducible characters $\chi$ of $G$,
and elements $g\in G$, we have
\begin{equation}
\label{main bound}
|\chi(g)| \le \chi(1)^{1 - c\frac{\log |g^G|}{\log |G|}}.
\end{equation}
\end{theorema}


There are already many bounds for irreducible character values in the literature.
For groups of Lie type, Gluck \cite{Gl} bounded the character ratio $|\chi(g)|/|\chi(1)|$ away from $1$ for all non-central $g$.  
\edt{For classical groups of Lie type, one can say more for almost all elements.  We define the \emph{support} $\supp(g)$ of an element $g\in \GL_n(\bar\F_q)$ to be the codimension
of the eigenspace of $g$ of maximal dimension.  This leads naturally to a definition for the support of an element $g$ of a classical finite group of Lie type; namely, we lift $g$ to an element
of its central extension which lies in $\GL_n$.
In \cite{LST}, the character ratio is shown to go to $0$ as $\supp(g)\to \infty$.}  
The papers \cite{GLT, GLT2} give
\emph{exponential character bounds}, i.e., upper bounds of the form $\chi(1)^\alpha$ as long as $\frac{\log |g^G|}{\log |G|}$ is close enough to $1$.
The exponents in these bounds go to $0$ as $\frac{\log |g^G|}{\log |G|}\to 1$, so in this regime, the bounds are better than those of Theorem~\ref{main}.
A bound of  type \eqref{main bound} is given in \cite{BLST} for many classes of elements,
and this has been further extended in \cite{TT}, \edit{yielding in particular optimal bounds for semisimple elements, whose centralizer is 
a proper Levi subgroup}.  Furthermore, good character bounds for the exceptional groups of Lie type, which all have bounded rank, are provided in \cite{LiT}.

The strength of our paper is that it gives an exponential bound covering {\it all} \edit{elements and all characters}. This is particularly valuable for applications involving the Frobenius formula, where the most
difficult cases cannot be excluded. We also note that, up to a multiplicative constant, the exponent in the bound 
$|\chi(g)/\chi(1)| \leq \chi(1)^{-c\log_{|G|}|g^G|}$
in Theorem~\ref{main} is optimal; see Examples \ref{sln2}, \ref{Steinberg}, and Lemma \ref{lower}. \edit{Furthermore, the constant $c$ is 
made explicit in the proof.}

We remark that there has been a parallel effort to \edit{obtain exponential bounds for irreducible character values of} symmetric (and alternating) groups;
see, for instance, \edit{\cite{FL,Ro,MS,RS,LaSh1,LiM}}.

\smallskip
Previous character bounds have seen a wide variety of interesting applications.  They play an important role in the proof of Ore's conjecture \cite{LOST}, versions of Waring's problem for finite simple groups \cite{LST, LST2}, covering number computations for conjugacy classes \cite{LiSh2}, and estimates for the number of points of representation varieties over finite fields \cite{LiSh4}.
Additional applications are described in Liebeck's survey article \cite{Li}.

\smallskip
We present several applications illustrating the power of the new bounds.
Thompson's conjecture \cite{AH} asserts that for every finite simple group $G$, there exists a conjugacy class $S$ such that $S^2 = G$.
Ellers and Gordeev \cite{EG} made substantial progress on this conjecture, leaving open the case of groups of Lie type  with $q\le 8$; also, they completely settled the case of groups of type $A_r$. 
In this paper, we give an asymptotic treatment of the case $C_r$.  Likewise, we treat $D_r$ and ${}^2D_r$ either in characteristic $2$ or in odd characteristic,
where $q$ satisfies a suitable condition (mod $4$).
That is, we show that Thompson's conjecture holds for all but finitely many such groups; see Theorem \ref{thompson}.
With finitely many exceptions, all that now remains are certain unitary groups with $q\le 7$ as well as odd-dimensional orthogonal groups 
and certain even-dimensional orthogonal groups over $\F_3$ and $\F_5$.
We also prove that various regular semisimple conjugacy classes $S$ in \edit{$G = \SL_n(q)$ or $\SU_n(q)$}, including all classes with irreducible characteristic polynomial, have the property that $S^2$ includes all
elements of $G$ whose support is larger than an absolute constant.

The mixing time for a random walk on a Cayley graph given by a conjugacy class $S$ has been an object of study since the celebrated work of Diaconis and Shahshani \cite{DS}.
For finite simple groups of Lie type, Liebeck and Shalev \cite[Corollary 1.14]{LiSh2} gave the upper bound $O\bigl(\frac {\log^3|G|}{\log^2|S|}\bigr)$. In Theorem \ref{mixing}, we  improve this to 
the optimal asymptotic, $O\bigl(\frac {\log|G|}{\log|S|}\bigr)$,  settling conjectures of Lubotzky \cite[p.179]{Lub} and 
of Shalev \cite[4.3]{Sh} in the affirmative. 

Each non-trivial character $\chi$ of a finite group $G$ determines a McKay graph on the vertices $\Irr(G)$.  Liebeck, Shalev, and Tiep \cite[Conjecture 1]{LiST} conjectured that the diameter of this graph is $O\bigl(\frac{\log |G|}{\log \chi(1)}\bigr)$ for all finite simple groups.
We prove this conjecture, as well as a related conjecture of Gill \cite{Gi} concerning products of irreducible complex characters, 
for all finite simple groups of Lie type, see Theorems \ref{mckay} and \ref{mckay2}. We also extend results of Fulman \cite{F} to 
determine the asymptotic of the convergence rate (to the stationary distribution) for random walks on McKay graphs for any simple
group of Lie type, see Theorem \ref{mixing3}.

The non-commutative Waring problem has received considerable attention recently, see e.g. \cite{LaSh2,LST,GLBST,LST2}. In particular, \cite[Theorem 4]{GLBST} states that there is a function $f:\Z_{\geq 1} \to \Z_{\geq 1}$ such that if $N\ge 1$ is an integer with at most $k$ prime factors,
the power word map $(x,y) \mapsto x^Ny^N$ is surjective on any alternating group $\mathsf{A}_n$ with $n \geq f(k)$, and on any simple group of Lie type 
of rank $n \geq f(k)$, excluding types $A$ and $\tw2 A$.  In \cite{GLBST}, it is asked whether the theorem extends to the excluded cases.  In Theorem \ref{burnside}, we prove that it does.

If $\uG$ is a simple algebraic group over any algebraically closed field $K$ and $\uS_1,\ldots,\uS_k$ are conjugacy classes of $\uG$, then multiplication defines a morphism 
of varieties $\uX:=\uS_1\times \cdots\times \uS_k\to \uG$.  We prove that if $\dim \uX$ is at least an absolute constant multiple of $\dim \uG$, then this morphism is flat,
meaning, in this situation, that 
the fibers of the morphism all have the same dimension, $\dim \uX-\dim \uG$, see Theorem \ref{fibers}.

\smallskip
Our strategy for proving the main theorem reverses the usual order of things.  Instead of using character estimates to prove mixing theorems, we use mixing theorems to prove character estimates.
More precisely, we use probability-theoretic methods to show that if $U_S$ is the uniform distribution on a very small conjugacy class $S$, the probability that a sample from the iterated convolution $U_S^{\ast b}$ lands in a small conjugacy class is very low.
We do this in two stages, first to obtain (with probability close to one), an element whose support is larger than a constant multiple of the dimension $n$ of the natural representation of $G$, and then
to obtain an element whose conjugacy class is large enough that the estimates of
\cite{GLT,GLT2} apply.

\smallskip
The paper is organized as follows.  In section 2, we prove  basic combinatorial estimates.  In section 3, we translate these into the 
probability-theoretic results necessary to bootstrap from elements of large support (satisfying a linear lower bound in $\mathrm{rank}(G)$)
to elements of small centralizer (satisfying an exponential upper bound in $|G|$ which can be taken as small as we wish).  
The bootstrapping argument is carried out in sections 4 and 5 and produces a uniform exponential 
character bound in terms of the support, Theorem \ref{main-bound3}. In section 6, we compare two notions of smallness for a conjugacy class given by support and by class size, and this allows us to deduce Theorem \ref{main} from Theorem \ref{main-bound3}. 
The applications to squares of conjugacy classes are given in section 7, and we conclude, in section 8, with the applications to mixing time, McKay graph diameter and products of irreducible complex characters, power word maps on simple groups,
and flatness of product morphisms.

\section{Counting lemmas}
The key result in this section is Proposition~\ref{tuples}, which given a classical group $G$ acting on a vector space $V$, 
a subspace $U\subset V$, a fixed element of $G$, a fixed integer $b$,
and a fixed polynomial $P(x)\in \F_q[x]$, bounds above the number of different ways it can happen that $P$ evaluated at a product of $b$ conjugates $g^{x_1},\ldots,g^{x_b}$ of $g$ annihilates $U$.
When this occurs, for each basis vector $u_i$ of $U$, we can track all the vectors appearing along the way in computing $P(g^{x_b}\cdots g^{x_1}) u_i$ and count tuples encoding all that intermediate information.  We use this to estimate the size of the projection to $(x_1,\ldots,x_b)$.  To accomplish this, we need various estimates for orbit sizes for classical group actions.

Let $q$ be a power of a prime $p$, $n \in \Z_{\geq 3}$, and $V = \F_q^n$. In what follows, by a {\it classical group $G=\Cl_n(q) = \Cl(V)$
on $V$} and its {\it dimension $D$} we specifically mean one of the following:

\begin{itemize}

\item $G=\SL(V)$ and $D=n^2-1$;
\item $\Sp(V)$ and $D=n(n+1)/2$, if $2|n$ and $V$ is endowed with a non-degenerate alternating bilinear form $(\cdot|\cdot)$;
\item $\SO(V)$ and $D=n(n-1)/2$, if $p > 2$ and $V$ is endowed with a non-degenerate symmetric bilinear form $(\cdot|\cdot)$;
\item $\Omega(V)$ and $D=n(n-1)/2$, if $p = 2|n$ and $V$ is endowed with a quadratic form $\sQ$, associated with a non-degenerate 
alternating bilinear form $(\cdot|\cdot)$;
\item $\SU(V)$ and $D=(n^2-1)/2$, if $q=q_0^2$ is a square and $V$ is endowed with a non-degenerate ($\F_{q_0}$-bilinear) Hermitian form $(\cdot|\cdot)$ so that $\SU(V) \cong \SU_n(q_0)$.
\end{itemize}
\edt{(See e.g. \cite[Chapter 2]{KlL} for definitions and basic facts on the associated forms for finite classical groups.)}
This \edt{convention} ensures that $q^D > |G| > q^D/2$ (cf. \cite[Lemma 4.1(ii)]{LMT}).
\edt{We remark that for unitary groups, we do not always denote the two relevant prime powers $q_0$ and $q=q_0^2$, sometimes
preferring $q$ and $q^2$, depending on whether the emphasis is on the field of definition of the algebraic group or on the field of definition of the natural representation.
Note also that  $\SU_n(q_0)$ and $\SU(\F_q^n)$ (and sometimes $\Cl_n(q)$ with specifying $\Cl=\SU$) are different names for the same group, and the same can be said for
$\SU_n(q)$ and $\SU(\F_{q^2}^n)$ (and $\Cl_n(q^2)$, with specifying $\Cl=\SU$).}

If $V_1,\ldots,V_k$ are vector spaces over a finite field $\F$, and $g_1\in \GL(V_1),\ldots,g_k\in \GL(V_k)$,
we denote by $\diag(g_1,\ldots,g_k)$ the image of $(g_1,\ldots,g_k)$ under the homomorphism 
$$\GL(V_1)\times \cdots\times \GL(V_k)\to \GL(V_1\oplus\cdots\oplus V_k).$$
We use the same notation for classical groups; for instance, if $g_i\in \Sp(V_i)$, we understand $\diag(g_1,\ldots,g_k)$ to be an element of $\Sp(V_1\oplus \cdots\oplus V_k)$.


\begin{lem}
\label{count-v}
If $k$ and $n$  are positive integers and $r\le k$ is a non-negative integer, $V=\F_q^n$, 
and $w_1,\ldots,w_k$ are linearly independent vectors in $V$,
then the number of sequences of vectors $v_1,\ldots,v_k\in V$ such that 
$$\dim \Span(v_1,\ldots,v_k,w_1,\ldots,w_k) = k+r$$
is less than
$$\binom{k}{r} q^{rn + k^2-r^2}.$$
\end{lem}

\begin{proof}
\edt{Any such sequence $(v_1,\ldots,v_k)$, determines
an $r$-subset} $S\subseteq \{1,\ldots,k\}$ such that $s\in S$ if and only if
$$v_s\not\in \Span(v_1,\ldots,v_{s-1},w_1,\ldots,w_k).$$
\edt{We will bound the number of such sequences by first fixing $S$, for which we have} 
$\binom kr$ possibilities.  Given $S$, there are less than $q^n$ possibilities for $v_s$ when $s\in S$ and $q^{k+r}$ possibilities for $v_s$ for each of the $k-r$ indices $s \notin S$.
\end{proof}

\begin{lem}\label{order}
Let $U$ be a subspace of $V= \F_q^n$ of dimension $d \leq (n-3)/2$, 
and let $H$ denote the subgroup of all elements in $G=\Cl(V)$ that act trivially on $U$. 
Then $|H| < q^{D-dn}$ if $\Cl=\SL$ and 
$|H| \leq q^{D-dn+d(d+1)/2}$ otherwise.
\end{lem}

\begin{proof}
If $G=\SL_n(q)$, then 
$$|H|=q^{d(n-d)}|\SL_{n-d}(q)| < q^{d(n-d)+(n-d)^2-1} = q^{D-dn}.$$

We will now consider the case $\Cl \neq \SL$ and so $V$ is endowed with $G$-invariant 
(bilinear or Hermitian) form $(\cdot|\cdot)$ (and quadratic form $\sQ$ when $p=2$ and $\Cl=\Omega$). 
Let $\kappa:=-1$ if $p>2$ and $\Cl=\Sp$, and $\kappa:=1$ otherwise, \edt{ and 
set 
$$W:=U \cap U^\perp,~a:= \dim W,~b:=d-a.$$
Consider the $H$-invariant (partial) flag 
$$\{0\} \subseteq W \subseteq U \subseteq W^\perp \subseteq V.$$
Note that $W^\perp = U + U^\perp$, and $(\cdot|\cdot)$ induces a non-degenerate bilinear form on $W^\perp/W$ (of the same kind).
With respect to this induced form, $U/W$ is a non-degenerate subspace with orthogonal complement $U^\perp/W$. So  we can find a basis
$$(e_1, \ldots,e_a,g_1, \ldots,g_b,h_1,\ldots,h_{n-2a-b})$$
of $W^\perp$, such that 
$$U = \Span(e_1, \ldots,e_a,g_1, \ldots,g_b),~U^\perp = \Span(e_1, \ldots,e_a,h_1,\ldots,h_{n-2a-b}).$$
and moreover the Gram matrix of the form induced by $(\cdot|\cdot)$ on $W^\perp$ in this basis is 
$\begin{pmatrix} 0 & 0 & 0\\0 & B & 0\\0 & 0 & C\end{pmatrix}$ with $\det(B)\det(C) \neq 0$. In particular, 
$S:=\Span(g_1, \ldots,g_b,h_1,\ldots,h_{n-2a-b})$ is a non-degenerate subspace of $V$. So $S^\perp$ is also a non-degenerate subspace
of $V$ of dimension $2a$, which contains $W$ as a maximal totally singular subspace. Hence we extend $(e_1,\ldots,e_a)$ to a basis
$(e_1, \ldots,e_a,f_1, \ldots,f_a)$ of $S^\perp$ in which $(\cdot|\cdot)$ has the Gram matrix $\begin{pmatrix}0 & I_a\\ \kappa I_a & 0\end{pmatrix}$. 
Thus the Gram matrix of $(\cdot|\cdot)$ in the basis 
$$(e_1, \ldots,e_a,g_1, \ldots,g_b,h_1,\ldots,h_{n-2a-b},f_1, \ldots,f_a)$$
of $V$ is
\begin{equation}\label{gram10}
  \begin{pmatrix} 0 & 0 & 0 & I_a\\0 & B & 0 & 0\\0 & 0 & C & 0\\ \kappa I_a & 0 & 0 & 0\end{pmatrix}.
\end{equation}  

Consider any element $h \in H$. Then $h$ acts trivially on $U$ and on $V/W^\perp \cong W^*$, and preserves the orthogonal complement $U^\perp/W$ to $U/W$ in 
$W^\perp/W$, whence we get a homomorphism
$$\varphi:H \to \Cl(U^\perp/W) \cong \Cl_{n-2a-b}(q).$$
}
We will bound $|\varphi(H)|$ and $|\Ker(\varphi)|$, representing each $x \in \Ker(\varphi)$ 
by the matrix 
$$\begin{pmatrix} I_a & 0 & X & Y\\0 & I_b & 0 & Z\\0 & 0 & I_{n-2a-b} & T\\ 0 & 0 & 0 & I_a\end{pmatrix}$$
in the chosen basis, where the matrices $X,Y,Z,T$ over $\F_q$ satisfy  the conditions 
\begin{equation}\label{for-x1}
  Z=0,~\kappa X + \tw t T^*C=0,~\kappa Y + \tw tY^*+\tw tT^*CT=0
\end{equation}
\edt{which are obtained using the fact that $x$ preserves $(\cdot|\cdot)$ with Gram matrix \eqref{gram10}. For instance, $Z=0$ because, for any $v \in V$ and $u \in U$,
\begin{equation}\label{gram11}
  (x(v)-v|u) = (x(v)|u)-(v|u) = (x(v)|x(u))-(v|u)=0,
\end{equation}  
i.e. $x(v)-v \in U^\perp$.
}
Here, for a matrix $A=(a_{ij})$ over $\F_q$, $A^*$ is interpreted as $A$ in the case $\Cl=\Sp$, $\SO$, or $\Omega$ and 
as $A^{(q_0)}=(a_{ij}^{q_0})$ in the case $\Cl=\SU$.

\smallskip
(a) Suppose $\Cl=\Sp$. Then $2|b$ and $2|n$, and writing $b=2c$ and $n=2m$ we have 
$$|\varphi(H)| \leq |\Sp_{2m-2a-2c}(q)| < q^{(m-a-c)(2(m-a-c)+1)}.$$
For $x \in \Ker(\varphi)$, there are $q^{2a(m-a-c)}$ choices for $T$, and, by \eqref{for-x1}, 
for each choice of $T$, $X$ is uniquely determined
and there are $q^{a(a+1)/2}$ choices for $Y$. Thus $|H| < q^E$, with
$$E := (m-a-c)(2(m-a-c)+1)+ 2a(m-a-c)+a(a+1)/2 = D-2dm+d(d-1)/2,$$
since $D=m(2m+1)$ and $d=a+2c$.

\smallskip
(b) Suppose $\Cl=\SU$. Then 
$$|\varphi(H)| \leq |\SU_{n-2a-b}(q_0)| < q_0^{(n-2a-b)^2-1)} = q^{((n-2a-b)^2-1)/2}.$$
For $x \in \Ker(\varphi)$, there are $q^{a(n-2a-b)}$ choices for $T$, and, by \eqref{for-x1}, 
for each choice of $T$, $X$ is uniquely determined
and there are $q_0^{a^2}=q^{a^2/2}$ choices for $Y$. Thus $|H| < q^E$, with
$$2E := (n-2a-b)^2-1+ 2a(n-2a-b)+a^2 = 2D-2kn+d^2,$$
since $D=(n^2-1)/2$ and $d=a+b$.

\smallskip
(c) Suppose $\Cl=\SO$ or $\Omega$. Then $n-3 \geq 2d=2a+2b$, whence $n-2a-b \geq 3$, and 
$$|\varphi(H)| \leq |\SO_{n-2a-b}(q)| < q^{(n-2a-b)(n-2a-b-1)/2}.$$
For $x \in \Ker(\varphi)$, there are $q^{a(n-2a-b)}$ choices for $T$, and, by \eqref{for-x1}, 
for each choice of $T$, $X$ is uniquely determined
and, if $p>2$ then there are at most $q^{a(a-1)/2}$ choices for $Y$, so 
\begin{equation}\label{for-x2}
  |\Ker(\varphi)| \leq q^{a(n-2a-b)+a(a-1)/2}.
\end{equation}  
We want to show that \eqref{for-x2} also holds when $p=2$. Indeed, suppose
$p=2$. Then the number of elements in $\Ker(\varphi)$ that correspond to the same $T$ is the number of elements 
in $\Ker(\varphi)$ that have $T=0$. In addition to \eqref{for-x1} which gives $\tw tY+Y=0$, i.e. $Y=(y_{ij})$ is symmetric,
any such element must satisfy
\begin{equation}\label{for-x3}
  \sQ(f_j)=\sQ(x(f_j))=\sQ(f_j+\sum_i y_{ij}e_i)=\sQ(f_j)+y_{jj}+\sQ(\sum_i y_{ij}e_i).
\end{equation}  
\edt{Since $p=2$, every scalar $z \in \F_q$ has a unique square root in $\F_q$, and hence the map 
$$\sqrt{\sQ}:V \to \F_q,~v \mapsto \sqrt{\sQ(v)}$$ 
is well-defined.  
Furthermore,} $\sqrt{\sQ}$ is $\F_q$-linear on $\Span(e_1, \ldots,e_a)$, so we may assume that 
$\sQ(e_1) = \ldots = \sQ(e_{a-1})=0$, and $\sQ(e_a)=0$ or $\sQ(e_a)=1$. In the former case, 
\eqref{for-x3} yields $y_{jj}=0$, whence $Y$ has zero main diagonal and so the number of such $Y$ is 
$q^{a(a-1)/2}$, and thus \eqref{for-x2} holds. In the latter case, \eqref{for-x3} yields $y_{jj}=y_{aj}^2$ 
when $j < a$, and $y_{aa}+y_{aa}^2=0$, i.e. $y_{aa} = 0$ or $1$. Thus, when $T$ is fixed, there are (at most)
$2q^{a(a-1)/2}$ choices for $x \in \Ker(\varphi)$. Next, writing $n-2a-b=:2c$, we can choose $(h_1, \ldots,h_{2c})$ 
so that the Gram matrix $C$ is 
$\begin{pmatrix}0 & 0 & \ldots & 0 & 1\\ 0 & 0 & \ldots & 1 & 0\\ & & \ldots & &\\1 & 0 & \ldots & 0 & 0 \end{pmatrix}$. Then
$$\sQ(h_j)= \sQ(x(h_j)) = \sQ(h_j+\sum_i x_{ij}e_i)=\sQ(h_j)+\sQ(\sum_i x_{ij}e_i) = \sQ(h_j)+x_{aj}^2,$$
and so $x_{aj}=0$. On the other hand, $X=\tw t TC$ by \eqref{for-x1}, hence 
$$0 = x_{aj} = \bigl( \tw t TC)_{aj} = t_{2c+1-j,a}$$
for $1 \leq j \leq 2c$. These relations show that there are at most $q^{(a-1)(n-2a-b)}$ choices for $T$, and, 
for each choice of $T$, $X$ is uniquely determined and there are at most $2q^{a(a-1)/2}$ choices for $Y$. Since 
$q^{n-2a-b} \geq 2$, \eqref{for-x2} holds in this case as well.

Thus we always have $|H| \leq q^E$, with
$$2E := (n-2a-b)(n-2a-b-1)+ 2a(n-2a-b)+a(a-1) = 2D-2kn+d(d+1),$$
since $D=n(n-1)/2$ and $d=a+b$.
\end{proof}

\edt{
\begin{lem}\label{orbit1}
Let $k \in \Z_{\geq 1}$ and let $q$ be any prime power. Suppose that 
either $K=\F_q^k$ is endowed with a nonzero alternating bilinear form $(\cdot|\cdot)$, or a quadratic form $\sQ$
associated to a nonzero symmetric bilinear form $(\cdot|\cdot)$, or $q$ is a square and $K=\F_q^k$ is endowed with a nonzero 
Hermitian form $(\cdot|\cdot)$. Let $K^\perp$ denote the radical of $(\cdot|\cdot)$, and let 
$v \in K \smallsetminus K^\perp$. Then the set $\Omega(v)$ of the vectors $u \in K \smallsetminus K^\perp$ with
$(v|v)=(u|u)$, respectively, $\sQ(v)=\sQ(u)$, $(v|v)=(u|u)$, has cardinality $\geq q^{k-2}$.
\end{lem}

\begin{proof}
(a) First we consider the case $(\cdot|\cdot)$ is non-degenerate, i.e. $K^\perp=0$.
The bound on $|\Omega(v)|$ is obvious if $k \leq 2$, so we may assume $k \geq 3$.
Let $\tG$ denote the full isometry group of $(\cdot|\cdot)$, respectively of 
$\sQ$, $(\cdot|\cdot)$. By Witt's lemma \cite[Proposition 2.1.6]{KlL}, $\Omega(v)$ is just the orbit $v^{\tG}$ of $v$. In the symplectic case, we have 
$$|v^{\tG}|=|K \smallsetminus \{0\}| = q^k-1 > q^{k-1}.$$

In the odd-dimensional orthogonal case, we have $2 \nmid q$ and $k=2m+1$ for some $m \geq 1$. Then $\tG=\GO_{2m+1}(q)$ has 
one orbit of isotropic vectors of length $q^{2m}-1$, $(q-1)/2$ orbits of anisotropic vectors of length $q^m(q^m-1)$ each, and 
$(q-1)/2$ orbits of anisotropic vectors of length $q^m(q^m+1)$ each, see e.g. \cite[\S4.1]{KlL}, and any of these lengths is 
at least $q^{2m-1}$.

In the even-dimensional orthogonal case, we have $k=2m$ with $m \geq 2$. Then $\tG=\GO^\eps_{2m}(q)$ for some $\eps=\pm$. If $2 \nmid q$, then $\tG$ has 
one orbit of isotropic vectors of length $(q^{m}-\eps)(q^{m-1}+\eps)$, $(q-1)/2$ orbits of anisotropic vectors of length $(q^m-\eps)(q^{m-1}+1)$ each, and 
$(q-1)/2$ orbits of anisotropic vectors of length $(q^m-\eps)(q^{m-1}+1)$ each, again see e.g. \cite[\S4.1]{KlL}, and any of these lengths is 
larger than $q^{2m-2}$. If $2|q$, then $\tG$ has 
one orbit of isotropic vectors of length $(q^{m}-\eps)(q^{m-1}+\eps)$, and $q-1$ orbits of anisotropic vectors of length $q^{m-1}(q^m-\eps)$ each, 
and any of these lengths is larger than $q^{2m-2}$.

In the unitary case, we have $q=q_0^2$. With $\eps:=-1$, $\tG=\GU_{k}(q_0)$ has one orbit of isotropic vectors of length $(q_0^{k}-\eps^k)(q_0^{k-1}+\eps^k)$, and 
$q_0-1$ orbits of anisotropic vectors of length $q_0^{k-1}(q_0^k-\eps^k)$ each, again see e.g. \cite[\S4.1]{KlL}, and any of these lengths is 
larger than $q_0^{2k-2}=q^{k-1}$.

\smallskip
(b) Now we consider the case the radical $K^\perp$ of $(\cdot|\cdot)$ has dimension $a \geq 1$ over $\F_q$. Then $(\cdot|\cdot)$ induces a 
non-degenerate form on $K/K^\perp$ of dimension $k-a \geq 1$. If we are in the orthogonal case 
with $2|q$, assume in addition that $\sQ$ is identically zero on $K^\perp$. Then note that all the $q^a$ vectors in the coset $v+K^\perp$ belongs to 
$\Omega(v)$. Applying (a) to $K/K^\perp$, we see that $|\Omega(v)| \geq q^{k-a-2}q^a =q^{k-2}$.

Assume now that we are in the orthogonal case with $2|q$ but $\sQ$ is not identically zero on $V^\perp$. As mentioned in the proof of Lemma \ref{order},
$\sqrt{\sQ}:V^\perp \to \F_q$ is $\F_q$-linear and nonzero, hence surjective. Fixing $w \in V^\perp$ with
$\sQ(w)=1$, we see that each coset $u+\langle w \rangle_{\F_q}$ with $u \in K \smallsetminus K^\perp$ contains a unique point in $\Omega(v)$. It follows
that $|\Omega(v)|=(q^k-q^a)/q \geq q^{k-2}$.
\end{proof}
}

\begin{lem}\label{orbit2}
Let $U$ be a subspace of $V= \F_q^n$ of dimension $d \leq (n-3)/2$, 
and let $H$ denote the subgroup of all elements in $G=\Cl(V)$ that act trivially on $U$. 
For any $v \in V \smallsetminus U$, the $H$-orbit $v^H$ has length $|v^H| =q^n-q^d > q^{n}/2$ if $\Cl=\SL$ and 
$|v^H| \geq q^{n-d-2}$ otherwise.
\end{lem}

\begin{proof}
By assumption, $\mathrm{codim}\,U \geq (n+3)/2 \geq 2$. Hence, in the case $\Cl=\SL$, $H$ acts transitively on $V \smallsetminus U$,
which has cardinality $q^n-|U| \geq q^n/2$. 

\smallskip
We will now consider the case $\Cl \neq \SL$ and so $V$ is endowed with a non-degenerate $G$-invariant 
(bilinear or Hermitian) form $(\cdot|\cdot)$ (and quadratic form $\sQ$ when $p=2$ and $\Cl=\Omega$). In particular, we consider the 
orthogonal complement $U^{\perp}$ of dimension $n-d$ and fix a basis $(u_1, \ldots ,u_{d})$ of $U$. \edt{We claim that $|v^H|$ is the number $N$ of 
vectors $v'=v+u \in V$ such that
\begin{equation}\label{orbit20}
  u \in U^\perp,~ v+u \notin U, \mbox{ and the subspaces }\langle U,v \rangle_{\F_q}, \langle U,v' \rangle_{\F_q} \mbox{ are isometric.}
\end{equation}
Indeed, if $v' =h(v)$ for some $h \in H$, then $u:=v'-v \in U^\perp$ by \eqref{gram11}, $v' \notin U$, and $h$ induces an isometry between 
$\langle U,v \rangle_{\F_q}$ and $\langle U,v' \rangle_{\F_q}$. Conversely, suppose $v'=v+u$ satisfies \eqref{orbit20}, and let 
$\tilde G$ denote the full isometry group of $V$. By Witt's lemma, there exists $g \in \tG$ that maps $u_i \mapsto u_i$ and $v \mapsto v'$.
We also note that the proof of Lemma \ref{count-v} shows that we can put $U$ in a non-degenerate (with respect to $(\cdot|\cdot)$) 
subspace $W$ of $V$ dimension $\leq 2d$. (Indeed, in the notation introduced prior to \eqref{gram10} we can take 
$W = \langle U,f_1, \ldots,f_a \rangle_{\F_q}$ of dimension $d+a \leq 2d$.) The same claim applied to $\langle U,v \rangle_{\F_q}$ allows 
us to put this subspace in a non-degenerate subspace $W$ of dimension $\leq 2(d+1) \leq n-1$. In particular, $\dim W^\perp \geq 1$. 
In fact, in the cases where $\Cl=\Sp$, or $2|q$ and $\Cl=\Omega$, we have $2|n$ and so $2(d+1) \leq n-2$, whence $\dim W^\perp \geq 2$. 
This condition on $\dim W^\perp$ ensures that $\tG = G\tG_W$, where $\tG_W$ consists of the elements of $\tG$ that act trivially on $W$
(and so is isomorphic to the full isometry group of $W^\perp$). Hence we can write $g=hy$ with $h  \in G$ and $y \in \tG_W$, and observe that
$h \in G$ still maps $u_i \mapsto u_i$, $v \mapsto v'$. Thus $h \in H$, and $v'=h(v) \in v^H$.

Next we consider the case 
$$v \in U+U^\perp,$$ 
so that $v=v_0 +v_1$ with  
$v_0 \in U$ and $v_1 \in U^\perp \smallsetminus U = K \smallsetminus K^\perp$, where $K := U^\perp$ and $K^\perp = U \cap U^\perp$ is 
the radical of the restriction of $(\cdot|\cdot)$ to $K$. Then 
$$|v^H| = |\Omega(v_1)|$$ 
with $\Omega(v_1)$ defined in Lemma \ref{orbit1}. Indeed,
for any $u \in U^\perp$, we have $v+u \notin U$ if and only if $v_1+u \in K \smallsetminus K^\perp$, furthermore, 
$$(v'|v')-(v|v) = (v_0+v_1+u|v_0+v_1+u)-(v_0+v_1|v_0+v_1) = (v_1+u|v_1+u)-(v_1|v_1),$$
and, in the presence of $\sQ$,
$$\sQ(v')-\sQ(v) = \sQ(v_0+v_1+u)-\sQ(v_0+v_1) = \sQ(v_0)+\sQ(v_1+u)-\sQ(v_0)-\sQ(v_1) = \sQ(v_1+u)-\sQ(v_1).$$ 
Thus the map $u \mapsto v_1+u$ is a bijection between the set of vectors $u$ satisfying \eqref{orbit20} and $\Omega(v_1)$. 
Hence $|v^H| = |\Omega(v_1)| \geq q^{n-d-2}$ by Lemma \ref{orbit1}, and we are done in this case.

\smallskip
From now on, we may assume that 
$$v \notin U+U^\perp,$$
in which case  $v+u \notin U$ for any $u \in U^\perp$. So $|v^H|$ is just the number $N$ of vectors $u$ such that
\begin{equation}\label{orbit21}
  u \in U^\perp \mbox{ and the subspaces }\langle U,v \rangle_{\F_q}, \langle U,v' \rangle_{\F_q} \mbox{ are isometric.}
\end{equation}
}     

\smallskip
(a) Let $\Cl=\Sp$. \edt{Then \eqref{orbit21} is equivalent to $u \in U^\perp$.}
It follows that $|v^H| = |U^\perp|=q^{n-d}$.

\smallskip
(b) Assume $\Cl=\SU$. \edt{Then \eqref{orbit21} is equivalent to  $u \in U^\perp$ and} 
$(v|v)=(v+u|v+u)$, i.e. 
\begin{equation}\label{orb10}
  (v|u)+(u|v)+(u|u)=0.
\end{equation}  
Thus we need to count the number $N$ of solutions $u \in U^\perp$ for \eqref{orb10}.

By Witt's lemma, we can find a basis $(e_1, \ldots,e_m,g_1, \ldots,g_k)$ of $U^\perp$, with $k,m \geq 0$ and $k+m=n-d$, such 
that the Gram matrix of $(\cdot|\cdot)$ on $U^\perp$ in this basis is $\begin{pmatrix} 0 & 0\\ 0 & I_k \end{pmatrix}$. 
Let $a_i := (e_i|v)$ and $b_j := (g_j|v)$. Since $v \notin U=(U^\perp)^\perp$, 
\begin{equation}\label{orb11}
  (a_1, \ldots,a_m,b_1, \ldots ,b_k) \neq (0,0, \ldots,0).
\end{equation}  
Writing $u = \sum_ix_ie_i+\sum_j y_jg_j$ with $x_i,y_j \in \F_q$, \eqref{orb10} amounts to 
\begin{equation}\label{orb12}
  0=\sum_j y_j^{q_0+1} + \sum_i \bigl(x_ia_i+(x_ia_i)^{q_0}\bigr) +  \sum_j \bigl(y_jb_j+(y_jb_j)^{q_0}\bigr).
\end{equation}  
Note that, for any $c \in \F_{q_0}$, the equation $x^{q_0+1}=c$ has at least one solution in $\F_q$. Hence, if $k \geq 1$, for any
choice of $(x_1, \ldots,x_m,y_2, \ldots,y_k)$ we have at least one choice of $y_1$ to fulfill \eqref{orb12}. It follows that
$N \geq q^{m+k-1} = q^{n-d-1}$. Assume $k=0$. Then we may assume by \eqref{orb11} that $a_1 \neq 0$. Then, for any $c \in \F_{q_0}$, the equation $x_1a_1+(x_1a_1)^{q_0}=c$ has $q_0$ solutions in $\F_q$. Hence, for any
choice of $(x_2, \ldots,x_m)$ we have at least one choice of $x_1$ to fulfill \eqref{orb12}, and thus
$N \geq q^{m-1} = q^{n-d-1}$.

\smallskip
(c) Consider the case $\Cl=\SO$ and $p > 2$. \edt{Then \eqref{orbit21} is equivalent to $u \in U^\perp$}
and $(v|v)=(v+u|v+u)$, i.e.  
\begin{equation}\label{orb20}
  \sQ(u)+(u|v)=0,
\end{equation}
where $\sQ(u):=(u|u)/2$.  
Thus we need to count the number $N$ of solutions $u \in U^\perp$ for \eqref{orb20}.

By Witt's lemma, we can find a basis $(e_1, \ldots,e_m,g_1, \ldots,g_k)$ of $U^\perp$, with $k,m \geq 0$ and $k+m=n-d$, such 
that the Gram matrix of $(\cdot|\cdot)$ on $U^\perp$ in this basis is $\begin{pmatrix} 0 & 0\\ 0 & E \end{pmatrix}$;
moreover, if $k \geq 3$, or if $k=2$ and $\Span(g_1,g_2)$ is of type $+$, then we can choose to have 
$$E = \diag\biggl(\begin{pmatrix}0 & 1\\1 & 0\end{pmatrix},\diag(\eps_{3}, \ldots,\eps_k) \biggr)$$
for some $\eps_j \in \F_q^\times$.
Let $a_i := (e_i|v)$ and $b_j := (g_j|v)$. Since $v \notin U=(U^\perp)^\perp$, \eqref{orb11} holds; also 
write $u = \sum_ix_ie_i+\sum_j y_jg_j$ with $x_i,y_j \in \F_q$ and $w := \sum_je_j$. 

Now, if $m \geq 1$ and, say $a_1 \neq 0$, then \eqref{orb20} amounts to
$$a_1x_1+\sum_{i \geq 2}a_ix_i + \sQ(w) + (w|v) = 0.$$
For every choice of $(x_2, \ldots,x_m,y_1, \ldots,y_k)$ we have  
a unique choice of $x_1$ to fulfill this equation, and so
$N = q^{m+k-1} = q^{n-d-1}$.  Assume now that $a_1=\ldots =a_m = 0$; hence $k \geq 1$ by \eqref{orb11}. 
If $k \geq 3$, or $k=2$ but $\Span(g_1,g_2)$ is 
of type $+$, then our choice of $E$ transforms \eqref{orb20} into 
$$0=y_1y_2+\sum_{j \geq 3}y_j\eps_j^{2} + \sum_j y_jb_j.$$
Note that, for any $c \in \F_q$, the equation $y_1y_2+b_1y_1+b_2y_2=c$, which is equivalent to
$$(y_1+b_2)(y_2+b_1)=c+b_1b_2,$$ 
has at least $q-1$ solutions in $\F_q$ (one for each choice of $y_2 \neq -b_1$), whence
$N \geq (q-1)q^{m+k-2} = (q-1)q^{n-d-2}$. If $k \leq 2$, then we can choose $y_1= \ldots = y_k=0$ and $x_i$ arbitrarily, yielding 
$N \geq q^{m} \geq q^{n-d-2}$. (A more careful analysis shows that $N = (q-1)q^m$ if $k=1$, and $N=(q+1)q^m$ if $k=2$, using the 
transitivity of $\GO^-_2(q)$ on vectors $y \in \Span(g_1,g_2)$ of given $\sQ(y) \in \F_q^\times$.) 

\smallskip
(d) Finally, let $\Cl=\Omega$ and $p = 2$. \edt{Then \eqref{orbit21} is equivalent to $u \in U^\perp$} 
and $\sQ(v)=\sQ(v+u)$, i.e. $u \in U^\perp$ satisfies \eqref{orb20}. 
Thus we need to count the number $N$ of solutions $u \in U^\perp$ for \eqref{orb20}.

By Witt's lemma, we can find a basis $(e_1, \ldots,e_m,g_1, \ldots,g_{2k})$ of $U^\perp$, with $k,m \geq 0$, and $2k+m=n-d$, such 
that the Gram matrix of $(\cdot|\cdot)$ on $U^\perp$ in this basis is $\begin{pmatrix} 0 & 0 & 0\\ 0 & 0 & I_{k}\\0 & I_{k} & 0 \end{pmatrix}$. Let 
$$a_i := (e_i|v),~b_j := (g_j|v),~c_i := \sQ(e_i),~d_j := \sQ(g_j).$$ 
Since $v \notin U=(U^\perp)^\perp$, \eqref{orb11} holds; also 
write $u = \sum_ix_ie_i+\sum_j y_jg_j$, $x_i,y_j \in \F_q$, and $w:=\sum_j y_jg_j$.
Then \eqref{orb20} amounts to
$$\sum_{i}\bigl(a_ix_i + c_ix_i^2\bigr) + \sQ(w) + (w|v) = 0.$$

(d1) Suppose $k \geq 2$, or $k=1$ but $\Span(g_1,g_2)$ is of type $+$. Then we can choose $(g_1, \ldots,g_{2k})$ so that 
$d_1=\sQ(g_1)=0$. Then, for any $c \in \F_q$, the equation $y_1y_{k+1}+b_1y_1+b_{k+1}y_{k+1}+d_{k+1}y_{k+1}^2=c$ has at least 
$q-1$ solutions in $\F_q$ (one for each choice of $y_{k+1} \neq b_1$). In this case, 
for every choice of $(x_1, \ldots,x_m,y_j \mid j \neq 1,k+1)$ we have at least $q-1$ choices of $(y_1,y_{k+1})$ to fulfill \eqref{orb20}, and so $N \geq (q-1)q^{m+2k-2} \geq q^{n-d-2}$.  

(d2) If, for instance, $a_1 =0$ but $c_1 \neq 0$, or $a_1 \neq 0$ but $c_1=0$, then for any $c \in \F_q$ the equation
$c_1x_1^2+a_1x_1 = c$ has a unique solution in $\F_q$. In this case, 
for every choice of $(x_2, \ldots,x_m,y_1, \ldots,y_{2k})$ we have at a unique choice of $x_1$ to fulfill \eqref{orb20}, and so
$N = q^{m+2k-1} = q^{n-d-1}$.  

(d3) Suppose $k=1$ but $\Span(g_1,g_2)$ is of type $-$. If, say, $c_1=\sQ(e_1) \neq 0$, then, replacing 
$g_1$ by $g_1+e_1/\sqrt{c_1}$ we get $\sQ(g_1)= 0$, and so $\Span(g_1,g_2)$ is now of type $+$ and we can finish as in (d1).
Otherwise we have $c_i=0$ for all $i$. Now, if, say $a_1 \neq 0$, then we can argue as in (d2). If $a_i=0$ for all $i$, then by choosing 
$y_1=y_2=0$ but $x_1, \ldots,x_m$ arbitrarily, we get $N \geq q^m = q^{n-d-2}$.

(d4) It remains to consider the case $k=0$. Then $\sqrt{\sQ}$ is $\F_q$-linear on $U^\perp = \Span(e_1, \ldots,e_m)$, 
so we can choose $(e_1, \ldots,e_m)$ such that $c_1 = \ldots = c_{m-1}=0$. Now, if $a_i \neq 0$ for some $1 \leq i \leq m-1$,
then we can argue as in (d2). Otherwise we have $a_1= \ldots = a_{m-1}=0$. Choosing 
$x_m=0$ but $x_1, \ldots,x_{m-1}$ arbitrarily, we get $N \geq q^{m-1} = q^{n-d-1}$.
\end{proof}

 \begin{prop}
 \label{trans}
Let $n \geq 5$, $m<n$, and $k\leq (n-1)/4$ be positive integers and $r$ a non-negative integer.
Let $V=\F_q^n$, $G = \GL(V)$ or $\Cl(V)$, and let $g$ be an element of $G$ with $\supp(g) \geq n-m$.
Let 
\edt{$$\bv=(v_1,\ldots,v_k) \mbox{ and }\bw=(w_1,\ldots,w_k)$$}
denote two sequences of linearly independent vectors of $V$, and suppose
$$r = \dim \Span(v_1,\ldots,v_k,w_1,\ldots,w_k)-k.$$
Let 
$$G_\edt{{\bv,\bw}} := \left\{x\in G \mid x^{-1}gx(w_i) = v_i,\;1 \leq i \leq k \right\}.$$
Then the number of elements in $G_\edt{{\bv,\bw}}$ is at most
$$2^rq^{n^2-k(n-m)-rm}$$
if $G = \GL(V)$ or $\SL(V)$, and at most
$$q^{D-k(n-m)+r(k-m+1)+k(k+1)/2}$$
if $G = \Cl(V)$ with $\Cl=\SU, \Sp$, $\SO$, or $\Omega$.
\end{prop}

\begin{proof}
(i) Suppose first that $r=0$, so
$v_1,\ldots,v_k$ and $w_1,\ldots,w_k$ span the same subspace $W$ of $V$.  
Let $T$ denote the linear transformation on $W$ defined by $w_i\mapsto v_i$.  
We consider $\overline V:=V\otimes_{\F_q}\Fqb$ and $\overline W:=W\otimes_{\F_q}\Fqb$ as $\F_q[t]$-modules,
where $t$ acts by $g\otimes 1$ and $T\otimes 1$ on these spaces respectively.
The condition $x^{-1}gx w_i = v_i$ implies that $x\otimes 1$ induces a $\Fqb[t]$-linear map from $\overline W$ to $\overline V$.

For $\lambda\in \Fqb$, let $\overline V_\lambda$ and $\overline W_\lambda$  be the corresponding generalized eigenspaces, 
i.e. $\Ker((t-\lambda)^n)$ on $\overline V$ and $\overline W$. Then
$$\Hom_{\Fqb[t]}(\overline W, \overline V) = \prod_\lambda \Hom_{\Fqb}(\overline W_\lambda,\overline V_\lambda).$$
For each $\lambda$, we choose decompositions
$$\overline V_\lambda = \bigoplus_i \overline V_{\lambda,i},\ \overline W_\lambda = \bigoplus_i \overline W_{\lambda,i},$$
where
$$\overline V_{\lambda,i}\cong  (\Fqb[t]/(t-\lambda)^i\Fqb[t])^{a_{\lambda,i}},\ \overline W_{\lambda,i}\cong  (\Fqb[t]/(t-\lambda)^i\Fqb[t])^{b_{\lambda,i}}.$$
Thus,
$$\dim_{\Fqb}\Hom_{\Fqb[t]}(\overline W, \overline V)  = \sum_\lambda\sum_{i,j} \min(i,j)a_{\lambda,i}b_{\lambda,j}.$$
As $\dim\,\Ker(g-\lambda)=\sum_i a_{\lambda,i} \le m$ for all $\lambda$, for each $j\ge 1$, we have
$$\sum_i \min(i,j) a_{\lambda,i} \le j \sum_i a_{\lambda,i} = j m.$$
Furthermore, $\sum_{j,\lambda}jb_{\lambda,j}=k$, hence
$$\dim_{\Fqb}\Hom_{\Fqb[t]}(\overline W, \overline V)  \le \sum_\lambda \sum_j jmb_{\lambda,j} = mk.$$
Thus, for any $x \in \End_{\F_q} V$ such that $(x\otimes 1)|_{\overline W} \in \Hom_{\Fqb[t]}(\overline W, \overline V)$, there are at most $q^{mk}$ possibilities for the restriction of $x$ to
$W$. If $G = \GL(V)$ or $\SL(V)$, then there are at most $q^{(n-k)n}$ possibilities for the
restriction of $x$ to
a complement to $W$ in $V$. Therefore,  
$$|G_\edt{{\bv,\bw}}| \le  |\{x\in \End_{\F_q} V: (x\otimes 1)|_{\overline W} \in \Hom_{\Fqb[t]}(\overline W, \overline V)\}| \le q^{n^2 - k(n-m)},$$
which implies the proposition in the case $r=0$ and $G = \GL(V), \SL(V)$. 

Suppose $G = \Cl(V) \neq \SL(V)$. Note that $\dim W = k \leq (n-1)/4 \leq (n-3)/2$. The number of elements
$x \in G_\edt{{\bv,\bw}}$ with a fixed action on $W$ is at most the order of the pointwise stabilizer $H$ of $W$ in $G$, which is 
bounded by $q^{D-kn+k(k+1)/2}$ by Lemma \ref{order}. Hence, 
$$|G_\edt{{\bv,\bw}}| \leq q^{D-k(n-m)+k(k+1)/2},$$
completing the proof of the statement in the case $r=0$. 

\smallskip
(ii) For the general statement, we use induction on $k$, with the obvious induction base $k=0$.  Suppose the proposition holds for 
$k-1 \geq 0$.  Using the case $r=0$ established in (i), we may assume without loss of generality that 
$v_k\not\in \Span(w_1,\ldots,w_k)$. Hence we can find $v^* \in V^*$, a dual vector of $V$, such that $v^*(w_i)=0$ for $1\le i\le k$ and $v^*(v_k) = 1$.  For $1\le i < k$, we denote $c_i := v^*(v_i)$, and observe that 
replacing $v_i$ by $v_i - c_i v_k$ and $w_i$ by $w_i - c_i w_k$ for $1\le i \leq k-1$ does not affect the set $G_\edt{{\bv,\bw}}$.
After this replacement, we now have $v^*(v_i)=0$ for $1 \leq i \leq k-1$, $v^*(v_k)=1$, and $v^*(w_i)=0$ for $1 \leq i \leq k$. Hence
\begin{equation}\label{vk1}
  v_k \notin U:=\Span(v_1,\ldots,v_{k-1},w_1,\ldots,w_k).
\end{equation}  
Let $\Omega:=V\smallsetminus U$, 
and let 
$$\begin{aligned}& \edt{\bv':=(v_1,\ldots,v_{k-1}),~\bw':=(w_1,\ldots,w_{k-1}),}\\
    & H := \{x \in G \mid x(v_i) = v_i,~1 \leq i \leq k-1,~x(w_j) = w_j,~1 \leq j \leq k\}.\end{aligned}$$ 
We also note  that $\dim U \leq 2k-1 \leq (n-3)/2$. Next, for all $h\in G$ we have that 
$$|G_\edt{{h(\bv),h(\bw)}}| = |G_\edt{{\bv,\bw}}|.$$
Taking $h \in H$, we have $h(v_k) \in \Omega$ by \eqref{vk1}. Moreover, if $h,h' \in H$ and $h(v_k) \neq h'(v_k)$, then
$G_\edt{{h(\bv),h(\bw)}}$ and $G_\edt{{h'(\bv),h'(\bw)}}$ are disjoint subsets of $G_\edt{{\bv',\bw'}}$ of the same size. It follows
that 
\begin{equation}\label{vk2}
  |G_\edt{{\bv,\bw}}| \le \frac{|G_\edt{{\bv',\bw'}}|}{|v^H|}.
\end{equation}
Also set
$$r' := \dim \Span(v_1,\ldots,v_{k-1},w_1,\ldots,w_{k-1})-(k-1).$$
Then \eqref{vk1} implies that $r' \leq r \leq r'+1$. 

\smallskip
(a) Here we consider the case $G = \GL(V)$ or $\SL(V)$. Then $H$ acts transitively on $\Omega$
which has cardinality at least $q^n/2$. Therefore, \eqref{vk2} implies that 
\begin{equation}\label{vk3}
  |G_\edt{{\bv,\bw}}| \le \frac{|G_\edt{{\bv',\bw'}}|}{|\Omega|} \leq \frac{2|G_\edt{{\bv',\bw'}}|}{q^n}.
\end{equation}
The proposition now follows by induction. Indeed, by induction hypothesis and \eqref{vk3},
$$\begin{aligned}|G_\edt{{\bv,\bw}}| & \leq 2^{r'+1}q^{n^2-(k-1)(n-m)-r'm-n}\\
    & = 2^{r'+1}q^{n^2-k(n-m)-(r'+1)m}\\
    & \leq 2^rq^{n^2-k(n-m)-rm}\end{aligned}$$
since $q^m \geq 2$.

\smallskip
(b) Now consider the case $G = \Cl(V) \neq \SL(V)$. Then the induction hypothesis for $k-1$ implies
$$|G_\edt{{\bv',\bw'}}| \leq q^{D-(k-1)(n-m)+r'(k-m)+k(k-1)/2}.$$
On the other hand, \eqref{vk1} implies that $\dim U = k+r-1$, and so Lemma \ref{orbit2} yields
$$|v^H| \geq q^{n-\dim U-2} = q^{n-k-r-1}.$$
It follows from \eqref{vk2} that $|G_\edt{{\bv,\bw}}| \leq q^E$, where 
$$\begin{aligned}E & = D-(k-1)(n-m)+r'(k-m)+k(k-1)/2 - (n-k-r-1)\\ 
   & = D-k(n-m) + k(k+1)/2 + r'(k-m) + r-m+1.\end{aligned}$$
If $r=r'$, then 
$$\begin{aligned}E & = D-k(n-m)+r(k-m+1) + k(k+1)/2+1-m\\ 
  & \leq D-k(n-m)+r(k-m+1)+k(k+1)/2 \end{aligned}$$ 
since $m \geq 1$.
If $r=r'+1$, then 
$$\begin{aligned}E & = D-k(n-m)+r(k-m+1) + k(k+1)/2+1-k\\
   &  \leq D-k(n-m)+r(k-m+1)+k(k+1)/2\end{aligned}$$ 
since $k \geq 1$, and the induction step is completed.
\end{proof}

We denote $x^{-1} g x$ by $g^x$.
\begin{prop}\label{tuples}
Let $n \geq 5$ be an integer, $V := \F_q^n$, and $G := \GL(V)$ or $\Cl(V)$.
Let $m<n$ be a positive integer and let $g\in G$ be an element with $\supp(g) \geq n-m$.
Let $a$, $b$ and $d$ be integers such that $k := ad \leq (n-1)/4$ and $b\ge\frac n{n-m}$.
Let $P(x) = \sum_h p_h x^h \in \F_q[x]$ be a monic polynomial of degree $\leq d-1$.
Let $u_1,\ldots,u_a\in V$ be linearly independent.  Then the number of $b$-tuples $(x_1,\ldots,x_b)\in G^b$ such that 
$P(g^{x_b}\cdots g^{x_1}) u_i=0$ for $1\le i\le a$ is bounded above by
$$q^{bn^2+(b-1)k^2+2bk-an+1}$$
if $G = \GL(V)$ or $\SL(V)$, and by
$$q^{bD+(\frac{5}{2}b-1)k^2+\frac{7}{2}bk-an}$$
if $G = \Cl(V) \neq \SL(V)$.
\end{prop}

\begin{proof}
(i) For $0\le j \le b-1$, $1\le i\le a$, and $0\le h \le d-1$, consider all choices $(\bw,\bv)$ of vectors $v^h_{i,j},w^h_{i,j}\in V$ satisfying the following conditions:
\begin{enumerate}[\rm(a)]
\item For $1\le i\le a$, $w^0_{i,0} = u_i$.
\item For $0\le j\le b-2$, $w^{h}_{i,j+1} = v^h_{i,j}$.
\item For $0\le h\le d-2$, $w^{h+1}_{i,0} = v^{h}_{i,b-1}$.
\item For $1\le i\le a$, $\sum_h p_h v^h_{i,b-1} = 0$.
\item For each $j$, the $k=ad$ vectors of the form $w^h_{i,j}$ are linearly independent.
\end{enumerate}
Given such choices, we define 
$$r_{\bw,\bv,j} := \dim \Span \bigcup_{h,i} \{v^h_{i,j},w^h_{i,j}\} - k.$$

For each $b$-tuple $(r_0,r_1,\ldots,r_{b-1})$, we would like to bound above the number of
pairs $(\bw,\bv)$ with $r_{\bw,\bv,j} = r_j$ for $j=0,1,\ldots,b-1$.  To do this, we choose $0\leq t \leq b-1$ such that 
$r_t = \max(r_0,\ldots,r_{b-1})$.
We first choose all the $w^h_{i,0}$.  By condition (a), the values for $h=0$ are determined, so there are less than $q^{(k-a)n}$ possibilities.
By conditions (c) and (d), these choices determine $v^{h}_{i,b-1}$ for all $h$ and $i$.

Next, iteratively, for $0 \le j < t$, we choose $w^h_{i,j+1} = v^h_{i,j}$ for $0\le h\le d-1$ and $1\le i\le a$, subject to the condition $r_{\bw,\bv,j} = r_j$.
By Lemma~\ref{count-v}, there are at most $\binom k{r_j}q^{r_j n + k^2}$ choices at each step.
For the remaining values of $j$, we work backward for $t<j<b$, choosing $w^h_{i,j}$ for $0\le h\le d-1$ and $1\le i\le a$, subject to the condition that $r_{\bw,\bv,j} = r_j$.
Again, there are at most $\binom k{r_j}q^{r_j n + k^2}$ choices at each step.  Therefore, the total number of choices is at most
$$q^{(k-a)n}\prod_{j\neq t} \binom k{r_j}q^{r_j n+k^2} \le q^{(k-a)n+\frac{(b-1)rn}b+(b-1)k^2}\prod_{j\neq t} \binom k{r_j},$$
where $r:=r_0+\cdots+r_{b-1}$. Since 
$$\sum_{r_0 + \ldots + r_{b-1}=r}\prod_j\binom k{r_j} \leq \prod^{b-1}_{j=0} \biggl(\sum^k_{r_j=0} \binom k{r_j}\biggr) = 2^{bk},$$ 
the number $N_1$ of pairs $(\bw,\bv)$ with $\sum_j r_{\bw,\bv,j} = r$ is less than
$$2^{bk} q^{(k-a)n+\frac{(b-1)rn}b+(b-1)k^2}.$$
For given $(\bw,\bv)$ we consider the number $N_2$ of $(x_0,\ldots,x_{b-1})\in G^b$ such that 
\begin{equation}
\label{constraint}
g^{x_j} w^h_{i,j} = v^h_{i,j}\ \forall h,i,j.
\end{equation}

\smallskip
(ii) Consider the case $G=\GL(V)$ or $\SL(V)$. By Proposition~\ref{trans}, 
$$N_2 \leq \prod_{j=0}^{b-1}2^{r_{\bw,\bv,j}}q^{n^2-k(n-m) -r_{\bw,\bv,j} m}= 2^rq^{bn^2-bk(n-m) -rm}.$$
Therefore, the number $N$ of tuples $(\bw,\bv,x_1,\ldots,x_b)$ satisfying \eqref{constraint} is at most 
$$2^{2bk+1}\max_{0\le r\le bk}q^{bn^2+(b-1)k^2-(bk-r)(\frac{b-1}bn-m)-an}.$$
As $b\le \frac{n}{n-m}$, this is bounded above by $q^{bn^2+(b-1)k^2+2bk-an+1}$.
The projection onto $G^b$ of the set of tuples $(\bw,\bv,x_1,\ldots,x_b)$ satisfying \eqref{constraint} therefore has order at most $q^{bn^2+(b-1)k^2+2bk-an+1}$,
which proves the proposition in this case.

\smallskip
(iii) Now let $G=\Cl(V) \neq \SL(V)$. By Proposition~\ref{trans}, 
$$N_2 \leq \prod_{j=0}^{b-1}q^{D-k(n-m)+k(k+1)/2 +r_{\bw,\bv,j}(k-m+1)}= q^{bD-bk(n-m)+bk(k+1)/2 +r(k-m+1)}.$$
Therefore, the number $N$ of tuples $(\bw,\bv,x_1,\ldots,x_b)$ satisfying \eqref{constraint} is at most 
$$2^{2bk}\max_{0\le r\le bk}q^{bD-bk(n-m)+bk(k+1)/2 +r(k-m+1)+(k-a)n+(b-1)k^2+(b-1)rn/b}.$$
As $b\le \frac{n}{n-m}$, this is bounded above by $q^{bD+(\frac{5}{2}b-1)k^2+\frac{7}{2}bk-an}$. Again projecting onto $G^b$, we obtain
the proposition in this case.
\end{proof}

\section{Probabilistic lemmas}
\edt{The probability theory terminology used in this section and beyond can be found in a standard text such as \cite{Durrett}.}
Given a specified finite group $G$, 
we denote by $\sX_1,\sX_2,\ldots$ a sequence of \edt{independent} uniformly distributed random variables on $G$.  Thus, for $g\in G$, $g^{\sX_i}$ are \edt{independent uniformly distributed} random elements of the conjugacy class of $g$ in $G$.  We use the counting results of the previous section to prove, roughly, that 
\edt{for finite classical groups} the maximum eigenspace dimension of $g^{\sX_1}\cdots g^{\sX_b}$ almost always 
grows sublinearly in $n$, provided
that $\supp(g)$ is sufficiently large and $b\,\supp(g)\ge n$.  We will be particularly interested in the case that $\supp(g)$ is bounded below by a constant multiple of $n$ as $n\to\infty$;
in this regime, it is important that the probability that a large eigenspace exists goes to zero exponentially in $n^2$.
\begin{prop}
\label{A-prods}
Let $0 < \eps < 1$, $d \in \Z_{\ge 2}$, $n \in \Z_{\geq 1}$, 
$G = \Cl(V)$, or $G=\Omega(V)$ when $2 \nmid q$. Suppose $s \in \Z_{\geq 1}$ is such that  
$n > s \geq 8d^2/\eps$ if $G = \SL(V)$ and  
$n > s \geq 23d^2/\eps$ if $G \neq \SL(V)$.
Then, with 
$$b:= \lceil n/s \rceil,$$ 
the following statement holds for any element 
$g \in G$ with $\supp(g) \geq s$. 
The probability that there exists a non-zero polynomial $P(x) \in \F_q[x]$ of degree $< d$ such that 
$$\dim\Ker P(g^{\sX_1}\cdots g^{\sX_b}) \ge \eps n$$ 
is less than 
$$q^{3+d-\frac{\eps^2ns}{18d^2}}$$
if $G = \SL(V)$, and less than 
$$q^{2+d-\frac{\eps^2ns}{31d^2}}$$
if $G \neq \SL(V)$.
\end{prop}

\begin{proof}
Since the number of  non-zero polynomials in $\F_q[x]$ of degree $< d$ is $q^d-1$, it suffices to prove that for each $P$,
$$\Prob[\dim\Ker P(g^{\sX_1}\cdots g^{\sX_b}) \ge \eps n] \le \left\{\begin{array}{ll}
    q^{3-\frac{\eps^2 ns}{18d^2}}, & G=\SL(V),\\
    q^{2-\frac{\eps^2 ns}{31d^2}}, & G \neq \SL(V). \end{array} \right.$$
We fix $P$ and, with $a$ chosen below, let $(\sU_1,\ldots,\sU_a)$ denote a random ordered $a$-tuple of linearly independent vectors in $V$, uniformly distributed among all such $a$-tuples and independent of the $\sX_i$.

\smallskip
(i) First consider the case $G = \SL(V)$.
Since $n > s \geq 8d^2/\eps$, we have $2 \leq b < n\eps/8d^2+1 \leq n\eps/7d^2$. In particular,
\begin{equation}\label{for-b1}
  2bd+2(b-1)d^2 < 3bd^2 < 3n\eps/7 < n\eps/2, \mbox{ and }2bd \leq bd^2 \leq n\eps/7.
\end{equation} 
We also choose 
\begin{equation}\label{for-a1}
  a:=\lfloor \alpha\rfloor, \mbox{ where }\alpha := \frac{\eps n-2bd}{2(b-1)d^2}.
\end{equation}
Note from \eqref{for-b1} that $\alpha > 2$. Furthermore, $ad \leq \alpha d < \eps n/4 < n/4$, and so $ad \leq (n-1)/4$. 
Hence by Proposition~\ref{tuples} we have
$$\Prob[P(g^{\sX_1}\cdots g^{\sX_b}) \sU_i = 0,\;
\forall\,i\leq a] \le \frac{q^{bn^2 + (b-1)a^2 d^2 + 2abd - an + 1}}{|G|^b}.$$
The number of $b$-tuples in $G$ is greater than $\frac{q^{bn^2}}{(4q)^b}\geq q^{bn^2 - 3b}$, so 
$$\Prob[P(g^{\sX_1}\cdots g^{\sX_b}) \sU_i = 0,\;
\forall\,i\leq a] \le q^{3b+(b-1)a^2 d^2 + 2abd - an + 1}.$$
If $W$ is a subspace of $V$ of dimension $\ge \eps n$, then
\begin{equation}\label{for-w1}
  \Prob[\sU_1,\ldots,\sU_a\in W] \ge \frac{(q^{\eps n}-1)\cdots (q^{\eps n}-q^{a-1})}{(q^n-1)\cdots(q^n-q^{a-1})}\ge \frac{q^{a\eps n}}{4q^{an}}   \ge q^{\eps an-an-2}.
\end{equation}
Thus,
\begin{align*}
q^{3b+(b-1)a^2 d^2 + 2abd - an + 1}  \\
							&\hskip -40pt\ge \Prob[\sU_1,\ldots,\sU_a\in \Ker  P(g^{\sX_1}\cdots g^{\sX_b})] \\
                                                           &\hskip -40pt\ge \sum_{\{W\mid \dim W\ge \eps n\}} \Prob[\Ker  P(g^{\sX_1}\cdots g^{\sX_b}) = W]\;\Prob[\sU_1,\ldots,\sU_a\in W] \\
                                                           &\hskip -40pt\ge q^{\eps an-an-2} \sum_{\dim W\ge \eps n} \Prob[\Ker  P(g^{\sX_1}\cdots g^{\sX_b}) = W] \\
                                                           &\hskip -40pt= q^{\eps an-an-2} \Prob[\dim \Ker P(g^{\sX_1}\cdots g^{\sX_b}) \ge \eps n].\\
\end{align*}
We deduce that
\begin{equation}
\label{q-power}
\Prob[\dim \Ker P(g^{\sX_1}\cdots g^{\sX_b}) \ge \eps n] \le q^{3+3b+(b-1)a^2 d^2 + 2abd -\eps an}.
\end{equation}
By \eqref{for-b1} and \eqref{for-a1}, the exponent in \eqref{q-power} is bounded above by
\begin{align*}
3+3b+(b-1)d^2\alpha^2 &+ (2bd-\eps n)(\alpha-1) = 3+3b - \frac{(\eps n-2bd)^2}{4(b-1)d^2} + (\eps n-2bd)\\
&< 3 - \frac{(\eps n-2bd)^2}{4(b-1)d^2}+\eps n <  3 -\frac{(\eps n - 2bd)^2}{4d^2n/s} +\eps n\\
&< 3-\frac{(6\eps n/7)^2}{4d^2 n/s} + \eps n
=3-\frac{9\eps^2 ns}{49d^2}+\eps n < 3-\frac{\eps^2 ns}{18d^2},
\end{align*}
since $s \geq 8d^2/\eps$.

\smallskip
(ii) Now let $G \neq \SL(V)$. Since $n > s \geq 23d^2/\eps$, we have $2 \leq b < n\eps/23d^2+1 \leq n\eps/22d^2$. In particular,
\begin{equation}\label{for-b2}
  \frac{7}{2}bd+2(10b-4)d^2 < 22bd^2 < n\eps, \mbox{ and }\frac{7}{2}bd \leq \frac{7}{4}bd^2 \leq \frac{n\eps}{22\cdot 4/7}< \frac{n\eps}{12}.
\end{equation} 
We also choose 
\begin{equation}\label{for-a2}
  a:=\lfloor \alpha\rfloor, \mbox{ where }\alpha := \frac{\eps n-\frac{7}{2}bd}{(10b-4)d^2}.
\end{equation}
Note from \eqref{for-b2} that $\alpha > 2$. Furthermore, $ad \leq \alpha d < \eps n/12 < n/4$, and so $ad \leq (n-1)/4$. 
Hence by Proposition~\ref{tuples} we have
$$\Prob[P(g^{\sX_1}\cdots g^{\sX_b}) \sU_i = 0,\;1 \leq i \leq a] \le \frac{q^{bD + (\frac{5}{2}b-1)a^2 d^2 + \frac{7}{2}abd - an}}{|G|^b}.$$
The number of $b$-tuples in $G$ is greater than $\frac{q^{bD}}{4^b}\geq q^{bD - 2b}$, where we use the bound 
$$|\Omega(V)| = |\SO(V)|/2 > q^D/4$$
when $2 \nmid q$. Therefore, 
$$\Prob[P(g^{\sX_1}\cdots g^{\sX_b}) \sU_i = 0,\;1 \leq i \leq a] \le q^{2b+(\frac{5}{2}b-1)a^2 d^2 + \frac{7}{2}abd - an}.$$
Again using \eqref{for-w1} as above, we deduce that
$$q^{2b+(\frac{5}{2}b-1)a^2 d^2 + \frac{7}{2}abd - an}  \geq q^{\eps an-an-2} \Prob[\dim \Ker P(g^{\sX_1}\cdots g^{\sX_b}) \ge \eps n],$$
and so
\begin{equation}
\label{q-power2}
\Prob[\dim \Ker P(g^{\sX_1}\cdots g^{\sX_b}) \ge \eps n] \le q^{2+2b+(\frac{5}{2}b-1)a^2 d^2 + \frac{7}{2}abd -\eps an}.
\end{equation}
Note that $16 \leq 10b-4 < 10n/s+6 < 11n/s$. 
Hence, by \eqref{for-b2} and \eqref{for-a2}, the exponent in \eqref{q-power2} is bounded above by
\begin{align*}
2+2b+(\frac{5}{2}b-1)d^2\alpha^2 &+ (\frac{7}{2}bd-\eps n)(\alpha-1) = 2+2b - \frac{(\eps n-\frac{7}{2}bd)^2}{(10b-4)d^2} + (\eps n-\frac{7}{2}bd)\\
&< 2 - \frac{(\eps n-\frac{7}{2}bd)^2}{(10b-4)d^2}+\eps n <  2 -\frac{(\eps n - \frac{7}{2}bd)^2}{11d^2n/s} +\eps n\\
&< 2-\frac{(11\eps n/12)^2}{11d^2 n/s} + \eps n
=2-\frac{11\eps^2 ns}{144d^2}+\eps n < 2-\frac{\eps^2 ns}{31d^2},
\end{align*}
since $s \geq 23d^2/\eps$.
\end{proof}

Recall that $\EB(\sX)$ denotes the expected value of the random variable $\sX$.

\begin{lem}
\label{Frob}
For any finite group $G$, any element $g \in G$, any irreducible character $\chi$ of $G$, and any positive integer $b$,
$$\EB[\chi(g^{\sX_1}\cdots g^{\sX_b})] = \frac{\chi(g)^b}{\chi(1)^{b-1}}.$$
\end{lem}

\begin{proof}
For any $h\in G$, the probability $\Prob[g^{\sX_1}\cdots g^{\sX_b}=h]$ is given by the Frobenius formula
$$\frac 1{|G|}\sum_{\varphi \in \Irr(G)}\frac{\varphi(g)^b \overline\varphi(h)}{\varphi(1)^{b-1}}.$$
%
Therefore, 
\begin{align*}
\EB[\chi(g^{\sX_1}\cdots g^{\sX_b})] &= \sum_{h\in G} \chi(h)\frac 1{|G|}\sum_{\varphi \in \Irr(G)}\frac{\varphi(g)^b \overline\varphi(h)}{\varphi(1)^{b-1}} \\
&=  \sum_{\varphi \in \Irr(G)} \frac{\varphi(g)^b}{\varphi(1)^{b-1} }\frac 1{|G|}\sum_{h\in G} \chi(h)\overline\varphi(h) = \frac{\chi(g)^b}{\chi(1)^{b-1}}
\end{align*}
by the orthonormality of irreducible characters.
\end{proof}

\section{Character bounds for elements with large support}

The main result in this section is Theorem~\ref{main-bound1}, which gives an exponential character bound at $g\in G$ whenever $\supp(g)$ is bounded below by
a fixed positive multiple of $n$.  We use the results of the previous section to show that, assuming $n$ is large, a random walk on the Cayley graph of $G$ with respect to $g^G$
almost always leads in a bounded number of steps to an element whose centralizer order is smaller than any desired power of $|G|$.  Using known character bounds for such elements,
we can estimate the expectation of $\chi$ on such elements, and deduce an exponential upper bound for $|\chi(g)|$.

\begin{lem}
\label{matrix-cent}
For any $0 < \nu < 1$, there exists $0 < \alpha < 1$ such that, for any $n \in \Z_{\geq 2}$ and any prime power $q$, if 
$V = \F_q^n$, $g\in \GL(V)$, and  
\begin{equation}
\label{linear-dim}
\dim \CB_{\End(V)}(g)\ge \alpha n^2,
\end{equation}
then $|\CB_{\SL(V)}(g)| > |\SL(V)|^{1-\nu}$.
\end{lem}

\begin{proof}
We can take $\alpha = 1-\nu^2/4$. 
If $d$ denotes $\dim \CB_{\End(V)}(g)$, then $d$ is the dimension of the centralizer $\uC(g)$ of $g$ in the algebraic group $\GL_n$. 
The finite group $\uC(g)(\F_q) = \CB_{\GL(V)}(g)$ has a normal series, whose factors $X_i$ are unipotent groups of 
order $q^{d_i}$, or $\GL_{m_i}(q^{a_i})$ with $d_i:=m_i^2a_i$, and $\sum_id_i = d$. Note that \edt{since $q^{a_i j}-1 \ge (q-1)^{a_i}q^{a_i j-a_i}$ for $1\le j\le m_i$,}
$$q^{d_i} \geq |\GL_{m_i}(q^{a_i})| = q^{a_im_i(m_i-1)/2}\cdot\prod^{m_i}_{j=1}(q^{a_ij}-1) \geq (q-1)^{m_ia_i}q^{d_i-m_ia_i},$$
and $\GL_{m_i}(q^{a_i})$ has $\F_q$-rank $m_ia_i$. Since the rank of $\uC(g)$ is at most $n$, it follows that 
\begin{equation}\label{for-c10}
  q^d \geq |\CB_{\GL(V)}(g)| \geq (q-1)^{n} q^{d-n},
\end{equation}  
and so 
\begin{equation}\label{for-c11}
  |\CB_{\SL(V)}(g)| \geq (q-1)^{n-1}q^{d-n}.
\end{equation}    
Now, if $n \ge 2/\nu$, then, since $\alpha =1 -\nu^2/4 > 1-\nu/2$, we have
$$|\CB_{\SL(V)}(g)| \ge q^{d-n} \ge  q^{(\alpha-1/n)n^2} >  q^{(1-\nu/2-\nu/2)n^2} > |G|^{1-\nu}.$$
If $n < 2/\nu$, then $\alpha =1-\nu^2/4> 1-1/n^2$, whence \eqref{linear-dim} implies that $g$ is a scalar matrix and 
therefore that $\CB_{\SL(V)}(g) = \SL(V)$.
\end{proof}

\begin{prop}\label{gl-s}
Let $n \geq 2$, $V = \F_q^n$, and let $g \in \GL(V)$ have support $s:=\supp(g)$. Then 
\begin{enumerate}[\rm(a)]
\item \edt{$(n-s)^2 \leq \dim \CB_{\GL(V \otimes_{\F_q}\Fqb)}(g) \leq n(n-s)$, }
\item $|\CB_{\GL(V)}(g)| \leq q^{n(n-s)}$, and 
\item $q^{ns-2} \leq |g^{\SL(V)}| \leq q^{2ns+n-s^2-1}$. If particular, 
$|\SL(V)|^{s/3n} \leq |g^{\SL(V)}| \leq |\SL(V)|^{3s/n}$; in fact, $|\SL(V)|^{s/2n} \leq |g^{\SL(V)}| \leq |\SL(V)|^{2.5s/n}$ if 
$(n,q) \neq (2,2)$, $(2,3)$.
\end{enumerate}
\end{prop}

\begin{proof}
(a) In the case $g$ is unipotent, the estimates were already proved in \cite[pp. 509--510]{LiSh1}.
In the general case, we can replace $V$ by $V \otimes_{\F_q}\Fqb$, and 
let $\lambda_1, \ldots,\lambda_m \in \Fqb^\times$ be all the distinct eigenvalues of 
the semisimple part $t$ of $g=tu$ on $V$, with multiplicities $n_1, \ldots, n_m$. If $V_i=\Fqb^{n_i}$ denotes
the corresponding $t$-eigenspace on $V$ and if the unipotent part $u$ of $g$ acts on $V_i$ as $u_i$, then  
$$\CB_{\GL(V)}(g) = \prod^m_{i=1}\CB_{\GL(V_i)}(u_i).$$
For $s_i:=\supp(u_i)$, the largest $g$-eigenspace on $V_i$ has dimension $n_i-s_i$, and we may assume that 
$$n_i-s_i \leq n-s=n_1-s_1.$$
By the unipotent case, $(n_i-s_i)^2 \leq \dim \CB_{\GL(V_i)}(u_i) \leq n_i(n_i-s_i)$. Hence
$$(n-s)^2 = (n_1-s_1)^2 \leq \sum_i (n_i-s_i)^2 \leq \dim \CB_{\GL(V)}(g) \leq \sum_in_i(n_i-s_i) \leq (n_1-s_1)\sum_i n_i = n(n-s).$$

\smallskip
(b) follows from (a) by \eqref{for-c10}.

\smallskip
(c) By \cite[Lemma 4.1(ii)]{LMT}, $q^{n^2-2} < |\SL_n(q)| < q^{n^2-1}$. On the other hand, setting
$d:=\dim \CB_{\GL(V)}(g)$, we have 
$q^{d-n} \leq |\CB_{\SL_n(q)}(g)| \leq q^{d}$ by \eqref{for-c10}--\eqref{for-c11}, and $(n-s)^2 \leq d \leq n(n-s)$ by (a).
It follows that $q^{ns-2} \leq |g^{\SL_n(q)}| \leq q^{2ns+n-s^2-1}$, yielding the first statement.

The second statement is obvious when $s=0$, and can be checked directly when $n=2$. When $n \geq 3$,
$2ns+n-s^2-1 \leq (n^2-2)(3s/n)$ and $ns-2 \geq (n^2-1)(s/3n)$.

The third statement is obvious when $s=0$, and can be checked directly when $n=2$ or $s=1$. When $n \geq 3$ and $s\ge 2$,
$2ns+n-s^2-1 \leq (n^2-2)(2.5s/n)$ and $ns-2 \geq (n^2-1)(s/2n)$.
\end{proof}

\begin{lem}
\label{alpha-to-d-epsilon}
For $1 \geq \eps > 0$ and $n \in \Z_{\geq 1}$, if $V = \F_q^n$, $g\in \GL(V)$, and 
$\dim  \Ker P(g) \leq \eps n$ for all polynomials $P(x)\in \F_q[x]$ of degree $<d:=\lceil 1/\eps \rceil$, then 
$$|\CB_{\GL(V)}(g)| \le q^{n^2\eps}.$$
\end{lem}

\begin{proof}
Let $s:=\supp(g)$ and suppose that $s < n-\eps n$. Then $n-s = \dim \Ker(g-\lambda)$ for some eigenvalue $\lambda$ of 
$g$ on $V \otimes_{\F_q}\Fqb$. Since any Galois conjugate of $\lambda$ over $\F_q$ is also an eigenvalue for $g$ with 
eigenspace of the same dimension $n-s$, the Galois orbit of $\lambda$ has length $e \leq n/(n-s) < 1/\eps \leq d$. Thus 
$\lambda$ is a root of some polynomial $P \in \F_q[x]$ of degree $e < d$. Hence, by hypothesis,
$$n-s = \dim\,\Ker(g-\lambda) \leq \eps n,$$
and so $s \geq n-\eps n$, a contradiction. We have shown that $s \geq n-\eps n$. By Proposition \ref{gl-s}(b), this implies 
that $|\CB_{\GL(V)}(g)| \leq q^{n^2 \eps}$. 
\end{proof}

Now we can prove character bounds for elements with large support.

\begin{thm}\label{main-bound1}
There exist explicit constants $\gamma > 0$ and $C \geq 4$ such that the following statement holds for any positive integer $n$, 
any $0 < \beta  < 1$, any $V=\F_q^n$ for any prime power $q$, any $G :=\SL(V)$, $\SU(V)$, $\Sp(V)$, or 
$\Omega(V)$ (or $\SO(V)$ or $\Spin(V)$ if $q$ is odd), any element $g\in G$, and any irreducible character $\chi \in \Irr(G)$.
If $s:=\supp(g) \geq \max(C,\beta n)$, then 
$$\frac{|\chi(g)|}{\chi(1)} \leq \chi(1)^{-\frac{\gamma s}{n \cdot \lceil 1/\beta \rceil}}.$$
\end{thm}

\begin{proof}
As $n\ge C$, by taking $C$ sufficiently large, we can guarantee that $n$ is as large as we wish;
also we may assume that $\chi(1) > 1$.
Let $b := \lceil n/s\rceil$. Also set 
$$\eps_0 = \frac{1}{12},~\delta_0 = \frac{8}{9}$$
if $G = \SL(V)$ or $\SU(V)$, and  
$$\eps_0 = 0.0011,~\delta_0 = 0.992$$
otherwise.
By Lemma~\ref{alpha-to-d-epsilon}, there exist $d \in \Z_{\geq 3}$ and $0 < \eps < 1$ such that for $h\in G$, if $\dim \Ker P(h) \le \eps n$ for all non-constant $P(x)\in \F_q[x]$ of degree $<d$, then
\begin{equation}\label{for-h0}
\begin{aligned}|\CB_{\GL(V)}(h)| \le q^{n^2/12}, & \mbox{ if }G=\SL(V) \cong \SL_n(q),\\
|\CB_{\GL(V)}(h)| \le q_0^{n^2/12}, & \mbox{ if }G = \SU(V) \cong \SU_n(q_0),\\
|\CB_{\GL(V)}(h)| \leq q^{(n/2-1)^2\eps_0}, & \mbox{ if }G = \Sp(V),~\SO(V),~\Omega(V),\\
|\CB_{\GL(V)}(\bar{h})| \leq q^{(n/2-1)^2\eps_0}/2, & \mbox{ if }2 \nmid q \mbox{ and }G= \Spin(V),\end{aligned}
\end{equation}
(with the convention that in the spin case, $\bar{h}$ is the image of $h$ in $\Omega(V)$ and $P(h)$ is replaced by $P(\bar{h})$;
this ensures $|\CB_G(h)| \leq q^{(n/2-1)^2\eps_0}$ in the spin case). Indeed, we can take 
$$\eps=1/4000,~d= \lceil \eps^{-1} \rceil = 4000,~C \geq 224,$$ 
and have
$$q^{(n/2-1)^2\eps_0}/2 \geq q^{(n/2 - 1)^2\eps_0-1} > q^{n^2\eps}.$$
The centralizer bound \eqref{for-h0} implies by \cite[Theorem~1.5]{GLT} and \cite[Theorem~1.4]{GLT2} that 
\begin{equation}\label{for-h1}
  |\chi(h)| \leq \chi(1)^{\delta_0}.
\end{equation}  
If $G = \SL(V)$, by Proposition~\ref{A-prods}, choosing $C \geq 8d^2/\eps$ and $\gamma > 0$ sufficiently small
so that 
\begin{equation}\label{for-g1}
  3+d-\frac{\eps^2ns}{18d^2} < -\gamma ns,
\end{equation}  
we then have
$$\Prob\bigl{[}|\CB_{\GL_n(q)}(g^{\sX_1}\cdots g^{\sX_b})| \geq q^{n^2/12}\bigr{]} < q^{-\gamma ns}.$$
If $G \neq \SL(V)$, by Proposition~\ref{A-prods}, choosing $C \geq 23d^2/\eps$ and $\gamma > 0$ sufficiently small
so that 
\begin{equation}\label{for-g2}
  2+d-\frac{\eps^2ns}{31d^2} < -\gamma ns,
\end{equation}
we then have
$$\Prob\bigl{[}|\CB_{\GL(V)}(g^{\sX_1}\cdots g^{\sX_b})| \geq q_0^{n^2/12}\bigr{]} < q^{-\gamma ns}$$
when $G = \SU(V)$, 
$$\Prob\bigl{[}|\CB_{\GL(V)}(g^{\sX_1}\cdots g^{\sX_b})| \geq q^{(n/2-1)^2\eps_0}\bigr{]} < q^{-\gamma ns}$$
when $G = \Sp(V)$, $\SO(V)$, or $\Omega(V)$, and 
$$\Prob\bigl{[}|\CB_{\GL(V)}(\bar{g}^{\sX_1}\cdots \bar{g}^{\sX_b})| \geq q^{(n/2-1)^2\eps_0}/2\bigr{]} < q^{-\gamma ns}$$
when $2 \nmid q$ and $G = \Spin(V)$. Indeed, by taking $C \geq 23d^2/\eps$ and $0 < \gamma \leq \eps^2/32d^2$, we have 
$$\frac{\eps^2ns}{31d^2}-\gamma ns \geq \frac{\eps^2 ns}{992d^2}  \geq \frac{C^2\eps^2}{992d^2} = \frac{529d^2}{992} \geq d+3$$
when $d \geq 3$, ensuring \eqref{for-g1} and \eqref{for-g2}.

By \eqref{for-h1} applied to $h=g^{\sX_1}\cdots g^{\sX_b}$, this implies that 
$$\Prob\bigl{[}|\chi(g^{\sX_1}\cdots g^{\sX_b})| \ge \chi(1)^{\delta_0}\bigr{]}< q^{-\gamma ns}.$$
Thus,
\begin{equation}\label{for-e2}
  \bigl{|}\EB[\chi(g^{\sX_1}\cdots g^{\sX_b})]\bigr{|} \le \EB\bigl{[}|\chi(g^{\sX_1}\cdots g^{\sX_b})|\bigr{]} \le \chi(1)^{\delta_0} + \chi(1)q^{-\gamma ns}.
\end{equation}  
On the other hand, 
\begin{equation}\label{for-e3}
  \EB[\chi(g^{\sX_1}\cdots g^{\sX_b})] = \chi(g)^b/\chi(1)^{b-1} 
\end{equation}  
by Lemma \ref{Frob}.

\smallskip Now assume that $s=\supp(g) \geq \beta n$. Then 
$$b = \lceil n/s \rceil \leq \lceil 1/\beta \rceil.$$ 
As $\chi(1) < |G|^{1/2} < q^{n^2/2}$, we have $q^{-\gamma ns} \leq \chi(1)^{-2\gamma s/n}$.  
Without loss of generality, we may assume $\gamma \le (1-\delta_0)/2=0.004$, so
$$ \chi(1)^{\delta_0} \le \chi(1)^{1-\frac{2\gamma s}n}.$$
For $n\ge 9$, the minimal degree for a non-trivial character of $G$ is at least $2^{n/2}$ \cite{LSe}, so we have 
$\chi(1)^{-\gamma s/n} \leq 2^{-\gamma s/2}$.  If $C$ is sufficiently large (say $C \geq 2/\gamma$), 
this is at most $1/2$, so when $\chi(1)>1$,
$$2\chi(1)^{1-\frac{2\gamma s}n} \le\chi(1)^{1-\frac{\gamma s}n}.$$
%
It now follows from \eqref{for-e2} and \eqref{for-e3} that 
$$|\chi(g)| \leq \chi(1)^{1-\frac{\gamma s}{nb}} \leq \chi(1)^{1-\frac{\gamma s}{n \lceil 1/\beta \rceil}},$$
as stated.
Moreover, our proof shows that one can take 
$$\gamma = \frac{\eps^2}{32d^2} = \frac{1}{2^{13} \cdot 10^{12}},~C = \frac{64d^2}{\eps^2} = 2^{14} \cdot 10^{12},$$
although this choice is not optimal.  
\end{proof}

\cite[Theorem 1.5]{GLT} and \cite[Theorem 1.3]{GLT2} produced the character bound $|\chi(1)|^\delta$ for any element $g$ in
a classical group $G$ with $|\CB_G(g)| \leq |G|^\eps$, but only {\it for certain positive constants $\eps < 1$}. Our next result 
generalizes this to {\it arbitrary constants}  $0 < \eps <1$:

\begin{thm}\label{main-bound2}
For any $0 < \eps < 1$, there exists a constant $0 < \delta < 1$ such that the following statement holds. For any $n \in \Z_{\geq 2}$, any 
prime power $q$, any quasisimple classical group 
$$G = \SL_n(q),~\SU_n(q),~\Sp_{2n}(q),~\Omega^\pm_n(q),~\Spin^\pm_n(q)$$
any element $g \in G$, and any irreducible character $\chi \in \Irr(G)$, if $|\CB_G(g)| \leq |G|^\eps$
we have
$$|\chi(g)| \le \chi(1)^{\delta}.$$
\end{thm}

\begin{proof}
(a) First suppose that $|\CB_H(h)| \leq |H|^\eps$ for some element $h$ of $H := \SU(V)$, $\Sp(V)$, $\SO(V)$, or $\Omega(V)$. 
\edt{Here $V$ is $\F_{q^2}^n$, $\F_q^{2n}$, $\F_q^n$, and $\F_q^n$ respectively
for $\SU_n(q)$, $\Sp_{2n}(q)$, $\SO_n(q)$, and $\Omega_n(q)$ respectively.}
Then $H \leq L := \SL(V)$ and $|H| > |L|^{1/3}$. Since
$$|L| \geq |H \CB_L(h)| = \frac{|H| \cdot |\CB_L(h)|}{|\CB_H(h)|},$$
we have
$$|\CB_L(h)|  \leq \frac{|L|}{|H|/|\CB_H(h)|} \leq \frac{|L|}{|H|^{1-\eps}} \leq |L|^{(2+\eps)/3}.$$  
 
Next, suppose that $2 \nmid q$ and $|\CB_H(h)| \leq |H|^\eps$ for some element $h$ of $H = \Spin(V)$. Then $H$ projects onto
$\bar{H}:=\Omega(V) \leq L=\SL(V)$, sending $h$ to $\bar{h}$, and $|\CB_{\bar H}(\bar h)| \leq |\CB_H(h)| \leq |H|^\eps$.
As above, we still have $|H| > |L|^{1/3}$. Hence the above argument yields 
$|\CB_L(\bar{h})|  \leq |L|^{(2+\eps)/3}$.  
 
\smallskip 
(b) Now set $\nu:=1-\eps$ if $G = \SL(V)$, and $\nu:=(1-\eps)/3$ otherwise. By the observations in (a), we have 
$|\CB_{\SL(V)}(g)| \leq |\SL(V)|^{1-\nu}$, with the convention that $g$ is replaced by its image in $\Omega(V)$ in the case
$G = \Spin(V)$. 
By Lemma \ref{matrix-cent}, there exists some $0 < \alpha < 1$ such that 
$$\dim_{\F_q} \CB_{\End(V)}(g) \leq \alpha n^2.$$ 
On the other hand, $g$ has an eigenspace of dimension $n-s$ on $\overline{V}=V \otimes_{\F_q}\Fqb$ for 
$s:=\supp(g)$, hence 
$$\dim_{\F_q}\CB_{\End(V)}(g) = \dim_{\Fqb}\CB_{\End(\overline{V})}(g) \geq (n-s)^2.$$ 
It follows that 
$s \geq n(1-\sqrt{\alpha})$. Now we can apply Theorem \ref{main-bound1}(i), with $\beta:= 1-\sqrt{\alpha}$, and take 
$\delta \geq \gamma\beta/\lceil 1/\beta \rceil$ when $s \geq C$. If $s < C$, then $n$ is bounded, and 
the result of Gluck \cite{Gl} implies the statement in this case.
\end{proof}

\section{Further bootstrapping and uniform character bounds}

In this section, we prove Theorem~\ref{main-bound3}, an exponential upper bound for $|\chi(g)|$
with exponent linear in $\frac{\supp(g)}{n}$, with an explicit, though very small, coefficient.  If $\supp(g)$ is greater than any given positive constant multiple of $n$, we already have this by the results of section 4,
so what is needed is a second bootstrapping argument to go from elements of small support to elements whose support satisfies a linear lower bound.
It may be useful for the reader to keep in mind the case that $g$ is a transvection.  Here we want  that the 
support of the product of $b$ random transvections, with probability very close to $1$,
grows linearly with $b$ for $b<n$; this is given by 
Proposition~\ref{big-support}.

\begin{lem}\label{supp-prod}
Let $V$ be a finite dimensional vector space over a field $\F$, $g,h \in \End(V)$, and let $\lambda,\mu \in \overline\F$.
Then the following statements hold.
\begin{enumerate}[\rm(i)]
\item $\codim\,\Ker(gh-\lambda\mu) \leq \codim\,\Ker(g-\lambda) + \codim\,\Ker(h-\mu)$.
\item $\supp(gh) \leq \supp(g)+\supp(h)$.
\end{enumerate}
\end{lem}

\begin{proof}
(i) Let $A:=\Ker(g-\lambda)$ and $B :=\Ker(h-\mu)$. As $A+B \subseteq V$, we have 
$$\dim (A \cap B) \geq \dim(A) + \dim(B) - \dim(V) = \dim(V) - \codim(A)-\codim(B).$$
Since $A \cap B \subseteq \Ker(gh-\lambda\mu)$, the statement follows.

\smallskip
(ii) Now choose $\lambda, \mu$ so that $\codim\,\Ker(g-\lambda) = \supp(g)$ and $\codim\,\Ker(h-\mu)=\supp(h)$. Since
$\supp(gh) \leq \codim\,\Ker(gh-\lambda\mu)$, the statement follows from (i).  
\end{proof}

In the next statement, we identify $V^* \otimes V$ with $\End(V)$ for any finite dimensional vector space over a field $\F$, and 
$\lambda \in \F$ with $\lambda \cdot \mathrm{Id}_V$.

\edt{
\begin{prop}
\label{spread}
Let $V$ be an $n$-dimensional vector space over a field $\F$ 
and $b$ and $k$ positive integers.  Let $x_1,\ldots,x_b$ be elements of $\GL(V)$ and
$v_1,\ldots,v_k\in V$ and $\phi_1,\ldots,\phi_k$ linearly independent vectors in $V$ and $V^*$ respectively.
Let $0 \neq \lambda \in \F$ be a scalar and let 
$$T := \lambda+ \sum_{j=1}^k \phi_j\otimes v_j$$ 
be regarded as an element of $\End(V)$.
For $1\le i\le b$ and $1\le j\le k$, let  $w_{i,j}:=x_i^{-1}(v_j)$ and $\psi_{i,j}:=x_i^{-1}(\phi_j)$.
For $1 \leq s \leq b$, let 
$$A_s := \dim\Span(w_{i,j}\mid 1\le i\le s, 1\le j\le k),\ B_s := \dim\Span(\psi_{i,j}\mid s\le i\le b, 1\le j\le k).$$
Then the rank of $T^{x_1}\cdots T^{x_b}-\lambda^b$ is at least 
$$\sum_{s=1}^b \max(0,A_s - A_{s-1} + B_s - B_{s+1} - k).$$
\end{prop}

\begin{proof}
It suffices to find vectors $u_{s,t}\in V$ and $\omega_{s,t}\in V^*$ indexed by
$$1\le s\le b,\ 1\le t\le \max(0,A_s - A_{s-1} + B_s - B_{s+1} - k)$$
such that
\begin{equation}
\label{u-pairing}
\omega_{s',t'}\bigl((T^{x_1}\cdots T^{x_b}-\lambda^b)(u_{s,t})\bigr) = \left\{ 
  \begin{array}{ll} \lambda^{b-1}, & \mbox{ if }(s',t')= (s,t),\\
  0, & \mbox{ if }s' \geq s \mbox{ and }(s',t') \neq (s,t).\end{array}\right.    
\end{equation}
Indeed, this guarantees that, with respect to the lexicographic ordering
on the two bases, the pairing $\langle \omega|v\rangle := \omega((T^{x_1}\cdots T^{x_b}-\lambda^b)(v))$ induced by \eqref{u-pairing} 
between the span of the $\omega_{s',t'}$ and the span of the 
$\lambda^{1-b}u_{s,t}$ is unitriangular in terms of these bases, hence perfect.

To achieve this, we construct vectors $u_{s,t}\in V$ and $\omega_{s,t}\in V^*$ with the following properties:
\begin{enumerate}[\rm(a)]
\item For $s<i$, $\psi_{i,j}(u_{s,t}) = 0$.
\item For $s>i$, $\omega_{s,t}(w_{i,j})=0$.
\item For all $s$, $t$, and $t'$, $\omega_{s,t}((T^{x_s}-\lambda)(u_{s,t'})) = \delta_{t,t'}$.
\end{enumerate}
To accomplish this goal, for each $s$, let 
$$W_s := \Span(w_{s,1},\ldots,w_{s,k}),\ \Psi_s := \Span(\psi_{s,1},\ldots,\psi_{s,k}).$$

For $1\le s,s'\le b$, we define 
$$W_{[s,s']} := W_s+W_{s+1}+\cdots+W_{s'},\ \Psi_{[s,s']} := \Psi_s+\Psi_{s+1}+\cdots+\Psi_{s'},$$
with the convention that, if $s > s'$ we have $W_{[s,s']}=\Psi_{[s,s']}=\{0\}$.

As 
$$\dim \Psi_s/(\Psi_s \cap \Psi_{[s+1,b]})= \dim \Psi_{[s,b]}/\Psi_{[s+1,b]} = B_s-B_{s+1},$$
there exists a $(B_s-B_{s+1})$-dimensional subspace $U_s\subseteq V$ such that $\psi_{i,j}(U_s) = 0$ whenever $i>s$ but for $u\in U_s$, $\psi_{s,j}(u)=0$ for all $j$ implies $u=0$.
Thus the operator
$$T^{x_s}-\lambda = \sum_{j=1}^k x_s^{-1}(\phi_j)\otimes x_s^{-1}(v_j) =  \sum_{j=1}^k \psi_{s,j}\otimes w_{s,j}$$
annihilates every element of $V$ killed by $\Psi_s$, maps $V$ to $W_s$, and
maps $U_s$ injectively to $W_s$, and so
\begin{equation}
\label{inject}
\Ker (T^{x_s}-\lambda) \cap U_s = \{0\}.
\end{equation}
Let $W'_s\subseteq W_s$ denote a subspace (of dimension $A_s-A_{s-1}$)
complementary in $W_s$ to $W_{[1,s-1]}\cap W_s$, and let 
$$U'_s := \{u\in U_s\mid (T^{x_s}-\lambda)(u)\in W'_s\}.$$ 
Then the dimension $c_s$ of $U'_s$ satisfies
$$c_s \geq \dim U_s + \dim W'_s - \dim W_s = B_s-B_{s+1} + A_s-A_{s-1}-k.$$
Let $(u_{s,1},\ldots,u_{s,c_s})$ denote any basis of $U'_s$.  %
Condition (a) holds for all vectors in $U_s$ and therefore for the $u_{s,t}$.

Next, for each $s$ we choose $\omega_{s,1},\ldots,\omega_{s,c_s}$ satisfying condition (c) and 
annihilating $W_{[1,s-1]}$ (guaranteeing condition (b)).
We can do this because the conditions on $\omega_{s,t}$ are that $\omega_{s,t}((T^{x_s}-\lambda)(u_{s,t}))=1$ and $\omega_{s,t}$ annihilates
\begin{equation}
\label{space}
(T^{x_s}-\lambda)(\Span(u_{s,1},\ldots,u_{s,t-1},u_{s,t+1},\ldots,u_{s,c_s})) + W_{[1,s-1]}.
\end{equation}
Thus, it suffices to show that $(T^{x_s}-\lambda)(u_{s,t})$ does not belong to the vector space \eqref{space}.  
As the vectors $u_{s,1},\ldots,u_{s,c_s}$ form a basis of $U'_s$, by \eqref{inject}, the latter condition holds since
 $(T^{x_s}-\lambda)(U'_s) \subseteq W'_s$ meets $W_{[1,s-1]}$ in $\{0\}$
by definition of $W'_s$.

We claim that for $s+1\le s'\le b+1$, we have
\begin{equation}\label{for-t1}
  T^{x_{s'}}\cdots T^{x_b} (u_{s,t}) = \lambda^{b+1-s'} u_{s,t},
\end{equation}
by descending induction on $s'$.  The statement is trivially true for $s'=b+1$ (since $T^{x_{s'}}\cdots T^{x_b}$ means 
$\mathrm{Id}_V$). If \eqref{for-t1} holds for $s'+1$, then
$$T^{x_{s'}}\cdots T^{x_b} (u_{s,t}) = T^{x_{s'}}(\lambda^{b-s'} u_{s,t}) = \lambda^{b+1-s'} u_{s,t} + \lambda^{b-s'}\sum_{j=1}^k \psi_{s',j}(u_{s,t})w_{s,j} =  \lambda^{b+1-s'} u_{s,t},$$
where the last equality follows from condition (a).

Applying $T^{x_s}$ to both sides of \eqref{for-t1} with $s'=s+1$, we obtain
\begin{equation}
\label{coset}
T^{x_s}T^{x_{s+1}}\cdots T^{x_b} (u_{s,t}) =  \lambda^{b-s} T^{x_s}(u_{s,t}).
\end{equation}
Next, we claim that for $1\le i\le s$, we have
\begin{equation}
\label{i to s}
T^{x_i} T^{x_{i+1}}\cdots T^{x_s}(u_{s,t}) \in \lambda^{s-i}T^{x_s}(u_{s,t})+W_{[i,s-1]}.
\end{equation}
Indeed, this is trivial when $i=s$. If \eqref{i to s} holds for $i+1$, then there exists $w\in W_{[i+1,s-1]}$ so that
\begin{align*}
T^{x_i} T^{x_{i+1}}\cdots T^{x_s}(u_{s,t}) &= T^{x_i}( \lambda^{s-i-1}T^{x_s}(u_{s,t})+w) \\
&\in \lambda^{s-i}T^{x_s}(u_{s,t}) + \lambda w + W_i  \subset \lambda^{s-i}T^{x_s}(u_{s,t})+ W_{[i,s-1]}.
\end{align*}
By \eqref{coset} and the $i=1$ case of \eqref{i to s}, we obtain
\begin{align*}
T^{x_1}\cdots T^{x_b}(u_{s,t}) &= \lambda^{b-s}T^{x_1}\cdots T^{x_s}(u_{s,t}) \in \lambda^{b-1} T^{x_s}(u_{s,t}) + W_{[1,s-1]}\\
                                                &= \lambda^b T^{x_s}(u_{s,t}) + \lambda^{b-1}(T^{x_s}-\lambda)(u_{s,t}) + W_{[1,s-1]}.
\end{align*}
Subtracting $\lambda^b u_{s,t}$ and
applying $\omega_{s',t'}$ to both sides with $s'\ge s$, condition (b) implies
$$\omega_{s',t'}((T^{x_1}\cdots T^{x_b}-\lambda^b)(u_{s,t})) =  
\lambda^{b-1} \omega_{s',t'}\bigl((T^{x_s}-\lambda)(u_{s,t})\bigr).
$$
If $s'=s$, this is $\lambda^{b-1}\delta_{t,t'}$ by condition (c), and if $s'>s$, it is zero by condition (b), yielding \eqref{u-pairing} as
desired.
\end{proof}
}

\begin{prop}
\label{usually-almost-indep}
Let $V=\F_q^n$, $b \geq 2$ and $k$ positive integers with $bk\le n/2$, and let $v_1,\ldots,v_k$ be linearly independent vectors in $V$.
Let $\sX_1,\ldots,\sX_b$ be uniform independent random variables on $G=\Cl(V)$.  Then we have
$$\Prob\bigl{[}\dim\Span(\sX_i (v_{j})\mid 1\le i\le b,1\le j\le k) \le \frac{2bk}3\bigr{]} < \left\{ 
   \begin{array}{ll}q^{bk(1-n/6)}, & \Cl = \SL,\\
   q^{bk(1-n/12)}, & \Cl \neq \SL. \end{array} \right.$$
\end{prop}

\begin{proof}
Consider the $bk$-term sequence of random vectors 
$$\sX_1 (v_{1}),\ldots,\sX_1 (v_{k}), \sX_2 (v_{1}),\ldots,\sX_2(v_k), \ldots, \sX_b(v_1), \ldots,\sX_b (v_{k}).$$
First we bound from the above the probability $\Prob_{ij}$ that $\sX_i(v_j)$ lies in the span 
$$S_{ij} := \Span(\sX_{i'}(j'),1 \leq i' < i,1 \leq j' \leq k,~\sX_i(v_l),1 \leq l \leq j-1)$$ 
of the preceding vectors, conditioning on all the vectors preceding $\sX_i (v_{j})$ in the sequence. Since $v_1, \ldots,v_k$ are linearly independent, this conditional probability is
$0$ if $i=1$. Next, let $i \geq 2$, and let $H$ denote the subgroup of all the elements in $G$ that fix each of $v_1, \ldots,v_{j-1}$
(in particular, $H=G$ if $j=1$). By Lemma \ref{orbit2} applied to $d=j-1$ (so that $d\leq k-1 \leq n/4-1 \leq (n-3)/2$), $H$ acts on 
$$\Omega:=V \smallsetminus \Span(v_1, \ldots,v_{j-1})$$ 
with orbits $\Omega_1 =w_1^H, \ldots,\Omega_s=w_s^H$,
each of length at least 
$$L := \left\{\begin{array}{ll}q^{n}-q^{j-1}, & \Cl = \SL,\\
   q^{n-d-2}=q^{n-j-1}, & \Cl \neq \SL. \end{array} \right.$$ 
We want to count the number of $g \in G$, with $g(v_1), \ldots ,g(v_{j-1})$ all fixed, and with 
$g(v_j) \in S_{ij}$. Fix such a $g$, and consider any such $g'$.    
Then $h:=g^{-1}g' \in H$, and so it sends $v_j \in \Omega$ to some $w \in \Omega_t$ with $1 \leq t \leq s$. 
With $w$ fixed, the number of possibilities for $h$ is at most $|\Stab_H(w_t)|$. Hence, with $t$ fixed, the number of 
possibilities for such $g'$ is at most
$$|\Omega_t \cap S_{ij}| \cdot |\Stab_H(w_t)| = |H|\cdot|\Omega_t \cap S_{ij}|/|\Omega_t| \geq \frac{|H| \cdot |\Omega_t \cap S_{ij}|}{L}.$$
As we condition on all the vectors preceding $\sX_i(v_j)$, it follows that 
$$\Prob_{ij} \leq \frac{1}{|H|} \cdot \sum^s_{t=1}\frac{|H| \cdot |\Omega_t \cap S_{ij}|}{L} \leq \frac{|S_{ij}|}{L}.$$  
Note that $\dim S_{ij} \leq k(i-1)+j-1$. Hence,
conditioning on the sequence of previous vectors, the probability $\Prob_{ij}$ that $\sX_i (v_{j})$ lies in their 
span $S_{ij}$ is at most  
$$\frac{q^{k(i-1)+j-1}}{q^n-q^{j-1}} < 2q^{k(i-1)+j-1-n} \leq q^{bk-n}$$
when $\Cl=\SL$, and at most
$$\frac{q^{k(i-1)+j-1}}{q^{n-j-1}} < q^{k(i-1)+2j-n} \leq q^{(b+1)k-n}$$
when $\Cl \neq \SL$, regardless of what the previous vectors are. 

We also note that, since $b \geq 2$ and $bk \leq n/2$,
$(b+1)k \leq 3bk/2 \leq 3n/4$. 
Therefore, the probability that there exist $r$ terms $\sX_i (v_{j})$ in this sequence belonging to the span of previous terms is less than
$$\binom{bk}r q^{r(bk-n)} \leq \binom{bk}r q^{-rn/2}$$
when $\Cl=\SL$, and less than
$$\binom{bk}r q^{r((b+1)k-n)} \leq \binom{bk}r q^{-rn/4}$$
when $\Cl \neq \SL$. 
In particular, the probability that $r\ge bk/3$ is less than
$$\sum_{r = \lceil bk/3\rceil}^{bk} \binom{bk}r q^{-rn/2\kappa} < 2^{bk}q^{-bkn/6\kappa} < q^{bk(1-n/6\kappa)}$$
with $\kappa=1$ when $\Cl=\SL$ and $\kappa=2$ when $\Cl \neq \SL$. 
\end{proof}

\begin{prop}\label{big-support}
If $n$ is a sufficiently large positive integer, $V = \F_q^n$, $g\in G:=\Cl(V)$, then there exists a positive integer $b$ such that 
$$b\cdot\supp(g) \le n,$$
and if $\sX_1,\ldots,\sX_b$ are i.i.d. uniform random variables on $G$, then
$$\Prob[\supp(g^{\sX_1}\cdots g^{\sX_b}) < n/9] < \left\{ 
    \begin{array}{ll}q^{-n^2/20}, & \Cl=\SL\\
    q^{-n^2/40}, & \Cl \neq \SL. \end{array} \right.$$
\end{prop}

\begin{proof}
If $\supp(g) \ge n/6$, we can take $b=1$. Hence, without loss of generality, we may assume $k:=\supp(g)< n/6$, and hence we may choose $b \in \Z_{\geq 2}$ so that $n/3\le bk <n/2$.
Let $\lambda$ be an eigenvalue of $g$ such that 
$k=\mathrm{codim}\,\Ker(g-\lambda)$,
and choose linearly independent vectors $v_1,\ldots,v_k\in V$ and $v^*_1,\ldots,v^*_k\in V^*$ such that
$$g = \lambda + \sum_{j=1}^k v^*_j\otimes v_j.$$
By Lemma \ref{supp-prod}(i), $\codim\,\Ker(g^{\sX_1}\cdots g^{\sX_b}-\lambda^b) \leq bk < n/2$. 
Thus $\Ker(g^{\sX_1}\cdots g^{\sX_b}-\lambda^b)$ is the largest eigenspace for $g^{\sX_1}\cdots g^{\sX_b}$, and so 
$\supp(g^{\sX_1}\cdots g^{\sX_b}) = \codim\,\Ker(g^{\sX_1}\cdots g^{\sX_b}-\lambda^b)$.
 
Again define $\kappa:=1$ if $G=\SL(V)$ and $\kappa:=2$ if $G \neq \SL(V)$.  
By Proposition~\ref{usually-almost-indep}, the probability that $\dim\Span(\sX_i (v_j))$ or 
$\dim\Span(\sX_i (v^*_j))$ is $\le 2bk/3$
is at most 
$$2q^{bk(1-n/6\kappa)} \le 2q^{n(1-n/6\kappa)/3} \le q^{-n^2/20\kappa}$$
if $n$ is sufficiently large. On the other hand, by Proposition \ref{spread}, if
$$\dim\Span(\sX_i(v_j)),\ \dim\Span(\sX_i(v^*_j)) > \frac{2bk}3,$$
then
\begin{align*}
\mathrm{codim}\,\Ker(g^{\sX_1}\cdots g^{\sX_b}-\lambda^b) &\ge \sum_{s=1}^b \max(0,A_s-A_{s-1}+B_s-B_{s+1}-k) \\
                                                      &\ge  \sum_{s=1}^b (A_s-A_{s-1}+B_s-B_{s+1}-k) \\
                                                      & = A_b-A_0+B_1-B_{b+1}-bk \\
                                                      &> \frac{2bk}3 + \frac{2bk}3 - bk \\
                                                      &= \frac{bk}3 \geq \frac{n}{9}.
\end{align*}
Hence the proposition follows.   
\end{proof}

Now we can prove one of the main results of the paper, giving a uniform exponential character bound in terms of the support.

\begin{thm}\label{main-bound3}
There exists an explicit constant $\sigma > 0$ such that the following statement holds for any positive integer $n \geq 3$, 
any $V = \F_q^n$ for any prime power $q$, any $G :=\SL(V)$, $\SU(V)$, $\Sp(V)$, or 
$\Omega(V)$ (or $\SO(V)$ or $\Spin(V)$ if $q$ is odd), any $g \in G$, and any irreducible character $\chi \in \Irr(G)$:
$$\frac{|\chi(g)|}{\chi(1)} \leq \chi(1)^{-\sigma \cdot \supp(g)/n}.$$
\end{thm}

\begin{proof}
First we apply Theorem \ref{main-bound1} to obtain the positive constants $C$ and $\gamma$.
By \cite[Theorem 1.2.1]{LST}, 
$$\frac{|\chi(g)|}{\chi(1)} \leq q^{-\sqrt{\supp(g)}/481}.$$
Therefore, by choosing $\sigma > 0$ small enough, say $\sigma \leq 1/(241 \cdot (9C)^{3/2})$ we may ignore the cases where $n < 9C$ (or if $\chi(1)=1$): if $s:= \supp(g) \leq n < 9C$, then $\chi(1) \leq |G|^{1/2} \leq q^{n^2/2}$, hence
$$\chi(1)^{\sigma s/n} \leq q^{\sigma ns/2} < q^{\sqrt{s}/481}.$$
Henceforth we may assume
that $n \geq \max(9C,9)$ and $\chi(1) > 1$; in particular, $\chi(1) > 2^{n/3}$ by \cite{LSe}.

Now, if $s:=\supp(g) \geq n/9$, then $s \geq C$, and so we are done by Theorem \ref{main-bound1}, taking $\beta:=1/9$ and 
$\sigma \leq \gamma/9$; in particular,
\begin{equation}\label{for-e4}
  |\chi(g)| \leq \chi(1)^{1-\gamma/81}.
\end{equation}  

Consider the case $1 \leq s < n/9$ and apply Proposition \ref{big-support} to get $2 \leq b \leq n/s$ and that 
$$\Prob[\supp(g^{\sX_1}\cdots g^{\sX_b}) < n/9] < q^{-n^2/20\kappa},$$
with $\kappa=1$ if $G = \SL(V)$ and $\kappa=2$ if $G \neq \SL(V)$.  By the previous bound \eqref{for-e4} for elements with support 
$\geq n/9$, this implies that 
$$\Prob\bigl{[}|\chi(g^{\sX_1}\cdots g^{\sX_b})| \ge \chi(1)^{1-\gamma/81}\bigr{]}< q^{-n^2/20\kappa} < \chi(1)^{-1/10\kappa}$$
since $\chi(1) \leq |G|^{1/2} < q^{n^2/2}$.
Thus,
$$\bigl{|}\EB[\chi(g^{\sX_1}\cdots g^{\sX_b})]\bigr{|} \le \EB\bigl{[}|\chi(g^{\sX_1}\cdots g^{\sX_b})|\bigr{]} \le \chi(1)^{1-\gamma/81} + \chi(1)^{1-1/10\kappa}.$$
Since $\chi(1) \geq q^{n/3} \geq 2^{3C}$ by \cite{LSe}, by choosing $\sigma > 0$ small enough, we then have  
\begin{equation}\label{for-e5}
  \bigl{|}\EB[\chi(g^{\sX_1}\cdots g^{\sX_b})]\bigr{|} \le \chi(1)^{1-\gamma/81} + \chi(1)^{1-1/10\kappa} \le \chi(1)^{1-\sigma}.
\end{equation}
It now follows from \eqref{for-e3} and \eqref{for-e5} that 
$$|\chi(g)| \leq \chi(1)^{1-\frac{\sigma}{b}} \leq \chi(1)^{1-\frac{\sigma s}{n}},$$
as stated. In fact, our proof shows that we can take
$$\sigma= \min\biggl( \frac{1}{241 \cdot (9C)^{3/2}},\frac{\gamma}{82} \biggr),$$
which is $1/(6507 \cdot 2^{21} \cdot 10^{18}) > 7 \cdot 10^{-29}$ for our chosen $C$ and $\gamma$. 
\end{proof}

As a consequence of Theorem \ref{main-bound3}, we can prove the following linear refinement of \cite[Theorem 1.2.1]{LST}:

\begin{cor}\label{linear-supp}
There exists an absolute constant $\gamma > 0$ such that the following statement holds. For any $n \in \Z_{\geq 2}$, any prime
power $q$, any quasisimple classical group 
$$G = \SL_n(q),~\SU_n(q),~\Sp_{2n}(q),~\Omega^\pm_n(q),~\Spin^\pm_n(q),$$
any $g \in G$, and any $\chi \in \Irr(G)$ of degree $\chi(1) > 1$, we have
$$\frac{|\chi(g)|}{\chi(1)} \leq q^{-\gamma \cdot \supp(g)}.$$
\end{cor}

\begin{proof}
By \cite[Theorem 1.2.1]{LST},
$$|\chi(g)|/\chi(1) \leq q^{-\sqrt{\supp(g)}/481}.$$
Hence, by choosing $\gamma \leq 1/1443$ we may ignore the cases where $\supp(g) \leq 9$,
in particular if $n \leq 9$. Assume now that $n \geq 10$, which implies $\chi(1) \geq q^{n/3}$ by \cite{LSe}. Hence Theorem \ref{main-bound3} yields $|\chi(g)/\chi(1)| \leq q^{-\sigma s/3}$, and we are done by taking 
$$\gamma = \min(1/1443,\sigma/3),$$
which is $1/(19521 \cdot 2^{21} \cdot 10^{18}) > 2 \cdot 10^{-29}$ for our chosen $\sigma$.
\end{proof}

We conclude this section with the following examples, which show that the exponent 
$\sigma\cdot\supp(g)/n$ in Theorem \ref{main-bound3},
is optimal (up to the constant $\sigma$).

\begin{exa}\label{sln2}
{\em Consider $G = \SL_n(2)$, $n \geq 3$, and the unique irreducible character $\tau$ of degree $2^n-2$ of $G$, so that
$2\cdot 1_G+\tau$ is the permutation character of $G$ acting on the point set of $V = \F_2^n$. 
Suppose $2 \leq s \leq n-1$. Choose $\xi \in \overline\F_2^\times$ of order $2^s-1$ and 
$g \in \SL_n(2)$ that is conjugate to $\mathrm{diag}\bigl(1, \ldots,1,\xi,\xi^2, \ldots,\xi^{2^{s-1}}\bigr)$ over $\overline\F_2$. 
Then $\supp(g)=s$, $\tau(1)=2^n-2$, and $\tau(g) = 2^{n-s}-2$, whence $|\tau(g)/\tau(1)| \approx \tau(1)^{-s/n}$.
If $s=1$, we can choose $g$ to be a transvection, for which we have $\supp(g)=1$, $\tau(g)=2^{n-1}-2$, and so 
again $|\tau(g)/\tau(1)| \approx \tau(1)^{-s/n}$.
}
\end{exa}

More generally, we have 

\begin{lem}\label{lower}
Let $G = \Cl_n(q)$ be a simple classical group with $n \geq 7$. If $|G|$ is large enough, and if 
$g \in G$ has support $s=\supp(g) \leq n-2$, then there is a non-trivial
$\chi \in \Irr(G)$ such that $|\chi(g)/\chi(1)| \geq \chi(1)^{-6s/n}$.
\end{lem}

\begin{proof}
The statement is obvious for $s=0$, so we assume $1 \leq s \leq n-2$. 
Choosing $G$ of large enough order, we have $\sum_{1_G \neq \chi \in \Irr(G)}\chi(1)^{-0.55} < 1$ by \cite[Corollary 1.3]{LiSh3}. 
Assume to the contrary that $|\chi(g)| \leq \chi(1)^{1-6s/n}$ for all $\chi \in \Irr(G)$. 
Choosing $k:= \lfloor (n-2)/s \rfloor$ we have 
$ks \leq n-2 \leq (k+1)s-1$, and so 
$$6ks - 2.55n \geq 3.45ks -2.55s - 2.55 \geq 6.9s-2.55s -2.55 >0$$
if $k \geq 2$. If $k = 1$, then $s \geq (n-1)/2$, and so $6ks \geq 3n-3 \geq 2.55n$ since $n \geq 7$. 
Hence, for any $x \in G$ we have  
$$\sum_{1_G \neq \chi \in \Irr(G)}\frac{|\chi(g)^k\bar\chi(x)|}{\chi(1)^{k-1}} \leq \sum_{1_G \neq \chi \in \Irr(G)}\frac{|\chi(g)^k|}{\chi(1)^{k-2}} \leq \sum_{1_G \neq \chi \in \Irr(G)}\frac{1}{\chi(1)^{6ks/n-2}} 
    \leq  \sum_{1_G \neq \chi \in \Irr(G)}\chi(1)^{-0.55} < 1,
$$
and thus every element $x \in G$ is a product of $k$ conjugates of $g$ and so has $\supp(x) \leq ks \leq n-2$ by Lemma 
\ref{supp-prod}(ii). But this is 
a contradiction since $G$ always contains elements of support $n-1$.
\end{proof}

\section{Support vs. class size, and proof of Theorem \ref{main}}
In this section, we deduce Theorem~\ref{main} from Theorem \ref{main-bound3}.  The main difficulty is to bound conjugacy class sizes $|g^G|$ in terms of the support $\supp(g)$ for all classical groups $G \leq \GL(V)$. 
To do this, we need an analogue of Proposition~\ref{gl-s}(c) for all classical groups.
\edt{There are results of Liebeck-Shalev \cite{LiSh1} and Liebeck-Schul-Shalev \cite{LSS} which are very much in this spirit.
However, we develop them from scratch because we want somewhat greater generality (not just the simple groups) and also because we do not want implicit constants.}
For any finite group $X$, let $P(X)$ denote the smallest index of any proper subgroup of $X$. Lower bounds for $P(X)$, $X$ a finite 
classical groups, are listed in \cite[Table 5.2.A]{KlL}. First we deal with unitary groups.

\begin{prop}\label{gu-s}
Let $n \geq 3$, $(n,q_0) \neq (3,2)$, $V = \F_{q_0^2}^n$, and let $g \in \GU(V)$ have support $s:=\supp(g)$. Then 
$$|\SU(V)|^{s/2n} \leq |g^{\SU(V)}| \leq |\SU(V)|^{3s/n}.$$ 
\end{prop}

\begin{proof}
Let $q:=q_0^2$, so that $V = \F_q^n$. Let $d$ denote the dimension of the centralizer $\uC(g)$ of $g$ in the algebraic group $\GL_n$. 
Then $d$ is bounded above and below in Proposition \ref{gl-s}(a).
The finite group $\CB_{\GU(V)}(g)$ has a normal series, whose factors $X_i$ are unipotent groups of 
order $q_0^{d_i}$, or $\GL_{m_i}(q_0^{a_i})$ with $2|a_i$, or $\GU_{m_i}(q_0^{a_i})$, with $d_i:=m_i^2a_i$, and $\sum_id_i = d$.
By \cite[Lemma 4.1(iv)]{LMT},
\begin{equation}\label{ord10}
  q_0^{m^2a} \leq |\GU_{m}(q_0^{a})| = q_0^{m^2a}\cdot\prod^{m}_{j=1}\bigl(1-\frac{(-1)^j}{q_0^{aj}}\bigr) \leq \frac{3}{2}q_0^{m^2a};
\end{equation} 
also, $3/2 < 2^{0.6} \leq q_0^{0.6}$, and $q_0^{n^2-1.5} < |\SU(V)| < q_0^{n^2-1}$. Now we can follow the proof of Proposition \ref{gl-s}(b), (c), but increasing the upper bound
for $\CB_{\SU(V)}(g)$ by $q_0^{0.6n}$ and decreasing the lower bound by $(q_0+1)/(q_0-1) \leq 3 < q_0^{1.6}$, going 
down from $\CB_{\GU(V)}(g)$ to $\CB_{\SU(V)}(g)$. It follows that 
$$q_0^{ns-0.6n-2} \leq |g^{\SU(V)}| \leq q_0^{2ns+n-s^2+0.6}.$$
Now, the statement can be checked directly for $s=0$ and for $\SU_3(5)$. If $(n,q) \neq (3,5)$ and $s \leq 2$, then 
$$|g^{\SU(V)}| \geq P(\SU(V)) > q_0^n > |\SU(V)|^{s/2n}.$$
If $s=1$, we also have $|g^{\SU(V)}| < q_0^{2n-1} < |\SU(V)|^{3/n}$. In all other cases, 
$$ns-0.6n-2 \geq \frac{s}{2n}(n^2-1),~~2ns+n-s^2+0.6 \leq \frac{3s}{n}(n^2-1.5),$$
proving the statement.
\end{proof}

Let $J_i$ denote the Jordan block of size $i \in \Z_{\geq 1}$ and with eigenvalue $1$. Then the Jordan canonical form of any unipotent element $u$ in $\GL(V)$ can be written as $\oplus_{i \geq 1}J_i^{n_i}$, meaning it contains $n_i \in \Z_{\geq 0}$ blocks $J_i$ for each
$i \geq 1$. Sometimes we will re-order the blocks into the form $\oplus^t_{k=1}J_{m_k}$ with $m_1 \geq m_2 \geq \ldots \geq m_t \geq 1$. 

\begin{lem}\label{2ways}
In the above notation, for any unipotent element $u \in \GL(V)$ we have 
\begin{equation}
\label{n vs m}
\sum_i in_i^2 +2\sum_{i < j}in_in_j = \sum^t_{k=1}(2km_k-m_k).
\end{equation}
\end{lem}

\begin{proof}
We induct on the number $r \geq 1$ of distinct sizes of Jordan blocks of $u$. Suppose $r=1$, i.e. $u \sim J_m^s$, $t=s$, and 
$m_1 =  \ldots = m_s=m$. Then the left-hand-size of the formula is $ms^2$, and the right-hand-side is $m\sum^s_{k=1}(2k-1)=ms^2$.

We suppose the formula holds for $r \geq 1$, and prove it for $r+1$. We can present $u$ as $\diag(v,J_m^s)$, where 
$v =J_{m_1} \oplus J_{m_2} \oplus \ldots \oplus J_{m_t}$, $u = J_{m_1} \oplus J_{m_2} \oplus \ldots \oplus J_{m_{t+s}}$, 
and $m_{t+1} = \ldots = m_{t+s}=m$.
Then, replacing $v$ by $u$ increases the left-hand-side of \eqref{n vs m} by $ms^2 + 2\sum_{i > m}msn_i = ms^2+2mst$, 
whereas the right-hand-side grows by $\sum^{t+s}_{k=t+1}m(2k-1) = m((t+s)^2-t^2) = m(s^2+2st)$.
\end{proof}

\begin{lem}\label{unip-odd}
Let $q$ be an odd prime power, $V=\F_q^n$ be endowed with a non-degenerate, symplectic or orthogonal, bilinear form, and let
$G = \Sp(V)$, respectively, $\GO(V)$, denote the corresponding isometry group of the form. Extend the form to 
$\overline V := V \otimes_{\F_q}\Fqb$, and let $\uG = \Sp(\overline V)$, respectively $\GO(\overline V)$.
Let $g =\oplus_iJ_i^{n_i} \in G$ be a unipotent element with $s=\supp(g)$, and let $D(g) := \dim \CB_{\uG}(g)$. Then the following statements hold.
\begin{enumerate}[\rm(a)]
\item If $G = \Sp(V)$, then 
$$\frac{(n-s)^2}{2} \leq D(g) \leq  \frac{n(n-s)+n}{2} - \frac{1}{2}\sum_{i: 2|i}in_i \leq \frac{n(n-s)+n}{2} - \sum_{i: 2|i, n_i > 0}1.$$ 
If $G = \GO(V)$ then 
$$\frac{(n-s)^2-n_1}{2}+\frac{1}{2}\sum_{2|i}n_i \leq 
   \frac{(n-s)^2-n_1}{2} \leq D(g) \leq \frac{n(n-s)}{2}-\frac{1}{2}\sum_{i:2 \nmid i,n_i > 0}n_i$$
\item If $G = \Sp(V)$, then 
$$(1-1/q)^{n/2}q^{(n-s)^2/2} \leq |\CB_G(g)| \leq q^{(n(n-s)+n)/2}.$$
If $G = \GO(V)$, then 
$$(1-1/q)^{n/2}q^{((n-s)^2-n)/2} \leq |\CB_G(g)| \leq q^{(n(n-s)+n/3)/2}.$$
\end{enumerate}
\end{lem}

\begin{proof}
(a) We will follow in part the proof of \cite[Lemma 3.4(ii)]{LiSh1} and note that $s=\sum_i(i-1)n_i$, $n = \sum_i in_i$,
so that $n-s=\sum_i n_i$.
Suppose $G = \Sp(V)$. By \cite[Theorem 3.1(iii)]{LS},
\begin{equation}\label{sum10a}
  D(g)= \frac{1}{2}\sum_i in_i^2 +\sum_{i<j}in_in_j + \frac{1}{2}\sum_{i: 2\nmid i}n_i.
\end{equation}
Note that 
\begin{equation}\label{sum10}
  \sum_i in_i^2 +2\sum_{i<j}in_in_j = \bigl(\sum_i n_i\bigr)^2 + \sum_i (i-1)n_i^2+2\sum_{i < j}(i-1)n_in_j.
\end{equation}  
Hence $2D(g) \geq \sum_i in_i^2 +2\sum_{i<j}in_in_j \geq \bigl(\sum_in_i\bigr)^2 = (n-s)^2$. Next, 
$$\begin{aligned}s(n-s)+n & = \bigl(\sum_i (i-1)n_i\bigr)\bigl(\sum_i  n_i\bigr) + \sum_i in_i\\
    & \geq \sum_i(i-1)n_i^2 + 2 \sum_{i < j}(i-1)n_in_j + \sum_{i:2 \nmid i}n_i + \sum_{i:2|i}in_i\\
    & \geq 2D_G(g)-(n-s)^2+\sum_{i:2|i}in_i, 
    \end{aligned}$$
implying the statement for $\Sp(V)$.    

Suppose now that $G = \GO(V)$. By \cite[Theorem 3.1(iii)]{LS}, 
\begin{equation}\label{sum10b}
  D(g)= \frac{1}{2}\sum_i in_i^2 +\sum_{i<j}in_in_j - \frac{1}{2}\sum_{i: 2\nmid i}n_i.
\end{equation}  
Using \eqref{sum10}, we obtain
$$2D(g) \geq \bigl(\sum_i n_i\bigr)^2 +\sum_i(i-1)n_i^2 - \sum_{i: 2\nmid i}n_i \geq 
    \bigl(\sum_in_i\bigr)^2 +\sum_{2|i}n_i- n_1= (n-s)^2+\sum_{2|i}n_i- n_1.$$ 
Next, again using \eqref{sum10}, we have 
$$\begin{aligned}s(n-s) & = \bigl(\sum_i (i-1)n_i\bigr)\bigl(\sum_i  n_i\bigr)\\
    & \geq \sum_i(i-1)n_i^2 + 2 \sum_{i < j}(i-1)n_in_j\\
    & = 2D_G(g)-(n-s)^2+\sum_{i:2 \nmid i}n_i, 
    \end{aligned}$$
implying the statement for $\GO(V)$.  

\smallskip
(b) Suppose $G = \Sp(V)$. By \cite[Theorem 7.1(ii)]{LS}, $|\CB_G(g)|$ is a polynomial in $q$ of degree $D(g)$, and 
$$|\CB_G(g)| = q^{D'}\cdot\prod_{2 \nmid i}|\Sp_{n_i}(q)|\cdot\prod_{2 |i}|\GO^{\eps_i}_{n_i}(q)|$$
for suitable $D'$ and $\eps_i = \pm$ (note that $2|n_i$ when $2 \nmid i$). Note that
\begin{equation}\label{ord11}
\begin{array}{rcl}
  (q-1)^{m/2}q^{m(m+1)/2-m/2} & < |\Sp_m(q)| & < q^{m(m+1)/2},\\
  2(q-1)^{\lfloor m/2 \rfloor }q^{m(m-1)/2-\lfloor m/2 \rfloor} & \leq 
  |\GO^{\pm}_m(q)| & \leq \left\{\begin{array}{ll}\frac{2(q+1)}{q}q^{m(m-1)/2}, & m = 2,\\
  2q^{m(m-1)/2}, & m \neq 2.\end{array}\right.
\end{array}  
\end{equation}   
In our case, $q \geq 3$, so $2(q+1)/q < q$, and $\sum_i n_i/2 = (n-s)/2 \leq n/2$. It follows that 
$$(1-1/q)^{n/2}q^{D(g)}=(q-1)^{n/2}q^{D(g)-n/2} \leq |\CB_G(g)| < q^{\sum_{2|i}1+D(g)},$$
Together with (a), this implies the statement for $\Sp(V)$. 

Suppose now that $G = \GO(V)$. By \cite[Theorem 7.1(iii)]{LS}, $|\CB_G(g)|$ is a polynomial in $q$ of degree $D(g)$, and 
$$|\CB_G(g)| = q^{D'}\cdot\prod_{2|i}|\Sp_{n_i}(q)|\cdot\prod_{2 \nmid i}|\GO^{\eps_i}_{n_i}(q)|$$
for suitable $D'$ and $\eps_i = \pm$ (note that $2|n_i$ when $2|i$). Using \eqref{ord11}, we get 
$$(1-1/q)^{n/2}q^{D(g)}=(q-1)^{n/2}q^{D(g)-n/2} \leq |\CB_G(g)| < Aq^{D(g)},$$
where $A:= \prod_{i: 2 \nmid i}\alpha_i$, with $\alpha_i = 2$ if $n_i \neq 2$ and $\alpha_i = 2(q+1)/q$ if $n_i=2$.
In particular, $\alpha_i < q^{2n_i/3}$,
and so $A \leq q^{\sum_{2 \nmid i}2n_i/3}$. Since $n_1 \leq \sum_i n_i \leq n$,  together with (a) this implies the statement for $\GO(V)$. 
\end{proof}

In what follows, by $\supp(g)$ for $g \in \Spin^\eps_n(q)$ we mean the support of its image in $\Omega^\eps_n(q)$.
Also, the notation $\GL^\eps_m(q)$ means $\GL(\F_q^m)$ when $\eps=+$ and $\GU(\F_{q^2}^m)$ when $\eps=-$.

\begin{prop}\label{bcd-odd}
Let $q$ be an odd prime power, $V=\F_q^n$ be endowed with a non-degenerate, symplectic or orthogonal, bilinear form, and let
$G = \Sp(V)$, respectively, $\SO(V)$, $\Omega(V)$, or $\Spin(V)$. 
Let $g \in G$ be any element with $s=\supp(g)$. Then the following statements hold.
\begin{enumerate}[\rm(a)]
\item If $2|n$ and $G = \Sp(V)$, then 
$$q^{(n-s)^2/2-0.2n} \leq |\CB_G(g)| \leq q^{n(n-s)/2+0.5n}.$$
In particular, if $2|n \geq 4$, then $|G|^{3s/n} \geq |g^G| \geq |G|^{s/2n}$.
\item If $n \geq 3$ and $\Omega(V) \leq G \leq \GO(V)$, then 
$$q^{(n-s)^2/2-0.7n-1.3} \leq |\CB_G(g)| \leq q^{n(n-s)/2+n/6}.$$
In particular, if $n \geq 7$ and $G = \SO(V)$, $\Omega(V)$, or $\Spin(V)$, then $|G|^{3s/n} \geq |g^G| \geq |G|^{s/3n}$.
\end{enumerate}
\end{prop}

\begin{proof}
Write $g=g_\ssp u$, with $g_\ssp$ the semisimple part and $u$ the unipotent part. Then $g$ preserves the orthogonal decomposition
$$V = V_1 \oplus V_2 \oplus \bigl(\oplus^{t+2}_{i=3}V_{i}\bigr),$$
into non-degenerate subspaces $V_i$ of dimension
$\dim V_i = n_i$, where $g_\ssp$ acts as $1$ on $V_1$, $-1$ on $V_2$, and 
$$\CB_{I(V)}(g_\ssp) = \prod^{t+2}_{i=1}\CB_{G \cap I(V_i)}(g_\ssp) = 
   I(V_1) \times I(V_2) \times \prod^{t+2}_{i=3}\GL^{\eps_i}_{m_i}(q^{a_i}),$$
where $I = \Sp$ or $\GO$, $\eps_i = \pm$, $m_i,a_i \in \Z_{\geq 1}$, and $m_ia_i = n_i/2$. Let $u_i$ denote the image of $u$ in 
$I(V_i)$ when $i \leq 2$, and in $\GL^{\eps_i}_{m_i}(q^{a_i})$ when $i > 2$, and let $d_i$ denote the dimension of its kernel $U_i$
on $V_i$ when $i \leq 2$, on $\F_{q^{a_i}}^{m_i}$ when $i > 2$ and $\eps_i=+$, and on $\F_{q^{2a_i}}^{m_i}$ when $i > 2$ and 
$\eps_i=-$. Then the subspaces $U_i \otimes \Fqb$ are the distinct eigenspaces for $g$ on $V \otimes_{\F_q}\Fqb$; in 
particular,
\begin{equation}\label{supp10} 
  n-s = \max_i d_i
\end{equation}
Furthermore, 
\begin{equation}\label{supp11}
 \CB_{I(V)}(g) = \CB_{\CB_{I(V)}(g_\ssp)}(u) = \prod^{t+2}_{i=1}\CB_{\CB_{G \cap I(V_i)}(g_\ssp)}(u_i).
\end{equation} 
  
\smallskip
(a) Consider the case $G = \Sp(V)$. By \eqref{for-c10} (and \eqref{ord10}), Proposition \ref{gl-s}(a), and Lemma \ref{unip-odd}(ii),
$|\CB_{G \cap \Sp(V_i)}(g)| = |\CB_{\CB_{G \cap \Sp(V_i)}(g_\ssp)}(u_i)|$ is bounded below by 
$(1-1/q)^{n_i/2}q^{d_i^2/2}$. Since $q^{0.4} \geq 3^{0.4} > 1/(1-1/q)$ and $n \geq n_i$, by \eqref{supp10}--\eqref{supp11} 
we get $|\CB_G(g)| > q^{(n-s)^2/2-0.2n}$. On the other hand, by \eqref{ord10}, Proposition \ref{gl-s}(b), and Lemma \ref{unip-odd}(b),
$|\CB_{G \cap \Sp(V_i)}(g)| = |\CB_{\CB_{G \cap \Sp(V_i)}(g_\ssp)}(u_i)|$ is bounded from the above by 
$q^{(n_id_i+n_i)/2}$ when $i \leq 2$, and by $q^{n_id_i/2+0.2n_i}$ when $i > 2$, with the extra factor $q^{0.2n_i}$ accounting
for $q^{0.2n_i} = q^{0.4m_ia_i} \geq (3/2)^{m_i}$ when $\eps_i=-$. Now $n = \sum_i n_i$, so by \eqref{supp10}--\eqref{supp11} 
we get $|\CB_G(g)| < q^{n(n-s)/2+n/2}$, proving the first statement. Using 
$$q^{n(n+1)/2-0.6} < (9/16)q^{n(n+1)/2} < |G| < q^{n(n+1)/2},$$ 
(see \cite[Lemma 4.1(ii)]{LMT}), we obtain
$$q^{ns/2-0.6} < |g^G| < q^{ns-s^2/2+0.7n}.$$
The second statement is obvious if $s=0$. If $s \geq 1$ then $ns/2-0.6 \geq (s/2n)(n(n+1)/2)$. 
If $s \geq 2$, then $ns-s^2/2+0.7n \leq (3s/n)(n(n+1)/2-0.6)$. Finally, if $s=1$, then $g$ is a transvection (up to a sign), hence
$|g^G| = (q^n-1)/2 < |G|^{3/n}$, completing the proof of the second statement.

\smallskip
(b) Now we consider the orthogonal case. By \eqref{for-c10} (and \eqref{ord10}), Proposition \ref{gl-s}(a), and Lemma \ref{unip-odd}(ii),
$|\CB_{G \cap \GO(V_i)}(g)| = |\CB_{\CB_{G \cap \GO(V_i)}(g_\ssp)}(u_i)|$ is bounded from below by 
$(1-1/q)^{n_i/2}q^{d_i^2/2-n_i/2}$. Since $q^{0.4} > 1/(1-1/q)$ and $n \geq n_i$, by \eqref{supp10}--\eqref{supp11} 
we get $|\CB_\GO(V)(g)| > q^{(n-s)^2/2-0.7n}$. On the other hand, by \eqref{ord10}, Proposition \ref{gl-s}(b), and Lemma \ref{unip-odd}(b),
$|\CB_{G \cap \GO(V_i)}(g)| = |\CB_{\CB_{G \cap \GO(V_i)}(g_\ssp)}(u_i)|$ is bounded from the above by 
$q^{n_id_i/2+n_i/6}$ when $i \leq 2$, and by $q^{n_id_i/2+0.2n_i}$ when $i > 2$, again with the extra factor $q^{0.2n_i}$ accounting
for $q^{0.2n_i} = q^{0.4m_ia_i} \geq (3/2)^{m_i}$ when $\eps_i=-$. Now $n = \sum_i n_i$, so by \eqref{supp10}--\eqref{supp11} 
we get $|\CB_{\GO(V)}(g)| < q^{n(n-s)/2+n/6}$.
Since $[\GO(V):\Omega(V)] = 4$, we have that 
$$q^{-1.3}|\CB_{\GO(V)}(g)| < |\CB_{\GO(V)}(g)|/4 \leq |\CB_G(g)| \leq |\CB_{\GO(V)}(g)|$$
when $\GO(V) \geq G \geq \Omega(V)$, proving the first statement. 

To prove the second statement, we may again assume $s \geq 1$, and note that 
$$q^{n(n-1)/2-1.16} < (9/32)q^{n(n-1)/2} < |\Omega(V)| < |\SO(V)| = |\Spin(V)| = 2|\Omega(V)| < q^{n(n-1)/2},$$ 
if $n \geq 7$, (see \cite[Lemma 4.1(ii)]{LMT}). Furthermore, if $\bar{g}$ denotes the image of $g \in \Spin(V)$ in 
$\Omega(V)$, then $|\CB_{\Omega(V)}(\bar{g})| \leq |\CB_{\Spin(V)}(g)| \leq 2 \cdot |\CB_{\Omega(V)}(\bar{g})|$,
and so 
$$\frac{|\Omega(V)|}{|\CB_{\Omega(V)}(g)|} \leq |g^{\Spin(V)}| \leq \frac{|\SO(V)|}{|\CB_{\Omega(V)}(g)|}.$$
Hence, it suffices to prove that 
\begin{equation}\label{ord12}
  |\SO(V)|^{1-3s/n} \leq |\CB_{\Omega(V)}(g)| \leq |\CB_{\SO(V)}(g)| \leq |\Omega(V)|^{1-s/3n} \cdot q^{-0.64} < |\Omega(V)|^{1-s/3n}/2.
\end{equation}  
When $n \geq 8$, or if $n = 7$ but $s=1$, we have 
$\frac{(n-s)^2}{2}-0.7n-1.3 \geq \bigl(1-\frac{3s}{n}\bigr)\frac{n(n-1)}{2}$. 
If $s \geq 3$ then we also have 
$\frac{n(n-s)}{2}+\frac{n}{6} \leq \bigl(1-\frac{s}{3n}\bigr)\bigl(\frac{n(n-1)}{2}-1.8\bigr)$, 
proving \eqref{ord12}.  Finally, if $s \leq 3$, then $|G|^{s/3n} \leq q^{(n-1)/2} < P(G) \leq |g^G|$; and if $(n,s) = (7,1)$, 
then $|G|^{3s/n} > q^{17} > q^{2n} > |g^G|$,
completing the proof of the second statement.
\end{proof}

\begin{lem}\label{unip-even}
Let $q$ be a power of $2$, let $n$ be even, let $V=\F_q^n$ be endowed with a non-degenerate alternating bilinear form $(\cdot|\cdot)$, respectively a 
quadratic form associated to $(\cdot|\cdot)$, and let
$G = \Sp(V)$, respectively, $\GO(V)$, denote the corresponding isometry group of the form(s). Extend the form to 
$\overline V := V \otimes_{\F_q}\Fqb$, and let $\uG = \Sp(\overline V)$, respectively $\GO(\overline V)$.
Let $g =\oplus_iJ_i^{n_i} \in G$ be a unipotent element with $s=\supp(g)$, and let $D^\sharp(g) := \dim \CB_{\uG}(g)$. Then the following statements hold.
\begin{enumerate}[\rm(a)]
\item If $G = \Sp(V)$, then 
$$\frac{(n-s)^2}{2} \leq D^\sharp(g) \leq \frac{n(n-s)+n}{2}.$$ 
If $G = \GO(V)$ then 
$$\frac{(n-s)^2-n-n_1}{2} \leq D^\sharp(g) \leq \frac{n(n-s)}{2}-\frac{1}{2}\sum_{i:2 \nmid i,n_i > 0}n_i$$
\item If $G = \Sp(V)$, then 
$$(1-1/q)^{n/2}q^{(n-s)^2/2} \leq |\CB_G(g)| \leq q^{n(n-s)/2+1.3n}.$$
If $G = \GO(V)$, then 
$$(1-1/q)^{n/2}q^{(n-s)^2/2-n} \leq |\CB_G(g)| \leq q^{n(n-s)/2+0.8n}.$$
\end{enumerate}
\end{lem}

\begin{proof}
(a) The conjugacy classes $g^G$ of unipotent elements in $G$ are best represented in the form \cite[(4.4)]{LS}, where one decomposes
the $g$-module $V$ as 
\begin{equation}\label{sum20}
  \bigl(\oplus_i W(m_i)^{a_i}\bigr) \oplus \bigl(\oplus_j V(2k_j)^{b_j}\bigr), 
\end{equation}
where $b_j \leq 2$, $k_1 > k_2 > \cdots$, $m_1 > m_2 > \cdots$,
$g$ acts on $W(m_i)$ as $J_{m_i}^2$, and on $V(2k_j)$ as $J(2k_j)$, see \cite[Table 4.1]{LS}. We again record the Jordan 
canonical form of $g$ as $\sum^r_{i=1}J_{m_i}$, with $m_1 \geq m_2 \geq \ldots \geq m_r \geq 1$. Then 
$$D^\sharp(g) = \sum^r_{i=1}\bigl(im_i - \chi_V(m_i)\bigr),$$
where the function $\chi_V$ is defined as follows (see \cite[Lemma 6.2]{LS}): 
$$\chi_V(m) = \chi_{V(m)}(m) = \left\{\begin{array}{ll} m/2, & G = \Sp(V)\\ m/2+1, & G = \GO(V) \end{array}\right.$$
if $V(m)$ occurs in \eqref{sum20}, and 
$$\chi_V(m) = \chi_{W(m)}(m) = \left\{\begin{array}{ll} \lfloor (m-1)/2 \rfloor, & G = \Sp(V)\\ \lfloor (m+1)/2\rfloor, & G = \GO(V) \end{array}\right.$$
otherwise.

Suppose $G= \Sp(V)$. Then $\chi_V(m_i) = m_i/2-\nu_i$, where $\nu_i = 0$ if $V(m_i)$ occurs in \eqref{sum20}, 
$\nu_i = 1$ if $2|m_i$ but $V(m_i)$ does not occur in \eqref{sum20}, and $\nu_i= 1/2$ if $2 \nmid m_i$. It follows that 
$$D^\sharp(g) =  \sum^r_{i=1}(im_i-m_i/2) + \sum_{2 \nmid m_i}a_i + \nu,$$
where $\nu := 2\sum_i a_i$, with $i$ running over those $m_i$ such that $2|m_i$ but $V(m_i)$ does not occur in \eqref{sum20}. With $g$ written as $\oplus_i J_i^{n_i}$, we have that $\nu \leq \sum_{2|i}n_i \leq \frac{1}{2}\sum_{2|i}in_i$, and 
$\sum_{2 \nmid m_i}a_i = \frac{1}{2}\sum_{2 \nmid i}n_i$. Using Lemma \ref{2ways} and \eqref{sum10a}, we get
$$D^\sharp(g)=\frac{1}{2}\sum_i in_i^2 + \sum_{i < j}in_in_j +  \frac{1}{2}\sum_{2 \nmid i}n_i + \nu = D(g)+\nu.$$
Since $0 \leq \nu \leq \frac{1}{2}\sum_{2|i}in_i$, together with Lemma \ref{unip-odd}(a), this implies the statement for $\Sp(V)$.

Next suppose that $G= \GO(V)$. Then $\chi_V(m_i) = m_i/2+\mu_i$, where $\nu_i = 1$ if $V(m_i)$ occurs in \eqref{sum20}, 
$\mu_i = 0$ if $2|m_i$ but $V(m_i)$ does not occur in \eqref{sum20}, and $\mu_i= 1/2$ if $2 \nmid m_i$. It follows that 
$$D^\sharp(g) =  \sum^r_{i=1}(im_i-m_i/2) - \sum_{2 \nmid m_i}a_i - \mu,$$
where $\mu := \sum_j b_j$. With $g$ written as $\oplus_i J_i^{n_i}$, we have that $\mu \leq \sum_{2|i}n_i$, and 
$\sum_{2 \nmid m_i}a_i = \frac{1}{2}\sum_{2 \nmid i}n_i$. Using Lemma \ref{2ways} and \eqref{sum10b}, we get
$$D^\sharp(g)=\frac{1}{2}\sum_i in_i^2 + \sum_{i < j}in_in_j -  \frac{1}{2}\sum_{2 \nmid i}n_i - \mu = D(g)-\mu.$$
Since $\mu \geq 0$ and $2(n_1+\sum_{2|i}n_i) \leq 2n_1 + \sum_{2|i}in_i \leq n+n_1$, together with Lemma \ref{unip-odd}(a), this implies the statement for $\GO(V)$.

\smallskip
(b) By \cite[Theorem 7.3(ii)]{LS}, $|\CB_G(g)|$ is a polynomial in $q$ of degree $D^\sharp(g)$, and 
$$|\CB_G(g)| = 2^{t+\delta}q^{D'}\cdot\prod_{2 \nmid m_i}|I_{2a_i}(q)|\cdot\prod_{2 |m_i}|\Sp_{2a_i}(q)|$$
for a suitable $D'$; moreover, $I$ is either $\Sp$ or $\GO^\pm$, $0 \leq \delta \leq 1$, and $t$ is the number of $j$ such that
$k_j-k_{j+1} \geq 2$ in \eqref{sum20}. In particular, $(t+\delta+\mbox{the number of factors }\GO\mbox{ among the }I_{2a_i})$ is at most $n/2$. Using \eqref{ord11} and $(q+1)/q \leq 1.5 < q^{0.6}$, we obtain 
$$(1-1/q)^{n/2}q^{D^\sharp(g)}=(q-1)^{n/2}q^{D^\sharp(g)-n/2} \leq |\CB_G(g)| < q^{0.8n+D^\sharp(g)},$$
Now we can apply the estimates in (a) for $D^\sharp(g)$. 
\end{proof}

\begin{prop}\label{bcd-even}
Let $q$ be a power of $2$, $2|n \geq 4$, $V=\F_q^n$ be endowed with a non-degenerate alternating bilinear form $(\cdot|\cdot)$, respectively a quadratic form associated to $(\cdot|\cdot)$, and let
$G = \Sp(V)$, respectively, $\GO(V)$ or $\Omega(V)$. 
Let $g \in G$ be any element with $s=\supp(g)$. Then the following statements hold.
\begin{enumerate}[\rm(a)]
\item If $G = \Sp(V)$, then 
$$q^{(n-s)^2/2-0.5n} \leq |\CB_G(g)| \leq q^{n(n-s)/2+1.3n}.$$
In particular, $|G|^{3s/n} \geq |g^G| \geq |G|^{s/3n}$.
\item If $n \geq 8$ and $\Omega(V) \leq G \leq \GO(V)$, then 
$$q^{(n-s)^2/2-1.5n-1} \leq |\CB_G(g)| \leq q^{n(n-s)/2+0.8n}.$$
In particular, if $n \geq 8$ and $G = \Omega(V)$, then $|G|^{5s/n} \geq |g^G| \geq |G|^{s/3n}$.
\end{enumerate}
\end{prop}

\begin{proof}
Write $g=g_\ssp u$, with $g_\ssp$ the semisimple part and $u$ the unipotent part. Then $g$ preserves the orthogonal decomposition
$$V = V_1 \oplus \bigl(\oplus^{t+1}_{i=2}V_{i}\bigr),$$
into non-degenerate subspaces $V_i$ of dimension
$\dim V_i = n_i$, where $g_\ssp$ acts as $1$ on $V_1$, and 
$$\CB_{I(V)}(g_\ssp) = \prod^{t+1}_{i=2}\CB_{G \cap I(V_i)}(g_\ssp) = 
   I(V_1) \times \prod^{t+1}_{i=2}\GL^{\eps_i}_{m_i}(q^{a_i}),$$
where $I = \Sp$ or $\GO$, $\eps_i = \pm$, $m_i,a_i \in \Z_{\geq 1}$, and $m_ia_i = n_i/2$. Let $u_i$ denote the image of $u$ in 
$I(V_i)$ when $i = 1$, and in $\GL^{\eps_i}_{m_i}(q^{a_i})$ when $i > 1$, and let $d_i$ denote the dimension of its kernel $U_i$
on $V_i$ when $i = 1$, on $\F_{q^{a_i}}^{m_i}$ when $i > 1$ and $\eps_i=+$, and on $\F_{q^{2a_i}}^{m_i}$ when $i > 1$ and 
$\eps_i=-$. Then the subspaces $U_i \otimes \Fqb$ are the distinct eigenspaces for $g$ on $V \otimes_{\F_q}\Fqb$; in 
particular,
\begin{equation}\label{supp20} 
  n-s = \max_i d_i
\end{equation}
Furthermore, 
\begin{equation}\label{supp21}
 \CB_{I(V)}(g) = \CB_{\CB_{I(V)}(g_\ssp)}(u) = \prod^{t+1}_{i=1}\CB_{\CB_{G \cap I(V_i)}(g_\ssp)}(u_i).
\end{equation} 
  
\smallskip
(a) Consider the case $G = \Sp(V)$. By \eqref{for-c10} (and \eqref{ord10}), Proposition \ref{gl-s}(a), and Lemma \ref{unip-even}(ii),
$|\CB_{G \cap \Sp(V_i)}(g)| = |\CB_{\CB_{G \cap \Sp(V_i)}(g_\ssp)}(u_i)|$ is bounded below by 
$q^{d_i^2/2-n_i/2}$. Hence \eqref{supp20}--\eqref{supp21} imply
$|\CB_G(g)| > q^{(n-s)^2/2-0.5n}$. On the other hand, by \eqref{ord10}, Proposition \ref{gl-s}(b), and Lemma \ref{unip-even}(b),
$|\CB_{G \cap \Sp(V_i)}(g)| = |\CB_{\CB_{G \cap \Sp(V_i)}(g_\ssp)}(u_i)|$ is bounded from the above by 
$q^{n_id_i/2+1.3n_i}$ when $i = 1$, and by $q^{n_id_i/2+0.3n_i}$ when $i > 1$, with the extra factor $q^{0.3n_i}$ accounting
for $q^{0.3n_i} = q^{0.6m_ia_i} > (3/2)^{m_i}$ when $\eps_i=-$. Now $n = \sum_i n_i$, so by \eqref{supp20}--\eqref{supp21} 
we get $|\CB_G(g)| < q^{n(n-s)/2+1.3n}$, proving the first statement. Using 
$$q^{n(n+1)/2-0.84} < (9/16)q^{n(n+1)/2} < |G| < q^{n(n+1)/2},$$ 
(see \cite[Lemma 4.1(ii)]{LMT}), we obtain
$$q^{ns/2-0.8n-0.84} < |g^G| < q^{ns-s^2/2+n}.$$
The second statement is obvious if $s=0$. If $s \geq 4$ then $ns/2-0.8n-0.84 \geq (s/3n)(n(n+1)/2)$, 
and if $1 \leq s \leq 3$, then $|G|^{s/3n} \leq |G|^{1/n} < q^{(n+1)/2} < P(G) \leq |g^G|$, showing $|g^G| \geq |G|^{s/3n}$. 
If $s \geq 2$, then $ns-s^2/2+n \leq (3s/n)(n(n+1)/2-0.84)$. Finally, if $s=1$, then $g$ is a transvection, hence
$|g^G| = q^n-1 < |G|^{3/n}$, completing the proof of the second statement.

\smallskip
(b) Now we consider the orthogonal case. By \eqref{for-c10} (and \eqref{ord10}), Proposition \ref{gl-s}(a), and Lemma \ref{unip-even}(ii),
$|\CB_{G \cap \GO(V_i)}(g)| = |\CB_{\CB_{G \cap \GO(V_i)}(g_\ssp)}(u_i)|$ is bounded below by 
$(1-1/q)^{n_i/2}q^{d_i^2/2-1.5n_i}$. Hence \eqref{supp20}--\eqref{supp21} imply that
$|\CB_\GO(V)(g)| > q^{(n-s)^2/2-1.5n}$. On the other hand, by \eqref{ord10}, Proposition \ref{gl-s}(b), and Lemma \ref{unip-even}(b),
$|\CB_{G \cap \GO(V_i)}(g)| = |\CB_{\CB_{G \cap \GO(V_i)}(g_\ssp)}(u_i)|$ is bounded from the above by 
$q^{n_id_i/2+0.8n_i}$ when $i = 1$, and by $q^{n_id_i/2+0.3n_i}$ when $i > 1$, again with the extra factor $q^{0.3n_i}$ accounting
for $q^{0.3n_i} = q^{0.6m_ia_i} > (3/2)^{m_i}$ when $\eps_i=-$. Now $n = \sum_i n_i$, so by \eqref{supp20}--\eqref{supp21} 
we get $|\CB_{\GO(V)}(g)| < q^{n(n-s)/2+0.8n}$.
Since $[\GO(V):\Omega(V)] = 2$, we have that 
$$|\CB_{\GO(V)}(g)|/q \leq |\CB_{\GO(V)}(g)|/2 \leq |\CB_G(g)| \leq |\CB_{\GO(V)}(g)|$$
when $\GO(V) \geq G \geq \Omega(V)$, proving the first statement. 

To prove the second statement, we may again assume $s \geq 1$, and note that 
$$q^{n(n-1)/2-0.84} < (9/16)q^{n(n-1)/2} < |\Omega(V)|  < q^{n(n-1)/2},$$ 
if $n \geq 8$, (see \cite[Lemma 4.1(ii)]{LMT}). Hence, 
$$q^{ns/2-1.3n-0.84} < |g^G| < q^{ns-s^2/2+2n+1}.$$
If $s \geq 4$ then $ns/2-1.3n-0.84 \geq (s/3n)(n(n-1)/2)$, 
and if $1 \leq s \leq 3$, then 
$$|G|^{s/3n} \leq |G|^{1/n} < q^{(n-1)/2} < P(G) \leq |g^G|,$$ 
showing $|g^G| \geq |G|^{s/3n}$. 
If $s \geq 2$, then $ns-s^2/2+2n+1 \leq (5s/n)(n(n-1)/2-0.84)$. If $s = 1$ then $g \in \GO(V) \smallsetminus \Omega(V)$
(and we still have $|g^G| \leq q^n-1 < |G|^{3/n}$).
\end{proof}

Together, Propositions \ref{gl-s}, \ref{gu-s}, \ref{bcd-odd}, and \ref{bcd-even} imply

\begin{cor}\label{supp-size}
Let $G$ be any of the following quasisimple classical groups: $\SL_n(q)$ with $n \geq 2$, $\SU_n(q)$ with $n \geq 3$, 
$\Sp_n(q)$ with $2|n \geq 4$, or $\Omega^\pm_n(q)$ or $\Spin^\pm_n(q)$ with $n \geq 7$. If $g \in G$ has 
$s=\supp(g)$, then $|G|^{s/3n} \leq |g^G| \leq |G|^{5s/n}$.
\end{cor}

\begin{proof}[Proof of Theorem \ref{main}]
If $G$ is an exceptional group of Lie type, then the statement follows from the main result \cite[Theorem 1]{LiT}. Choosing $c$ small enough, we may assume that $G$ is (a quotient of) one of the groups listed in Corollary \ref{supp-size}; in 
particular, for $S:=g^G$ with $s=\supp(g)$ we have $s/3n \leq \log_{|G|}|S| \leq 5s/n$. Hence the statement follows from
Theorem \ref{main-bound3}, by taking $c \leq \sigma/5$. 
\end{proof}

In addition to Example \ref{sln2} and Lemma \ref{lower}, we offer another example showing that the term 
$\log_{|G|}|g^G|$ in Theorem A is optimal, up to a constant.

\begin{exa}
\label{Steinberg}
{\em Let $G$ be any finite group of Lie type, $g$ a semisimple element, and $\chi$ the Steinberg character of $G$.
Then by \cite[Theorem~15.5]{Steinberg}, $|\chi(g)| = |\CB_G(g)|_p$, the $p$-part of  
$|\CB_G(g)|$.  For instance, if $G: = \SL_n(q)$ and $g$ is a diagonal element with eigenvalue multiplicities  $a_1,\ldots,a_m$, then in the large $q$ limit,
$$|g^G| \sim q^{n^2 - \sum_i a_i^2}\sim |G|^{\frac{n^2-\sum a_i^2}{n^2-1}}$$
while
$$|\chi(g)| = q^{\sum_i \binom {a_i}2} = \chi(1)^{\frac{\sum_i a_i^2 - \sum a_i}{n^2-n}},$$
so if $\sum_i a_i^2$ is large compared to $\sum_i a_i$, then
$$1-\frac{\log |g^G|}{\log |G|}\approx \frac{\log |\chi(g)|}{\log \chi(1)}.$$}
\end{exa}

\section{Squares of conjugacy classes and Thompson's conjecture}

In this section we consider situations in which the square of a conjugacy class $x^G$ can be shown to be all or nearly all of $G$.  The main result is Theorem~\ref{thompson}, which proves
Thompson's conjecture for various families of unitary, symplectic, and orthogonal groups.  The strategy here is to choose a class $x$ with small centralizer and use the Frobenius formula 
in conjunction with character estimates
to show that every target element $g$ lies in $x^G\cdot x^G$.  This breaks down when $g$ has very small support, necessitating a separate analysis of such elements.  If $g$ is of the form
$\diag(g_1,I_{n-k})$ for some small value of $k$, and if $x$ is conjugate to an element of the form $\diag(x_1,x_2)$, where $x_2$ is real and $g_1$ can be written as a product of two conjugates of $x_1$, then $g$ lies in $x^G\cdot x^G$.  By choosing $x$ carefully, we can hope to treat all elements of bounded support.  Of course, the primary eigenvalue of an element of small support need not be $1$.  Because of this difficulty, our strategy at present assumes congruence conditions relating $n$ and $q$ for orthogonal and unitary groups.

\edt{We remark that Ore's conjecture, now a theorem of Liebeck, O'Brien, Shalev, and Tiep \cite{LOST2}, plays an important role in the proof of Theorem~\ref{thompson}, via Lemma~\ref{real}.}

For groups of type $\PSL_n(q)$, Thompson's conjecture is already known \cite{EG}.  Theorem~\ref{almost Singer SLU} shows that there are many regular semisimple conjugacy classes \edit{in $\SL_n(q)$ and $\SU_n(q)$}, including all those with irreducible characteristic polynomial, for which 
the first part of the argument works, and $x^G\cdot x^G$ contains all elements whose support is greater than an absolute constant.

\begin{lem}\label{Sp splitting}
Let $V = \F_q^n$ with \edt{$n \geq 117$}. If $G := \SU(V)$, $\Sp(V)$, or $\Omega(V)$, and $g\in G$ satisfies $|\CB_G(g)| \ge |G|^{6/7}$, then $V$ admits an orthogonal decomposition $V_1\oplus V_2$ of non-degenerate subspaces
with $\dim(V_2) >  2(\dim V)/3$, such that $g(V_i) = V_i$, and $g$ acts as a scalar $\lambda$ on $V_2$, with
$\lambda^{\edt{q_0+1}}=1$ in the $\SU$-case and $\lambda^2=1$ otherwise.
\end{lem}

\begin{proof}
Let $D=n(n+1)/2$ when $G = \Sp(V)$ and $D = n(n-1)/2$ when $G = \Omega(V)$. As mentioned in the proofs of Propositions
\ref{bcd-odd} and \ref{bcd-even}, \edt{$|G| > q^{D-1.16}$}. Now, if $g \in G$ has support $s=\supp(g)$, then Propositions 
\ref{bcd-odd} and \ref{bcd-even} show that $|\CB_G(g)| \leq q^{D+1.3n-ns/2}$. 
If $0< \varepsilon < 1$ and
$|\CB_G(g)| \geq |G|^{1-\eps}$, then  \edt{$D+1.3n-ns/2 \geq (D-1.16)(1-\eps)$}, and so 
$$\edt{s < \frac{2\eps D}{n}+2.6 + \frac{2.32}{n} \leq \eps(n+1)+2.6 +\frac{2.32}{n}.}$$
Taking $\eps=1/7$, when \edt{$n \geq 117$} we then have $s < n/6$; in particular, the primary eigenvalue $\lambda$ of $g$ satisfies
$\lambda^{q_0+1}=1$ in the case $G=\SU(V) \cong \SU_n(q_0)$, respectively $\lambda=\pm 1$ in the remaining cases. By \cite[Lemma 6.3.4]{LST}, $V$ admits a $g$-invariant orthogonal decomposition $V_1\oplus V_2$ such that 
$\dim(V_2) \ge n-2s > 2n/3$ and $g$ acts as $\lambda$ on $V_2$.

The same argument applies to the case $G=\SU(V)$, using the estimates in Proposition \ref{gu-s}.
\end{proof}

\smallskip
In what follows, we will fix $\Cl \in \{\SU,\Sp,\Omega\}$ and work with $\Cl^\eps_{n}(q)$,
with the convention that, if we choose $\Cl= \Sp$ then all $\Cl^\eps$ will be $\Sp$ (regardless of $\eps$) and $2|n$, 
and if we choose $\Cl= \SU$ then all $\Cl^\eps$ will be $\SU$, whereas if we choose
$\Cl=\Omega$, then $\Cl^\eps=\Omega^\eps$ with $\eps = \pm$ and also $2|n$ if $2|q$. If $m < n$, then 
$\Cl^\eps_{m}(q)$ can be naturally embedded in $\Cl^{\eps'}_{n}(q)$ via $x \mapsto \diag\bigl(x,I_{n-m}\bigr)$.
For $g\in \Cl^\eps_{m}(q)$ and $S$ a normal subset of $\Cl^{\eps'}_{n}(q)$, where either $(m,\eps) = (n,\eps')$ or $n > m$, we say $S$ \emph{represents $g$} if the natural embedding of $\Cl^\eps_{m}(q)$ into 
$\Cl^{\eps'}_{n}(q)$ maps $g$ to an element of $S$.  We say an element $x\in G=\Cl^{\eps'}_{n}(q)$ \emph{covers $g$} if $x^G\cdot x^G$ represents $g$.

\begin{lem}
\label{real extension}
If $g\in \Cl^\eps_{r}(q)$ is covered by $x\in \Cl^{\alpha}_{m}(q)$, where $m \geq r$, and $y$ is any real element of $\Cl^\beta_{n}(q)$, then $g$ is covered by 
$$\diag(x,y)\in \Cl^{\alpha}_{m}(q)\times \Cl^{\beta}_{n}(q)< \Cl^{\alpha\beta}_{m+n}(q).$$
\end{lem}

\begin{proof}
By assumption, $g$ viewed as an element of $\Cl_{m}(q)$ is $x_1x_2$ for some conjugates $x_1, x_2$ of $x$. As $y$ is 
real, $\diag(x,y)$ is conjugate to $\diag(x_1,y)$ and $\diag(x_2,y^{-1})$. Hence $x_1 x_2$ is covered by $\diag(x,y)$.
\end{proof}

\begin{lem}\label{flip}
Let $x \in \Cl^+_{2m}(q)$ and $y \in \Cl^+_{2n}(q)$. If $\Cl=\Omega$ and $2 \nmid q$, assume in addition that $2|m$ and $2|n$. Then 
$\diag(x,y)$ is conjugate to $\diag(y,x)$ in $\Cl^+_{2m+2n}(q)$.
\end{lem}

\begin{proof}
We may assume that $\Cl^+_{2m}(q) = \Cl(U)$, where $U = \oplus^m_{i=1}\Span(u_{2i-1},u_{2i})$ is an orthogonal sum of 
$2$-spaces, with a Witt basis $(u_{2i-1},u_{2i})$ and moreover $\sQ(u_{2i-1})=\sQ(u_{2i})=0$ if in addition 
$\Cl=\Omega$ and $2|q$, and with $\sQ(u_{2i-1}) = 1$, $\sQ(u_{2i}) = -1$, $(u_{2i-1}|u_{2i}) = 0$ when
$\Cl=\Omega$ and $2 \nmid q$. Write $\Cl^+_{2n}(q) = \Cl(V)$ with $V = \edit{\oplus^n_{i=1}}\Span(v_{2i-1},v_{2i})$ in a similar manner. 

First we assume that $n=1$, and either $\Cl=\SU$, $\Sp$, or $2|q$ and \edt{$\Cl=\Omega$}. Then the linear transformation
$$\begin{aligned}f:&~u_1 \mapsto v_1,~u_2 \mapsto v_2,~u_3 \mapsto u_1,~u_4 \mapsto u_2,~u_5 \mapsto u_3,\ldots\\ 
    &~u_{2m-3} \mapsto u_{2m-1},~u_{2m} \mapsto u_{2m-2},~v_1 \mapsto u_{2m-1},~v_2 \mapsto u_{2m}\end{aligned}$$
belongs to $\SU(U \oplus V)$, respectively $\Sp(U \oplus V)$. Suppose $2|q$ and $\Cl=\Omega$. Then 
$f$ fixes the maximal totally singular subspace $\Span(u_1,u_3, \ldots,u_{2m-1},v_1)$ of $U \oplus V$, hence
$f \in \Omega(U \oplus V)$ by \cite[Lemma 2.5.8]{KlL}.

Next suppose that $n=2$, $2 \nmid q$, and $\Cl=\Omega$, and consider the linear transformation
$$\begin{aligned}f:&~u_1 \mapsto v_1,~u_2 \mapsto v_2,~u_3 \mapsto v_3,~u_4 \mapsto v_4, u_5 \mapsto u_1,~
    u_{6} \mapsto u_{2}, u_7 \mapsto u_3,~u_8 \mapsto u_4,\ldots,\\
     &~u_{2m-1} \mapsto u_{2m-5},~u_{2m} \mapsto u_{2m-4},~v_1 \mapsto u_{2m-3},~v_2 \mapsto u_{2m-2},
     ~v_3 \mapsto u_{2m-1},~v_4 \mapsto u_{2m}.\end{aligned}$$
Clearly $f \in \GO(U \oplus V)$, but we want to show that $f \in \Omega(U \oplus V)$.  Note that
$$u_1 \mapsto v_1 \mapsto u_{2m-3} \mapsto u_{2m-7} \mapsto \ldots \mapsto u_{5} \mapsto \edt{u_1},$$
is an $(m/2+1)$-cycle, which is a product of $m/2$ reflections of the form
$$\rho_{w}: x \mapsto x-\frac{2(x|w)}{(w|w)}w$$
with $w = u_1-v_1,v_1-u_{2m-3}, \ldots,u_5-u_1$, each of norm $\sQ(w) = 2$. The same holds for the sequence starting at $u_3$. The
two sequences starting at $\edt{u_2}$ and $u_4$ each give us a product of $m/2$ reflections of the form $\rho_w$ with $\sQ(w)=-2$.
Thus $f$ is a product of $2m$ reflections, whence $\det(f)=1$, and its spinor norm is the class of 
$2^{m}(-2)^m$, a square since $2|m$. 
Hence $f \in \Omega(U \oplus V)$ in this case as well.

The above $f$ moves $v_1,v_2$, respectively $v_1, \ldots,v_4$, to the front of $u_1, \ldots,u_{2m}$. In the general case of any 
$n$, a sequence of such transformations moves $v_1, \ldots,v_{2n}$ to the front of $u_1, \ldots,u_{2m}$, and thus conjugates 
$\diag(x,y)$ to $\diag(y,x)$.
\end{proof}

\begin{lem}
\label{union}
Suppose $r,m,n \in \Z_{\geq 1}$, and moreover $2|m$ and $2|n$ if $\Cl=\Omega$ and $2 \nmid q$.
If the elements $g_1,\ldots,g_k\in \Cl^\alpha_{r}(q)$ are all covered by a real element $x \in \Cl^+_{2m}(q)$ and the elements $h_1,\ldots,h_l\in \Cl^\beta_{s}(q)$ are all covered by a real element $y \in \Cl^+_{2n}(q)$, then the $g_i$ and $h_j$ are all covered by 
the real element $\diag(x,y) \in \Cl^+_{2m+2n}(q)$.
\end{lem}

\begin{proof}
The assumptions and Lemma \ref{flip} imply that the elements 
$$z_1:=\diag(x,y),~z_2:=\diag(x,y^{-1}),~z_3:=\diag(y,x),~z_4:=\diag(y,x^{-1})$$ 
are all in the same conjugacy class $C$.  Conjugating $z_1$ and $z_2$ by elements in 
$\Cl^+_{2m}(q)\times I_{2n}$ and multiplying together, we see that every 
$\diag(g_i,I_{2m+2n-r})$ belongs to $C^2$.  Conjugating $z_3$ and $z_4$ by elements in 
$I_{2m}\times \Cl^+_{2n}(q)$ and multiplying, we see that every
$\diag(h_j,I_{2m+2n-s})$ lies in $C^2$.
\end{proof}


\begin{lem}
\label{real}
For every positive integer $r \geq 3$, every element $g\in \Cl^\alpha_{r}(q)$ is covered by a real element in $\Cl^+_{4m}(q)$,
where $\edt{\max(4,r) < 2m \leq r+3}$, and $2|m$ if $\Cl=\Omega$ and $2 \nmid q$.  
\end{lem}

\begin{proof}
Embedding $\Cl^\alpha_r(q)$ in $\Cl^+_{2m}(q)$ and replacing $g$ by $\diag(g,I_s)$ for a suitable $s$, we may assume that 
$g \in \Cl^+_{2m}(q)$ with $m$ as specified.
\edt{(Note that in the case of $\Omega_r^{-}$ with $2|r$, we take $m = r/2+1$.)}
By \cite[Theorem 1]{LOST2}, every $g$ in $\Cl^+_{2m}(q)$ is a commutator $xyx^{-1}y^{-1}$. 
By Lemma \ref{flip}, $z := \diag(x, x^{-1}) \in \Cl^+_{4m}(q)$ is conjugate to $\diag(x^{-1}, x) = z^{-1}$, and thus $z$ is real. 
Conjugating $z^{-1}$ by $\diag(y,I_{2m})$ we see that $z$ is also conjugate 
to $t:=\diag(yx^{-1}y^{-1}, x)$.  It follows that $\diag(g, I_{2m})=zt$ lies in the square of the conjugacy class of $z$.
\end{proof} 

\begin{lem}
\label{Sp master element}
For all positive integers $k$ and prime powers $q$, there exists a positive integer $r$ and a real element $x\in \Cl^+_{2r}(q)$, both depending on \edt{$k$} and $q$, such that $x$ covers every element of $\Cl^\alpha_{l}(q)$ for all integers $l\in [1,k]$ and $\alpha = \pm$.
\end{lem}

\begin{proof}
Let $N$ denote the sum of the conjugacy class numbers of all $\Cl^\alpha_l(q)$ with $1 \leq l \leq k$ and $\alpha=\pm$.
By Lemma \ref{real}, each such class $g_i$ is covered by a real class $x_i$ in $\Cl^+_{4m_i}(q)$.
The statement now follows from Lemma~\ref{union}, by taking $r = 2\sum^{N}_{i=1}m_i$ and $x :=\diag(x_1, \ldots,x_N)$.
\end{proof}

\edt{In the next theorem, we remark that the congruence conditions on $q$ ensure that the central extension of $G$ which lies in $\GL_n(\bar\F_q)$ has a large enough center  that every element of $G$ of small support can be represented by an $n\times n$ matrix for which the primary eigenvalue is $1$, as needed for \eqref{g sign}.
In particular, when $n$ and $q$ are odd we have no results about $\Omega_{n}(q)$ because the center of $\SO_n(q)$ is trivial, and we do not know how to show that elements of small support with principal eigenvalue $-1$ lie in $S^2$.}

\begin{thm}
\label{thompson}
Let $q$ be a prime power and let $G \in \{\mathrm{PSU}_n(q),\PSp_{n}(q),\mathrm{P}\Omega^\eps_n(q)\}$.  Suppose that 
$(q+1)|n$ in the $\SU$-case, and that, if $2 \nmid q$ then $2|n$ and $\eps = (-1)^{n(q-1)/4}$ in the $\Omega$-case. 
If $n$ is sufficiently large, then there exists a conjugacy class $S$ in $G$ such that $S^2 = G$.
\end{thm}

\begin{proof}
(a) Ellers and Gordeev \cite[Table 1]{EG} already proved Thompson's conjecture  for simple classical groups when \edt{$q \geq 8$}. Hence it suffices to prove the theorem for $q \leq 7$ and \edt{$n \geq 117$} sufficiently large. 
For consistency with the $\Cl^\eps_n(q)$ notation, in the $\mathrm{PSU}$-case, 
we write $\Cl_n(q^2) = \SU_n(q)$, $\F := \F_{q^2}$, and $V:=\F^n$; otherwise, $\Cl_n(q)$ is $\Sp_n(q)$ or $\Omega_n^{\pm}(q)$, $\F := \F_q$, and  $V:=\F^n$. Replacing $G$ by $G = \Cl(V)$,
it suffices to prove that there exists a real conjugacy class $S$ in $G$ such that $S^2$ contains a scalar multiple of every non-central element $g \in G$. 

\smallskip
(b) By Lemma~\ref{Sp splitting}, if $g\in G$ satisfies $|\CB_G(g)| \ge |G|^{6/7}$, then $V$ admits an orthogonal decomposition 
$V_1\oplus V_2$ with $\dim(V_2) > 2\dim(V)/3$, such that $g(V_i) = V_i$, and $g$ acts as a scalar $\lambda$ on $V_2$,
with $\lambda^{q+1}=1$ in the $\SU$-case and $\lambda^2=1$ otherwise.
By Theorem~\ref{main-bound2}, if $|\CB_G(g)| \le |G|^{6/7}$, then there exists $\delta > 0$, independent of $V$, such that 
\begin{equation}
\label{ok g}
|\chi(g)| \le \chi(1)^{1-\delta}
\end{equation}
for every  irreducible character of $G$.
By \cite[Theorem~1.3]{GLT2}, there exists $\alpha>0$ (depending on $\delta$) such that if $x \in G$ satisfies 
$|\CB_G(x)| \le |G|^\alpha$, then 
\begin{equation}
\label{good x}
|\chi(x)| \le \chi(1)^{\delta/3}
\end{equation}
for all irreducible characters $\chi$.
By the Frobenius formula, 
$g\in G$ lies in $x^G \cdot x^G$ if
\begin{equation}
\label{positive}
\sum_{\chi \in \Irr(G)} \frac{\chi(x)^2\bar\chi( g)}{\chi(1)} > 0.
\end{equation}
By \cite[Theorem~1.2]{LiSh3}, $\sum_{1_G \neq \chi \in \Irr(G)}\chi(1)^{-\delta/3} \to 0$ when $n \to \infty$. Hence, if $n$ is large enough
\edt{and both \eqref{ok g} and \eqref{good x} hold, then}
$$\sum_{\chi\neq 1_G} \frac{|\chi(x)^2\bar\chi(g)|}{\chi(1)} < 1,$$
implying \eqref{positive} .   We fix $B > 0$ such that if $x$ satisfies \eqref{good x} and $g$ satisfies \eqref{ok g}, then $\dim(V) \ge B$ implies $g \in x^G \cdot x^G$.

\smallskip
(c) For any sufficiently large integer $d$, if $V$ has an orthogonal decomposition $V_3\oplus V_4$, and 
$$x=\diag(x_3, x_4)\in \Cl(V_3)\times \Cl(V_4)< \Cl(V)=G,$$
where $\dim(V_3) > d\dim(V_4)$ and the characteristic polynomial of $x_3$ has no irreducible factors of degree $<d$, then 
\begin{equation}
\label{small x centralizer}
|\CB_G(x)| \le |\CB_{\GL(V)}(x)| < |G|^\alpha.
\end{equation}
Indeed, suppose $d > 4\alpha$ and $n=\dim(V)$ is sufficiently large such that 
$$n\biggl( \frac{\alpha}{4}-\frac{1}{d} \biggr)  > 1.3.$$ 
The assumptions imply that any eigenspace of $x$ has dimension at most 
$$\dim(V_4)+\dim(V_3)/d < 2\dim(V_3)/d \leq 2n/d,$$
hence $s:=\supp(x) \geq n-2n/d$. By Propositions  \ref{bcd-odd} and \ref{bcd-even}, 
$$|\CB_G(x)| \leq q^{n(n-s)/2 + 1.3n} \leq q^{n^2/d+1.3n} < q^{\alpha n^2/4} < |G|^{\alpha}$$
in the non-$\SU$ cases. In the $\SU$-case we argue similarly, using Proposition \ref{gu-s}. Fix such a $d$.

\smallskip
(d) We now fix a non-degenerate
space $W$ over $\F$ of dimension $\ge d$, unitary if $G = \SU(V)$, symplectic if $G=\Sp(V)$, and quadratic of type $+$ if $G = \Omega(V) \cong \Omega^\eps_n(q)$, and a real semisimple element $h\in \Cl(W)$ whose characteristic polynomial has no irreducible factors over $\F$ of degree less than $d$. 
(For instance, $\SU_{4d}(q)$ and its subgroup
$\Sp_{4d}(q)$ contain a semisimple element of order $(q^{2d}+1)/\gcd(2,q-1)$ as does $\Omega^-_{4d}(q)$; moreover,
such an element is real by \cite[Proposition 3.1]{TZ2}.  Next, $\Omega^+_{4d}(q) > \Omega^-_{2d}(q) \times \Omega^-_{2d}(q)$ contains 
a semisimple element of order $(q^d+1)/\gcd(2,q-1)$, which is again real by \cite[Proposition 3.1]{TZ2}.)

Next, we fix some integer
\begin{equation}\label{for-k}
  k > \max(B,\dim(W))
\end{equation}  
and apply Lemma~\ref{Sp master element} 
to find a non-degenerate space $V_0$ over $\F$, unitary if $G = \SU(V)$, symplectic if $G = \Sp(V)$, and quadratic if 
$G = \Omega(V)$, and a real element $y_k\in \Cl(V_0)$ such that $y_k$ covers every 
element in \edt{$\Cl^\beta_l(|\F|)$}, $1 \leq l \leq k$.  By Lemma~\ref{real extension}, if $y_k$ is replaced by $\diag(y_k,z)$ for any real element 
$z$ in any
\edt{$\Cl^\gamma_s(|\F|)$}, it still has the above covering property.  
We may therefore assume that $\dim(V_0)$ lies in any desired congruence class modulo $\dim(W)$, and that $V_0$ is also of 
type $\eps$ when $G = \Omega^\eps_n(q)$.

Suppose that $n$ is sufficiently large. Then we can write $n=\dim(V) = \dim(V_0)+N\dim(W)$, with $N$ sufficiently large,
and that $V$ is isometric to $V_0 \oplus W^N$. 

\smallskip
We claim that if $V_0$ and $y_k$ are fixed as above, $N > d\dim(V_0)/\dim(W)$ is sufficiently large,  
$$V := V_0\oplus\underbrace{W\oplus W\oplus\cdots\oplus W}_N,$$
and
$$x_N := \diag\bigl(y_k,\underbrace{h, h, \ldots, h}_N\bigr)\in \Cl(V_0)\times \Cl(W)^N< \Cl(V)=G,$$
then for every $g\in \Cl(V)$, subject to the hypothesis on $(n,q,\eps)$ (which guarantees that $\mathbf{Z}(G) \cong C_2$ 
in the case of $\Sp/\Omega$ with $2 \nmid q$ and $\mathbf{Z}(G) \cong C_{q+1}$ in the case $G = \SU(V)$),
$zg$ lies in $(x_N)^{G} \cdot (x_N)^{G}$ for some $z \in \mathbf{Z}(G)$.
%

\smallskip
Let $f_\lambda(g)$ denote the maximum dimension of $V_2$ where 
\begin{equation}\label{for-f1}
  V = V_1\oplus V_2
\end{equation}  
is a $g$-stable orthogonal splitting and 
$g$ acts as a scalar $\lambda$ on $V_2$, with $\lambda^{q+1}=1$ in the $\SU$-case and $\lambda=\pm 1$ otherwise;
and let $f(g) = \max_\lambda f_\lambda(g)$.

To prove the claim in general, we divide into three cases.  
\vskip 10pt\noindent
(d1) $f(g) \geq \dim(V)-k$.  
\par In this case, we may assume $\dim(V_2) \geq 4$ in the decomposition \eqref{for-f1} for some eigenvalue $\lambda$. Hence, in the case $G= \Omega(V)$, $g$ is centralized by elements $u$ 
in any chosen $\Omega(V_2)$-coset in $\GO(V_2)$. 
Likewise, in the case $G= \SU(V)$, $g$ is centralized by elements $u$ 
in any chosen $\SU(V_2)$-coset in $\GU(V_2)$. 
Conjugating $g$ using elements in $\Sp(V)$, $\GU(V)$, or $\GO(V)$, and then 
by suitable elements like $u$ in the case $G = \SU(V)$ or $\Omega(V)$, and replacing $g$ by $zg$ for a suitable 
$z \in \mathbf{Z}(G)$ if necessary, we may assume that
\begin{equation}\label{g sign}
  g = \diag(g_1,I_{f(g)})\in \Cl(V_1) \times \Cl(V_2)< \Cl(V).
\end{equation}
As $\dim V_1 \leq k$, $g_1$ is covered by $y_k$, so as $h$ is real, Lemma~\ref{real extension} implies \edt{$g$} belongs to 
$(x_N)^G \cdot (x_N)^G$.
\vskip 10pt\noindent
(d2) $f(g) \le 2\dim(V)/3$.\par
This condition implies \eqref{ok g} for $g$ \edt{by the argument of (b)}.  The choice of \edt{$N$} guarantees that
$d\dim(V_0) < N\dim(W)$, hence \eqref{small x centralizer} holds for $x_N$. Now we deduce \eqref{good x} for $x_N$, which implies 
$g \in (x_N)^G \cdot (x_N)^G$.
\vskip 10pt\noindent
(d3) $\dim(V)-k > f(g) > 2\dim(V)/3$.
\par Recall $n=\dim(V)$. When $N$ is large enough, we have $2f(g)-n > n/3 > \dim(V_0)$. Let $t$ denote the largest integer such that 
\begin{equation}\label{for-f2}
  e:=\dim(V_0) + t\dim(W)\le 2f(g)-n.
\end{equation} 
By the choice of $t$, $e > 2f(g)-n-\dim(W)$. Hence
$$e-(3f(g)-2n) > n-f(g)-\dim(W) > k-\dim(W) > 0$$
by \eqref{for-k} and (d3). 
Together with 
\eqref{for-f2}, this implies that
\begin{equation}\label{for-f3}
  \frac{n-e}{2} \leq f(g)-e < \frac{2(n-e)}{3}.
\end{equation}

Arguing as in part (d1) we may assume that $f(g)=f_1(g)$ and that
$$g = \diag(I_{\dim(V_0)+t\dim(W)},g_1)\in \Cl(V_0 \oplus W^t) \times \Cl(W^{N-t})< \Cl(V).$$
Note that $x_N$ is conjugate to both
$$\diag\bigl(y_k,\underbrace{h,\ldots,h}_{t},\underbrace{h,\ldots,h}_{N-t}\bigr)\in \Cl(V_0) \times \Cl(W)^{t}\times \Cl(W)^{N-t}< \Cl(V)$$
and
$$\diag\bigl(y_k^{-1},\underbrace{h^{-1},\ldots,h^{-1}}_{t},\underbrace{h,\ldots,h}_{N-t})\in \Cl(V_0) \times \Cl(W)^{t}\times \Cl(W)^{N-t}< \Cl(V).$$
Conjugating each element by an element of the form 
$$(I_{\dim(V_0)+t\dim(W)},v) \in \Cl(V_0 \oplus W^t) \times \Cl(W)^{N-t}< \Cl(V),$$
it suffices to prove that $g_1$ is contained in the square of the \edt{conjugacy class in $\Cl(W^{N-t})$ of}
$$x' := \diag\bigl(\underbrace{h,\ldots,h}_{N-t}\bigr).$$ 
By the choice of $h$, inequality \eqref{small x centralizer} holds for $x'$, which implies \eqref{good x} for $x'$.
The construction of $g_1$ shows that $f_1(g_1) = f_1(g)-e=f(g)-e$, and $\dim(W^{N-t}) = n-e$, so 
$$\frac{1}{2}\dim(W^{N-t}) \leq f_1(g_1) < \frac{2}{3}\dim(W^{N-t})$$
by \eqref{for-f3}.
It follows that $f(g_1) = f_1(g_1) < 2\dim(W^{N-t})/3$, and so $g_1$ satisfies \eqref{ok g}.
As 
$$\dim(W^{N-t}) = n-e \geq n-(2f(g)-n) \geq 2(n-f(g)) > 2k > B$$ 
by \eqref{for-f2}, it follows that $g_1$ is in the square of the \edit{conjugacy} class of $x'$, completing the proof.
\end{proof}

\edit{
\begin{thm}
\label{almost Singer SLU}
For all $A>0$, there exists $B>0$ such that the following statement holds for all $n \in \Z_{\geq 1}$ and all prime powers $q$. 
If $G = \SL^\eps_n(q)$ for some $\eps=\pm$ and 
the characteristic polynomial of a semisimple element $x\in G$ factors, over $\F_q$ if $\eps=+$ and over $\F_{q^2}$ if $\eps=-$,  
into pairwise distinct irreducible polynomials $P_1,\ldots,P_k$
of degrees $\deg P_i \ge n/A$ for all $i$,  
then $x^{G}\cdot x^{G}$
contains every element $g \in G$ of support $\ge B$.  In particular, the square of the conjugacy class of a Singer element in $\SL_n(q)$ covers
all elements $g\in \SL_n(q)$ for which $\supp(g)$ exceeds an absolute constant value.
\end{thm}

\begin{proof}
Since the support of an element of $\tilde G:=\GL^\eps_n(q)$ is at most $n$, by enlarging $B$, we are free to make $n \geq A$ as large as we wish.
Also note that $k \leq A$.

Note that the element $x$ is regular semisimple, and $T:=\CB_{\tilde G}(x)$ is a maximal torus, so of order at most $(q+1)^n$. Moreover, the image of $T$ under the determinant map is 
the same as of $\tilde G$. Hence the conjugacy class of $x$ in $G$ is the same as its class in $\tilde G$.  Let $g\in G$.  To show that $g\in x^{G}\cdot x^{G}$, it suffices to prove that
$$\sum_{\chi \in \Irr(\tilde G)} \frac{\chi(x)^2\bar\chi(g)}{\chi(1)} \neq 0.$$
As $\det(g)=\det(x)=1$, for every character $\chi$ of degree $1$ we have $\chi(x) = \chi(x)^2\bar\chi(g) = 1$. 
Therefore, it suffices to prove that
\begin{equation}
\label{abs bound}
\sum_{\{\chi \in \Irr(\tilde G) \mid \chi(1)>1\}} \frac{|\chi(x)|^2|\chi(g)|}{\chi(1)} < q-\eps.
\end{equation}

For any fixed $\epsilon > 0$, choosing $B$ sufficiently large, the contribution of characters $\chi$ satisfying $\chi(1)\ge q^{\epsilon n^2}$ to \eqref{abs bound} is $o(1)$.  Indeed, consider any such character $\chi$ and any irreducible constituent $\psi$ of $\chi|_G$. 
Since $\tilde G/G \cong C_{q-\eps}$, by Clifford's theorem we have $\chi|_G = \psi_1 + \ldots +\psi_t$, where
$\psi_1=\psi, \ldots,\psi_t$ are distinct $\tilde G$-conjugates of $\psi$, and $t|(q-\eps)$. By Theorem~\ref{main-bound3}, 
$$|\psi_i(g)| \leq \psi_i(1)^{1-\sigma B/n} = (\chi(1)/t)^{1-\sigma B/n},$$
and so $|\chi(g)| \leq t(\chi(1)/t)^{1-\sigma B/n}$. As $\chi(1) \geq (q+1)^2 \geq t^2$, we obtain 
$$|\chi(g)/\chi(1)| \leq \chi(1)^{-\sigma B/2n} \leq q^{-\eps\sigma Bn/2}.$$
Since $|T| \leq (q+1)^n < q^{2n}$, it follows that the contribution of all these characters to \eqref{abs bound} is at most
$$q^{-\eps\sigma Bn/2}\sum_{\chi}|\chi(x)|^2 \leq q^{-\eps\sigma Bn/2}|T| <  q^{2n(1-\eps\sigma B/4)}$$
which is $o(1)$ when $B$ is large enough.

\smallskip
Any irreducible character $\chi$ of $\tilde G$ belongs to the rational Lusztig series labeled by a semisimple element 
$s$ in the dual group which can be identified with $\tilde G$. Consider the case $s \notin \ZB(\tilde G)$. Then 
$L:=\CB_{\tilde G}(s)$ is a proper Levi subgroup of $\tilde G$. Hence $\chi = \pm R^G_L(\varphi)$ is Lusztig induced from 
an irreducible character $\varphi$ of $L$, see \cite[Theorem 13.25]{DM}. We claim that either $\chi(x) = 0$ or 
$\chi(1) \geq q^{n^2/A^2-1}$. Indeed, assume that $\chi(x) \neq 0$. As $x$ is regular semisimple, the Steinberg characters 
${\mathsf{St}}_{\tilde G}$ of $\tilde G$ and ${\mathsf{St}}_L$ of $L$ take values $\pm 1$ at $x$. Applying \cite[Proposition 9.6]{DM} we have 
$$0 \neq \chi(x) = \pm ({\mathsf{St}}_{\tilde G}\cdot\chi)(x) = \pm \mathrm{Ind}^G_L({\mathsf{St}}_L\cdot \varphi)(x),$$
and so $x$ is contained in a conjugate of $L$. If $L = \GL^\pm_{a}(q^b)$ with $ab=n$, then $b > 1$ as $s \notin \ZB(\tilde G)$, and so
$\chi(1) \geq q^{n^2/4-2}$ by \cite[Lemma 5.8]{GLT}. Thus we may assume $L$ is of type 
$\GL^\pm_{m_1}(q^{a_1}) \times \ldots \times \GL^\pm_{m_r}(q^{a_r})$ with 
$r \geq 2$ and each $m_ia_i$ is a sum of some $n_j$'s; in particular, $m_ia_i \geq n/A$. Using \cite[Lemma 5.1(vi)]{GLT}, we then have
$\chi(1) \geq [G:L]_{p'} \geq  q^{m_1(n-m_1)}/2 \ge q^{n^2/A^2-1}$.

\smallskip
It remains therefore to consider the case $s \in \ZB(\tilde G)$, i.e. $\chi$ is a unipotent character times a linear character, so on $g$ and $x$, it can be treated as a unipotent character (but each such character occurs $q-\eps$ times).  
Each unipotent character $\chi$ is associated to a partition $\lambda(\chi)$ of $n$, and the value of $\chi$ at the regular semisimple 
element $x$ is given, up to a sign, by the value at the permutation $\pi \in \SSS_n$, which is a product of $k$
cycles of length $n_1, \ldots,n_k$, of the irreducible character of $\SSS_n$ associated to 
$\lambda(\chi)$, see e.g. \cite[Proposition~3.3]{LM}.  As $\pi$ consists of $k$ cycles, by \cite[Theorem~7.2]{LaSh2}, $|\chi(x)| \le 2^{k-1} k!$, which is bounded in terms of $A$.

Note that $\chi$ has level $j=n-\lambda_1$ by \cite[Theorem 3.9]{GLT}. 
If $\lambda_1 \le n/2$, then $\chi(1) \ge q^{n^2/4-2}$ by \cite[Theorem 1.2(ii)]{GLT}, and, as before, the contribution of all such unipotent characters to the left hand side of \eqref{abs bound} is $o(1)$. Hence it remains to consider the characters $\chi$ with $\lambda_1 > n/2$;
any such unipotent character is irreducible over $G$, see \cite[Corollary 8.6]{GLT}. 
For any fixed positive value of $j=n-\lambda_1 < n/2$, the number of partitions of $n$ with largest part $\lambda_1$ is 
$p(n-\lambda_1)$ (where $p(\cdot)$ is the partition function), and the degree of the associated unipotent character is at least
$q^{j(n-j)-1} > q^{nj/3}$ by \cite[Theorem 1.2(i)]{GLT}.
For these characters $\chi$, $\supp(g) \ge B$ implies by Theorem \ref{main-bound3} that
$|\chi(g)|/\chi(1) < q_0^j$ with $q_0:=q^{\sigma B/3}$. Note that 
$\sum^\infty_{j=1}1/(q_0^j-1) < \sum^{\infty}_{j=1}q_0^{-j+1}/(q_0-1) < 1/(q_0-2)$, and so
$$\edt{\sum_{j=1}^\infty \frac{p(j)}{q_0^j} < \sum_{j=1}^\infty \frac{2^j}{q_0^j} = \frac 2{q_0-2},}$$
which is $o(1)$ when $B$ is large enough.
Hence, the contribution of these characters to \eqref{abs bound} is less than $2^{k-1}k! (e^{1/(q_0-2)}-1)(q-\eps)=o(q-\eps)$, and 
the theorem follows.
\end{proof}

Since this paper was written, Theorem~\ref{almost Singer SLU} has been generalized: see \cite[Theorem 1.1]{LTT}.}

\section{Further applications}

\subsection{Mixing time of random walks on Cayley graphs}

Recall that the mixing time of a probability distribution $P$ on a finite group $G$ is the smallest integer $n$ such that $\Vert P^{*n}-U_G\Vert_1 < 1/e,$
where $U_G$ denotes the uniform distribution on $G$.  The mixing time of a generating set $S$  of $G$ means the mixing time of the uniform distribution $U_S$ on $S$.
The theorem of Diaconis and Shahshahani \cite{DS} asserts that for the set of transpositions of $\mathsf{S}_n$,
the mixing time is asymptotic to $\frac{n\log n}2$.  By comparison, every element in $\mathsf{S}_n$ is the product of at most $n-1$ transpositions.

For any constants $C_1$ and $C_2$ there exists $\epsilon > 0$ so that if $n$ is sufficiently large and  $S$ is any conjugacy class in $\mathsf{S}_n$ of permutations fixing all but $C_1$ points, a random product of less than $C_2 n$ elements has probability greater than $1/2$ of fixing more than $\epsilon n$ elements.  Thus the mixing time for conjugacy classes of bounded support is superlinear in $n$, implying that for symmetric and alternating groups, the maximum ratio of mixing time over covering number for conjugacy classes goes to $\infty$.

In this subsection, we show that the situation is different for finite simple groups of Lie type. Liebeck and Shalev proved \cite[Corollary~1.2]{LiSh2} that if $G$ is such a group and $S$ is a conjugacy class of $G$,
then the diameter of the Cayley graph $\Gamma(G,S)$ is less than $C\frac{\log |G|}{\log |S|}$, where $C$ is an absolute constant.  
Note that this bound is optimal \edt{up to a constant factor}.

Here we prove  the same result, though with a different constant, for mixing time.
(The special example of transvections in
$\SL_n(q)$ was handled by Hildebrand \cite{Hil}; \edit{furthermore,  the case of semisimple classes 
whose centralizer is a Levi subgroup is treated in \cite[Theorem 1.12]{BLST}}.) 
This improves on the previously known upper bound $O\bigl(\frac {\log^3|G|}{\log^2|S|}\bigr)$ \cite[Corollary~1.14]{LiSh2}, and 
proves a conjecture of Shalev \cite[4.3]{Sh}.
It also resolves a conjecture made by Lubotzky in \cite[p.179]{Lub}, stating that, if $G$ is a finite simple group 
(of Lie type) and $S$ is a non-trivial conjugacy class
of $G$, then the mixing time of the Cayley graph $\Gamma(G, S)$ is
linearly bounded above in terms of the diameter of $\Gamma(G, S)$.

\begin{thm}\label{mixing}
There exists an absolute constant $C'$ such that if $S=g^G$ is any non-trivial conjugacy class in a finite simple group $G$ of Lie type, then the mixing time of the random walk on the Cayley graph $\Gamma(G,S)$ is less than $C' \frac{\log |G|}{\log |S|}$.
\end{thm}

\begin{proof}
Let $c$ denote the constant in Theorem~\ref{main}.
We choose $C' > 26/c+1$,
so $N\ge C'\frac{\log|G|}{\log|S|}-1$ implies $N \ge \frac{26\log|G|}{c\log|S|}$, whence
$$c\frac{\log|S|}{\log|G|} > \frac{26}{N}.$$
By Theorem~\ref{main}, this implies
$|\chi(g)| \le \chi(1)^{1-\frac{26}{N}}$,
so
$$|\chi(g)|^N\le \chi(1)^{N-26}.$$

To prove the theorem, it suffices to prove that for all $x\in G$, the probability that the product of $N$ i.i.d. random variables with distribution $U_S$ gives $x$ is within $1/e|G|$ of $1/|G|$.
By the Frobenius formula, this probability is
$$\frac 1{|G|}\sum_{\chi\in \Irr(G)} \frac{\chi(g)^N\bar\chi(x)}{\chi(1)^{N-1}}.$$
By \cite[Theorem~1.1]{FG}, $|\Irr(G)| \le 27.2 q^r$, if $G$ is of Lie type of rank $r$ defined over $\F_q$, and 
$\chi(1) > q^{r/3}$ by \cite{LSe} when $1_G \neq \chi \in \Irr(G)$.
Now,
$$\sum_{\chi\neq 1_G} \frac{|\chi(g)^N\bar\chi(x)|}{\chi(1)^{N-1}} \le \sum_{\chi\neq 1_G}\chi(1)^{-24} \frac{|\chi(g)|^N}{\chi(1)^{N-26}}\le \sum_{\chi\neq 1_G} \chi(1)^{-24}\le \frac{|\Irr(G)|}{\min_{\chi\neq 1_G} \chi(1)^{24}}.$$
This is less than
$$\frac{27.2 q^r}{q^{8r}} \le \frac{27.2}{128} < \frac 1e.$$
In fact, given any $\eps > 0$, we have $27.2/q^{7r} < \eps$, except possibly for a finite number of possibilities for 
$(q,r)$. 
This proves Lubotzky's conjecture \cite[p. 179]{Lub} (since, as noted above, the diameter of $\Gamma(G,S)$ is 
of the same magnitude as $(\log |G|)/(\log |S|)$.
\end{proof}

\begin{prop}\label{mixing2}
There exists an absolute constant $C''>0$ such that if $S$ is a non-trivial conjugacy class in a finite simple group $G$ of Lie type, then the mixing time of
the random walk on the Cayley graph $\Gamma(G,S)$ is greater than $C''\frac{\log|G|}{\log |S|}$.
\end{prop}

\begin{proof}
Since $\frac{\log|G|}{\log |S|}$ is bounded in bounded rank, we may assume without loss of generality that the rank of $G$ is as large as we wish,
in particular that $G$ is classical.  By Corollary~\ref{supp-size}, \edt{$\frac{\log |G|}{\log |S|} \le \frac n{5\supp(g)}$, so}
it suffices to prove that the mixing time for $\Gamma(G,S)$ is greater than $m := \bigl\lfloor \frac n{2 \supp(g)}\bigr\rfloor$, where $S=g^G$.
Every element in $S^m$ has support $\le m\supp(g) \le n/2$ by Lemma \ref{supp-prod}.  
Therefore, the characteristic polynomial of \edt{each such element} has at least $n/2$ irreducible factors.  By \cite[Proposition~3.4]{LaSh3},
the proportion of elements of $G$ satisfying this property is $o(1)$, so 
$$\Vert U_{g^G}^{\ast m} - U_G\Vert_1 = 2-o(1) >  \frac 1e.$$
\end{proof}

Theorem \ref{mixing} and Proposition \ref{mixing2} imply that the diameter and the mixing time of $\Gamma(G,S)$ are linearly 
bounded in terms of one another, and so are of the same magnitude, for all non-trivial conjugacy classes $S$ in all simple groups 
of Lie type $G$; \edit{and they have the same magnitude as $\mathrm{rank}(G)/\supp(g)$, as shown by Corollary \ref{supp-size}}.

\subsection{McKay graphs and products of irreducible characters}

Let $G$ be any finite group, $\Irr(G)$ the set of irreducible characters, and $\chi$ a complex character of $G$.  
Recall \cite{LiST} that the \emph{McKay graph} $\cM(G,\chi)$ associated to $\chi$
is the directed graph on vertex set $\Irr(G)$ such that there is an edge from $\chi_1$ to $\chi_2$ if and only if $\chi_2$ is a constituent of $\chi\chi_1$.
This graph is connected if and only if $\chi$ is faithful \edt{\cite[Chapter XV, Theorem IV]{Burnside}}.
One also considers random walks on $\cM(G,\chi)$,
starting from any vertex $\alpha \in \Irr(G)$ and with
the transition
probability from vertex $\chi_1$ to vertex $\chi_2$ equal to $\langle \chi\chi_1,\chi_2 \rangle_G \cdot \chi_2(1)/\chi(1)\chi_1(1)$
(proportional to the dimension of the $\chi_2$-homogeneous component in a representation affording $\chi\chi_1$), see \cite[\S1]{F}.

For groups of Lie type, we now settle in the affirmative a question of Liebeck, Shalev, and Tiep \cite[Conjecture 1]{LiST}
(note that the case of alternating groups is handled in \cite[Theorem 2]{LiST2}). \edit{In fact, we prove a slightly stronger result,
in which, by $\chi^*(1)$ we mean the sum of degrees of distinct irreducible constituents of a character $\chi$. As discussed in 
\cite{LiST}, this upper bound on the diameter is optimal.}

\begin{thm}\label{mckay}
There exists an absolute constant $\gamma$ such that for every finite simple group $G$ and every \edit{faithful  
(not necessarily irreducible) character $\chi$ of $G$}, the diameter of the McKay graph of $\chi$ is less than 
$\gamma\frac{\log |G|}{\log \chi^*(1)}$.
\end{thm}

\begin{proof}
\edit{Let $\chi_1$ be an irreducible constituent of largest degree of $\chi$. If $G$ is of Lie type, then 
the results of \cite{FG} and \cite{LSe} imply that 
$\chi_1(1) \geq k(G)^{1/6}$. It follows that $\chi_1(1) \leq \chi^*(1) \leq \chi_1(1)^7$. If $G = \mathsf{A}_n$, then for 
$n$ sufficiently large, \cite[Theorem 1.1(i)]{LiSh4} implies that the number of distinct irreducible characters of $G$ of 
degree $\leq \chi_1(1)$ is less than $\chi_1(1)^2$, so the same inequalities hold. Thus in all cases
$\log \chi^*(1)$ and $\log \chi_1(1)$ are of the same magnitude. Hence it suffices to prove the conjecture in the case $\chi$ is 
irreducible.}
 
The theorem is known for alternating groups \cite[Theorem 2]{LiST2} and for Lie-type groups of bounded rank \edt{$r$} \cite[Theorem 2]{LiST}, so without loss of generality, we may assume $r\ge 9$ and $G$ is classical.
Choose $\gamma = 7/c+1$, \edt{where $c$ is defined as in Theorem~\ref{main}.}
Let $\chi_1$ and $\chi_2$ denote irreducible characters of a finite simple group $G$ of classical type.  Let 
$$N := \biggl{\lceil} \frac{7}{c} \cdot \frac{ \log |G|}{\log \chi(1)} \biggr{\rceil} \le \gamma \frac{\log |G|}{\log |\chi(1)|}.$$
Then,
$$\langle \chi^N\chi_1,\chi_2\rangle_G = \frac 1{|G|}\sum_{g\in G} \chi(g)^N\chi_1(g)\overline\chi_2(g) = \frac 1{|G|}\sum_{S=g^G\subset G} |S|\chi(g)^N\chi_1(g)\overline\chi_2(g),$$
where the last sum is taken over conjugacy classes $S=g^G$.  To prove this is non-zero, it suffices to prove 
$$\sum_{S=g^G\neq \{1\}} |S||\chi(g)|^N|\chi_1(g)||\overline\chi_2(g)| < \chi(1)^N \chi_1(1)\chi_2(1).$$
As $|\chi_i(g)| \le \chi_1(1)$, it suffices to prove that
$$\sum_{S=g^G\neq \{1\}} |S|\Bigl(\frac{|\chi(g)|}{\chi(1)}\Bigr)^{N}  < 1.$$
By Theorem~\ref{main} and since $N \geq (7/c)(\log |G|)/(\log \chi(1))$, it suffices to prove 
%
$$\sum_{S\neq \{1\}} |S|\chi(1)^{-7\frac{\log |S|}{\log \chi(1)}} = \sum_{S \neq \{1\}}\frac{1}{|S|^6} < 1.$$
%
Clearly,
$$\sum_{S\neq \{1\}} \frac 1{|S|^{6}} < \frac{k(G)}{\min_{S\neq \{1\}} |S|^{6}}.$$
By \cite[Theorem~1.1]{FG}, $|\Irr(G)| \le 27.2 q^r$.
Now, $|S|$ is the degree of a permutation representation of $G$, so 
by \cite[Table 5.2.A]{KlL}, this is greater than $q^r$, so $|S|^{6} > q^{6r} > 32 q^r$.
\end{proof}

Random walks on some McKay graphs defined for $\mathsf{S}_n$ and $\GL_n(q)$ are studied in \cite[Theorems 4.1, 5.1]{F}. The following result determines the asymptotic of the convergence rate (to the stationary distribution)  of random walks on general McKay graphs for simple groups of Lie type.

\begin{thm}\label{mixing3} 
There exist absolute constants $C_1 > C_2 > 0$ such that the following statements hold for any non-trivial irreducible character
$\chi$ of any finite simple group $G$ of Lie type. The convergence rate of the random walk on the McKay graph of $\chi$ starting from any vertex $\alpha \in \Irr(G)$ is less than $C_1 \frac{\log |G|}{\log \chi(1)}$, and more 
than $C_2 \frac{\log |G|}{\log \chi(1)}$ if $\alpha=1_G$.
\end{thm}

\edt{This means that in any sequence of examples, the difference in total variation, between the stationary probability distribution and the  distribution obtained from any initial state after at least $C_1 \frac{\log |G|}{\log \chi(1)}$ steps, goes to $0$, while the total variation difference after at most $C_2 \frac{\log |G|}{\log \chi(1)}$ steps remains bounded away from $0$.}
\begin{proof}
Let $K^l_\alpha$ denote the probability measure given by taking $l$ steps from the starting vertex $\alpha \in \Irr(G)$, and let
$\pi$ denote the stationary distribution, which is known to be the Plancherel measure $\pi(\beta) = \beta(1)^2/|G|$, see \cite[\S2]{F}. If 
$$\Vert P-Q\Vert = \frac{1}{2}\sum_{\beta \in \Irr(G)}|P(\beta)-Q(\beta)|$$ 
denotes the total variation distance between two probabilistic measures 
$P$ and $Q$ on $\Irr(G)$, then \cite[Lemmas 2.1, 3.1]{F} shows that
$$4\Vert K^l_\alpha-\pi\Vert ^2 \leq \sum_{S=x^G \neq \{1\}}\bigl{|}\frac{\chi(x)}{\chi(1)}\bigr{|}^{2l}|S|\bigl{|}\frac{\alpha(x)}{\alpha(1)}\bigr{|}^{2}.$$
Clearly, $|\alpha(x)| \leq \alpha(1)$. Applying Theorem A and choosing $l \geq (4/c)\log |G|/\log \chi(1)$, we obtain 
$|\chi(x)/\chi(1)|^{2l} \leq \chi(1)^{-8\log_{\chi(1)}|S|} = |S|^{-8}$, and so 
$$4\Vert K^l_\alpha-\pi\Vert ^2 \leq \sum_{S=x^G \neq \{1\}}\frac{1}{|S|^7},$$
which is less than $1/q^r$ if $G$ is of rank $r$ over $\F_q$, as shown in the proof of \edt{Theorem \ref{mckay}}.

For the lower bound, for any $l < (1/4) \frac{\log |G|}{\log \chi(1)}$, we see that $\chi^l$ cannot contain any irreducible character $\beta$ of 
degree $\geq |G|^{1/4}$ (e.g. the Steinberg character), and thus $K^l_{1_G}(\beta)=0$. Taking $C_2$ small enough, we may 
assume that the rank $r$ of $G$ is large enough, so that $|\Irr(G)| \leq 27.2 q^r \leq |G|^{1/3}$. For such $l$ and $G$, now we have
$$2\Vert K^l_{1_G}-\pi\Vert  \geq \sum_{\beta(1) \geq |G|^{1/4}}\frac{\beta(1)^2}{|G|} 
   = 1 -  \sum_{\gamma(1) < |G|^{1/4}}\frac{\gamma(1)^2}{|G|} > 1- \frac{|G|^{1/2}|\Irr(G)|}{|G|} \geq 1-|G|^{-1/6} > 2/3.$$ 
\end{proof}

The next result generalizes Theorem \ref{mckay} and proves a conjecture of Gill \cite{Gi}. Note that the case 
$G = \PSL_n(q)$ or $\mathrm{PSU}_n(q)$, with $q$ large enough compared to $n$, was handled in \cite[Theorem 3(ii)]{LiST2};
on the other hand, the case of alternating groups is still open.

\begin{thm}\label{mckay2}
There exists an absolute constant $\delta$ such that for all finite simple groups of Lie type $G$ and all non-trivial 
$\chi_1, \chi_2, \ldots ,\chi_m \in\Irr(G)$, if $\chi_1(1)\chi_2(1) \ldots \chi_m(1) \geq |G|^{\delta}$, then 
$\chi_1 \chi_2 \ldots \chi_m$ contains every irreducible character of $G$.
\end{thm}

\begin{proof}
Suppose $G$ has bounded rank $r \leq l$. Taking $\delta \geq 245l^2$, we see that the condition 
$\prod^m_{i=1}\chi_i(1) \geq |G|^\delta$ implies that $m \geq 490l^2$, since $|\chi_i(1)| \leq |G|^{1/2}$. It follows from
\cite[Theorem 3(i)]{LiST2} that $\prod^m_{i=1}\chi_i$ contains $\Irr(G)$.

Hence we may assume that $r\ge 9$ and $G$ is classical.
Choose $\delta \geq 7/c$.  For any $\theta \in \Irr(G)$, we have 
$$\langle \prod^m_{i=1}\chi_i,\theta\rangle_G = \frac 1{|G|}\sum_{g\in G} \prod_i\chi_i(g)\overline\theta(g) = 
\frac 1{|G|}\sum_{S=g^G\subset G} |S|\prod_i\chi_i(g)\overline\theta(g),$$
where the last sum is taken over conjugacy classes $S=g^G$.  To prove this is non-zero, it suffices to prove 
$$\sum_{S=g^G\neq \{1\}} |S||\prod_i\chi_i(g)||\overline\theta(g)| < \prod_i\chi_i(1)\theta(1).$$
As $|\theta(g)| \le \theta(1)$, it suffices to prove that
$$\sum_{S=g^G\neq \{1\}} |S|\prod_i\frac{|\chi_i(g)|}{\chi_i(1)}  < 1.$$
By Theorem~\ref{main} we have $|\chi_i(g)/\chi_i(1)| \leq \chi_i(1)^{-c\log_{|G|}|S|}$, hence 
$$\sum_{S=g^G\neq \{1\}} |S|\prod_i\frac{|\chi_i(g)|}{\chi_i(1)} \leq \sum_{S \neq \{1\}}|S|\bigl(\prod_i\chi_i(1) \bigr)^{-c\log_{|G|}|S|}.$$
Since $\prod_i\chi_i(1) \geq |G|^{\delta}$, we now have 
$$\sum_{S=g^G\neq \{1\}} |S|\prod_i\frac{|\chi_i(g)|}{\chi_i(1)} \leq \sum_{S\neq \{1\}} |S|\cdot |S|^{-c\delta} \leq \sum_{S \neq \{1\}}\frac{1}{|S|^6} < 1,$$
the last inequality already established in the proof of Theorem \ref{mckay}.
\end{proof}

The next result proves \cite[Conjecture 4]{LiST2} for simple groups of Lie type.

\begin{cor}\label{mckay3}
There exists an absolute constant $\delta'$ such that for all finite simple groups of Lie type $G$ of rank $r$ and all non-trivial 
$\chi_1, \chi_2, \ldots ,\chi_m \in\Irr(G)$, if $m \geq \delta'r$, then 
$\chi_1 \chi_2 \ldots \chi_m$ contains every irreducible character of $G$.
\end{cor}

\begin{proof}
Take $\delta' = 12\delta$, with $\delta$ the constant in Theorem \ref{mckay2}. Since $\chi_i(1) \geq q^{r/3}$ by \cite{LSe}
and $|G| \leq q^{4r^2}$ if $G$ is defined over $\F_q$,
we have $\prod^m_{i=1}\chi_i(1) \geq q^{12\delta r^2/3} \geq |G|^\delta$. Hence the statement follows from Theorem \ref{mckay2}.
\end{proof}

Taking $\chi_1 = \ldots = \chi_m = \chi$ in Corollary \ref{mckay3}, we obtain the following consequence, which was proved
in \cite[Theorem 3]{LiST} for $q$ large enough compared to $n$ (but with a much smaller constant).

\begin{cor}\label{mckay4}
There exists an absolute constant $\delta'$ such that for all finite simple groups $\PSL_n(q)$ and $\mathrm{PSU}_n(q)$ and all non-trivial 
$\chi \in\Irr(G)$, if $m \geq \delta'(n-1)$, then 
$\chi^m$ contains every irreducible character of $G$.
\end{cor}

\subsection{Power word maps on simple groups}
Recall that $\GL^\eps_n(q)$ denotes $\GL(\F_q^n)$ if $\eps=+$ and $\GU(\F_{q^2}^n)$ if $\eps=-$, and similarly for 
$\SL^\eps_n(q)$ and $\PSL^\eps_n(q)$. The notion of the {\it level $\cl(\chi)$} of a character $\chi$ of $\GL^\eps_n(q)$ and $\SL^\eps_n(q)$ was introduced in \cite[Definitions 1, 2]{GLT}. The following 
result gives a somewhat better bound than \cite[Theorem 1.6(iii), (iv)]{GLT}, which is needed in Theorem~\ref{burnside} below. 

\begin{prop}\label{level-glu}
Let $q$ be any prime power, $n \geq 1$, $G = \GU_n(q)$ or $\SU_n(q)$, and $\chi \in \Irr(G)$.
\begin{enumerate}[\rm(i)]
\item If $\cl(\chi) \leq \sqrt{n-3/4}-1/2$, then $|\chi(g)| < 1.93\chi(1)^{1-1/n}$
for all $g \in G \smallsetminus \ZB(G)$.
\item If $\cl(\chi) \leq \sqrt{n/2-1}$, then 
$|\chi(g)| < 1.93\chi(1)^{\max(1-1/2\cl(\chi),1-\supp(g)/n)}$ 
for all $g \in G$.
\end{enumerate}
\end{prop}

\begin{proof}
We follow the proof of \cite[Theorem 1.6(iii), (iv)]{GLT}, and assume first that $j:=\cl(\chi) \geq 3$. Using 
\cite[Lemma 5.1(iii)]{GLT} we see that 
$$|\GU_j(q)|/q^{j^2} \leq (q+1)(q^2-1)(q^3+1)/q^6 < 1.266.$$
Hence we can use \cite[(8.18)]{GLT} with the improved bound $|S| < 1.266q^{j^2}$ for $S := \GU_j(q)$, 
which leads to the improved upper bound 
$0.747q^{(n-1)j}$ for  $|S|\bigl(q^{n(j-1)}+\sqrt{16.52q^{j^2+j-1}}q^{n(j-2)}\bigr)$, and obtain 
$$|\chi(1)| \geq \frac{q^{nj}(1-0.747q^{-j})}{|S|\alpha(1)} > \frac{0.906q^{nj}}{|S|\alpha(1)},~~|\chi(g)| < \frac{1.747q^{(n-1)j}}{|S|\alpha(1)}$$ 
if $\chi=D^\circ_\alpha$ for $\alpha \in \Irr(S)$ in the notation of \cite[Theorem 1.1]{GLT}. 

Suppose $j=2$, and let $\chi=D^\circ_\alpha$ for $\alpha \in \Irr(S)$ and $k:=n-\supp(g)$. In the notation of \cite[\S8.3]{GLT}, $N' =1$, 
so $D'_\alpha(1) \leq q^2\sqrt{2}$ in \cite[(8.17)]{GLT}, and instead of 
\cite[(8.18)]{GLT} we now have
\begin{equation}\label{eq-gu1}
  \chi(1) \geq \frac{q^{2n}-|S|(q^n+q^2\sqrt{2})}{|S|/\alpha(1)},~~|\chi(g)| \leq \frac{q^{2k}+|S|(q^n+q^2\sqrt{2})}{|S|/\alpha(1)}.
\end{equation}
Also, $|S|=|\GU_2(q)| \leq 1.125q^4$, hence $|S|(q^n+q^2\sqrt{2})$ is less than $0.588q^{2n-2}$ when $n \geq 7$ and less than $0.588q^{1.5n}$ when $n \geq 10$. Now we can repeat the rest of the proof of \cite[Theorem 1.6(iii), (iv)]{GLT} verbatim to obtain
the result for $j \geq 2$. 

The estimates are trivial if $j=0$ or if $g \in \ZB(G)$, i.e. $k:=n-\supp(g)=0$. If $j=1$, then, as shown in the proof \cite[Theorem 1.6(iii)]{GLT}, 
we have $\chi(1) \geq (q^n-q)/(q+1)$, and 
$$|\chi(g)| \leq \frac{q^{n-1}+q}{q+1},~|\chi(g)| \leq \left\{ \begin{array}{ll}q^k < 1.93\chi(1)^{1/2}, & k \leq (n-1)/2,\\
   (2q^k+q)/(q+1) < 1.93\chi(1)^{k/n}, & k \geq n/2,\end{array}\right.$$
yielding the result. 
\end{proof}
  
\begin{prop}\label{pairs}
There exists an integer $N \geq 1$ such that the following statement holds for any prime power $q$, any integer $n \geq N$, any 
integer $a$ with $n/3 \leq a \leq 2n/3$, and any $\eps=\pm$. If $G = \SL^\eps_n(q)$ and $s,t \in G$ are regular semisimple elements
belonging to maximal tori $T_1$ of type $T_{a,n-a}$ and $T_2$ of type $T_{a+1,n-a-1}$, then 
$s^G \cdot t^G$ contains every non-central element $g \in G$, except possibly when $q=2$ and $g$ is a scalar multiple of a transvection.
\end{prop}

\begin{proof}
Let $\chi \in \Irr(G)$ be such that $\chi(s)\chi(t) \neq 0$. As shown in the proof of \cite[Proposition 8.4]{GLBST}, $\chi$ is a 
unipotent character $\chi^\lambda$ labeled by a partition $\lambda \vdash n$, and $|\chi(s)\chi(t)| \leq 16$. The choices for $\lambda$ 
are listed in \cite[Corollary 3.1.3]{LST}: either $\lambda=(n-j,1^j)$ with $0 \leq j \leq n-1$, or its largest part $\lambda_1$ 
satisfies $n-\lambda_1 \geq \min(a,n-a)-1 \geq n/3-1$, and there are at most $4an \leq 8n^2/3$ of them.

\smallskip
Since $g \notin \ZB(G)$, we have that $\supp(g) \geq 1$. Note that $\cl(\chi)=n-\lambda_1$ by \cite[Theorem 3.9]{GLT}. Next, by \cite[Theorem 1.3]{GLT}, if $\cl(\chi) \geq n_0:=\sqrt{n}/5$, then $\chi(1) > q^{n_0(n-n_0)-3}$. Hence, applying Theorem \ref{main-bound3} we now have
$$\Sigma_1:= \sum_{\chi \in \Irr(G),~\cl(\chi) \geq n_0}\frac{|\chi(s)\chi(t)\bar\chi(g)|}{\chi(1)}\leq \frac{16 \cdot 8n^2/3}{q^{\sigma(n_0(n-n_0)-3)/n}}
= O\bigl(\frac{n^2}{2^{\sqrt{n}/6}}\bigr).$$
Choosing $N$ large enough, we have that $\Sigma_1 \leq 0.01$.

\smallskip
Now we look at $\chi$ with $\chi(s)\chi(t) \neq 0$ and $\cl(\chi) < n_0$. By the above considerations, $\chi= \chi^{(n-j,1^j)}$ 
with $0 \leq j = \cl(\chi) < n_0 < n/3$. The Murnaghan-Nakayama rule applied to $\chi(s)$, $\chi(t)$ and the hook partition 
$(n-j,1^j)$ shows that $|\chi(s)\chi(t)| =1$, see \cite[Proposition 3.1.1, Corollary 3.1.2]{LST}. On the other hand,
since $\cl(\chi) < n_0$, \cite[Theorem 1.6(ii)]{GLT} and Proposition \ref{level-glu}(ii) apply to $\chi$ (when $N$ is large enough) and yield 
$|\chi(g)| < 1.93\chi(1)^{1-1/n}$. We also have by \cite[Lemma 4.1]{LMT} that 
\begin{equation}\label{eq-gu2}
  \chi^{(n-j,1^j)}(1) = q^{j(j+1)/2}\frac{\prod^{n-1}_{i=n-j}(q^i-\eps^i)}{\prod^{j}_{i=1}(q^i-\eps^i)} > q^{nj-j(j+1)/2-2}.
\end{equation}  
In particular, when $N$ is large enough, $\chi(1)^{1/n} > q^{j-1/49}$, and so
$$\Sigma_2:=\sum_{\chi \in \Irr(G),~1 \leq\cl(\chi) < n_0}\frac{|\chi(s)\chi(t)\bar\chi(g)|}{\chi(1)} 
   \leq \sum^{n_0}_{j=1}\frac{1.93}{q^{j-1/49}} < \sum^{\infty}_{j=1}\frac{1.93q^{1/49}}{q^j}.$$
If $q \geq 3$, then $\Sigma_2 < 0.99$ and so $\Sigma_1 + \Sigma_2 < 1$, showing $g \in s^G \cdot t^G$.   

\smallskip
From now on we assume $q=2$. If $g$ is semisimple, then $g \in s^G \cdot t^G$ by \cite[Lemma 5.1]{GT}. Hence we may assume
that $\supp(g) \geq 2$. First we note that 
$$\Sigma_3:=\sum_{\chi \in \Irr(G),~4 \leq\cl(\chi) < n_0}\frac{|\chi(s)\chi(t)\bar\chi(g)|}{\chi(1)} 
   \leq \sum^{n_0}_{j=4}\frac{1.93}{2^{j-1/49}} < \sum^{\infty}_{j=4}\frac{1.93\cdot 2^{1/49}}{2^j} < 0.25.$$
Next we bound $|\chi(g)/\chi(1)|$ for $1 \leq j \leq 3$. If $j=1$, then
$\chi(1) \geq (2^n-2)/3$ and $|\chi(g)| \leq (2^{n-2}+4)/3$ by \cite[Lemma 4.1]{TZ1}, hence $|\chi(g)|/\chi(1) < 0.26$ when 
$N$ is large enough. For $j=2$, \eqref{eq-gu1} implies $|\chi(g)|/\chi(1) < (1.1)q^{-4} < 0.07$ when $N$ is large enough.
When $j=3$, $\chi(1) > q^{3n-8}$ by \eqref{eq-gu2}, and so $|\chi(g)/\chi(1)| < 1.93\chi(1)^{-1/n} < (1.1)(1.93)q^{-3} < 0.27$ (when
$N$ is large enough) by \cite[Theorem 1.6(ii)]{GLT} and Proposition \ref{level-glu}(ii). It follows that 
$$\sum_{1_G \neq \chi \in \Irr(G)}\frac{|\chi(s)\chi(t)\bar\chi(g)|}{\chi(1)} \leq \Sigma_1 + \Sigma_3 + 0.26 + 0.07 + 0.27 = 0.86,$$
again showing $g \in s^G \cdot t^G$.    
\end{proof}

Now we can answer an open question raised in \cite{GLBST} and prove the following result, which strengthens Theorems 4 and 5 of \cite{GLBST}. As shown in \cite[Example 8.10]{GLBST}, the statement does not hold for simple groups of Lie-type of bounded rank.

\begin{thm}\label{burnside}
There exists a function $f:\Z_{\geq 1} \to \Z_{\geq 1}$ such that the following statement holds. For any integer $k \geq 1$ and 
any integer $N \geq 1$ with  
at most $k$ distinct prime divisors, the power word map $(x,y) \mapsto x^Ny^N$ is surjective on any alternating group
$\mathsf{A}_n$ with $n \geq f(k)$ and any simple classical group of rank $r \geq f(k)$.
\end{thm} 

\begin{proof}
Fix any $k \geq 1$.
By \cite[Theorem 4]{GLBST}, it suffices to prove the theorem for any finite classical group $S = \PSL^\eps_n(q)$ with $n$ sufficiently
large. Recall \cite{Zs} that if $m \geq 7$, then $(\eps q)^m-1$ admits a primitive prime divisor $\ell_m$, that is a prime divisor 
which is coprime to $\prod^{m-1}_{i=1}((\eps q)^i-1)$. Choosing $n \geq 12k+24$, we can find $k+1$ 
integers $a_i$, $1 \leq i \leq k+1$, such that
\begin{equation}\label{eq-pr1}
  n/3 \leq a_1 < a_2 < \ldots < a_{k+1} < n/2,~a_{i+1}-a_i \geq 2 \mbox{ for all }i.
\end{equation}  
Then, for each $i$, we can find a regular semisimple element $s_i \in G:=\SL^\eps_n(q)$ of order 
$\ell_{a_i}\ell_{n-a_i}$ belonging to a maximal torus of type $T_{a_i,n-a_i}$ and a regular semisimple element $t_i \in G$ 
of order $\ell_{a_i+1}\ell_{n-a_i-1}$ belonging to a maximal torus of type $T_{a_i+1,n-a_i-1}$. Condition \eqref{eq-pr1} ensures 
that $\gcd(|s_i|\cdot|t_i|,|s_j|\cdot|t_j|)=1$ whenever $i \neq j$. Since $N$ has most $k$ distinct prime factors, it follows that $N$ is 
coprime to $|s_{i_0}|\cdot|t_{i_0}|$ for some $i_0$ and so both $s_{i_0}$ and $t_{i_0}$ are $N^{\mathrm {th}}$ powers in $G$.

Now assume $n$ is sufficiently large and consider any $g \in G \smallsetminus \ZB(G)$. 
If $q \geq 3$, or if $q=2$ but $g$ is not a scalar multiple of a transvection, then 
$g$ belongs to $s_{i_0}^G \cdot t_{i_0}^G$ by Proposition \ref{pairs},
and so it is a product of two $N^{\mathrm {th}}$ powers.  Suppose now that 
$q=2$ and $g$ is a transvection. Since $n \geq 12k+24$, we can find $k+1$ odd integers 
$9 \leq n_1 < n_2 < \ldots < n_{k+1} < n$. By \cite[Theorem 2.1]{GLBST}, for each $i$, $g$ embedded in $\SL^\eps_{n_i}(q)$ is 
a product $u_iv_i$ where $|u_i|=\ell_{n_i}$ and $|v_i| = \ell_{n_i-1}$. Arguing as above, we see that $N$ is coprime to 
$|u_{j_0}| \cdot |v_{j_0}|$ for some $j_0$, hence $g=u_{j_0}v_{j_0}$ is again a product of two $N^{\mathrm {th}}$ powers. 
\end{proof}

\subsection{Fibers of product morphisms on semisimple algebraic groups}

Our character estimates have consequences for the geometry of semisimple algebraic groups
in all characteristics, of which the following result is a sample.

\begin{thm}\label{fibers}
There exists a constant $C$ with the following property.
Let $K$ be an algebraically closed field and $\uG$ a simple algebraic group over $K$.
Let $\uS_1,\ldots,\uS_k$ be conjugacy classes in $\uG$, and $\uX := \uS_1\times \cdots \times \uS_k$.
If $\dim \uX \ge C\dim \uG$, then the multiplication morphism $\mu_K\colon \uX\to \uG$
is  flat.
\end{thm}

\begin{proof}
As conjugacy classes are non-singular varieties, $\uX$ and $\uG$ are both non-singular, so by miracle flatness, the theorem is equivalent to the statement that for all $g\in \uG(K)$, $\mu_K^{-1}(g)$ has dimension
$\dim \uX - \dim \uG$. 

If $\tilde{\uG}$ denotes the simply connected cover of $\uG$, 
$\tilde{\uS}_i := \uS_i\times_{\uG}\tilde{\uG}$,  $\tilde{\uX} := \prod_i \uS_i$, 
and $\tilde\mu_K\colon \tilde{\uX}\to \tilde{\uG}$ denotes the product morphism,
then the natural morphisms $\pi_G\colon \tilde{\uG}\to \uG$ and $\pi_X\colon \tilde{\uX}\to \uX$ are
finite and surjective.  If $g\in \uG(K)$, then 
$$\mu^{-1}(g) = \bigcup_{\tilde g\in \pi_G^{-1}(g)} \pi_X(\tilde\mu^{-1}(\tilde g)),$$
and
$$\dim \pi_X(\tilde \mu_K^{-1}(\tilde g)) = \dim \tilde\mu^{-1}(\tilde g).$$
We may therefore reduce to the case that $\uG$ is simply connected.

Next, we assume that $K$ is algebraic over $\F_p$ for some prime $p$.
If $\uG_0$ denotes the split, simply connected simple algebraic group over $\F_p$
with the same Dynkin diagram as $\uG$, then $\uG_0\times_{\F_p} K$ is isomorphic to $\uG$.
We fix an isomorphism. 
Via this isomorphism, all varieties $\uS_i$ are defined over some common finite extension $\F_q$ of 
$\F_p$.  

Fixing $q$, we define $\tilde G := \uG_0(\F_q)$, $\tilde S_i := \uS_i(\F_q)$, $\tilde X := \tilde S_1\times \cdots\times \tilde S_k$, and 
the multiplication map
$\tilde \mu_q\colon \tilde X\to \tilde G$. 
It suffices to prove that for $\tilde g\in \tilde G$, $\tilde \mu_q^{-1}(\tilde g)$ has $O(q^{\dim \uX-\dim \uG})$ elements, where the implicit constant depends on $\uG$ and the $\uS_i$ but not on $q$.

Let $Z$ denote the center of $\tilde G$, and $G := \tilde G/Z$.  
Let $S_i$, $X$, and $\mu$ denote the counterparts for $G$ to $\tilde S_i$, $\tilde X$, and $\tilde \mu$.
Then
$$\mu^{-1}(g) = \sum_{\{\tilde g\in \tilde G\mid \pi_G(\tilde g) = g\}} \pi_X(\tilde \mu_q^{-1}(\tilde g)),$$
so it suffices to prove that $\mu^{-1}(g) = O(q^{\dim \uX-\dim \uG}).$

Now,
$$\pi^{-1}(g) = \{(s_1,\ldots,s_k)\in X\mid s_1\cdots s_k=1\}.$$
Writing $S_i=x_i^{G}$, we have
$$|\pi^{-1}(g)| = \frac {|X|}{|G|} \sum_{\chi\in \Irr(G)} \frac{\chi(x_1)\cdots \chi(x_k)\bar\chi(g)}{\chi(1)^{k-1}}.$$

We claim that each \edt{$\tilde S_i$} is a union of $O(1)$ conjugacy classes in $G$, where the implicit constant does not depend on $\uS_i$ or $q$.  
This follows from the analogous claim for $\tilde G$-conjugacy classes in $\tilde S_i$.  To prove this, consider
the subvariety $\uW$ of $\uG_0\times \uG_0$ consisting of commuting pairs $(\tilde g_1,\tilde g_2)$.  The fiber $\uW_{\tilde g_1}$ of this variety over $\tilde g_1\in \uG_0(\F_q)$ is the centralizer of $\tilde g_1$ in $\uG_0$.  By \cite[Corollaire~9.7.9]{EGA4}, the number of geometric components of the fiber is a constructible function on $\uG_0$, so it is bounded above.  The number of points of the algebraic group $\uW_{\tilde g_1}$ over 
$\F_{q}$ is at most $C (q+1)^{\dim \uW_{\tilde g_1}}$, where $C$ is the number of components.  Likewise, the number of points of $\uG_0$ over $\F_{q}$ is at least $(q-1)^{\dim \uG_0}$.
Thus, the size of the conjugacy class of $\tilde g_1$ in $\uG_0(\F_q)$ is bounded below by a positive constant multiple of $q^{\dim \uG_0-\dim \uW_{\tilde g_1}}$.
By the Lang-Weil estimate, the number of $\F_q$-points on the conjugacy class of $\tilde g_1$ in the algebraic group $\uG_0$ is at least $(1-o(1))q^{\dim \uG_0-\dim \uW_{\tilde g_1}}$.  This gives an upper bound on the number of $\uG(\F_q)$-conjugacy classes in $S_i$ and therefore
an upper bound on the number of $G$-conjugacy classes in $S_i$.

We may therefore pick representatives $x_i$ of each $S_i$ and prove that the number of $k$-tuples $(s_1,\ldots,s_k)$
such that $s_1\cdots s_k=1$ and each $s_i$ is conjugate in $G$ to $x_i$ is $O(|X|/|G|)$.
By Theorem~\ref{main},
$$\frac{|\chi(x_1)\cdots \chi(x_k)\bar \chi(g)|}{\chi(1)^{k-1}} \le \chi(1)^{2-c\frac{\sum_{i=1}^k \log |S_i| }{\log |G|}}.$$
As $\log |S_i| = (\dim \uS_i)\log q + O(1)$, we have
$$\sum_{i=1}^k \log |S_i| = (\dim \uX) \log q + O(1),$$
so
$$|\pi^{-1}(g)|  = \frac{|X|}{|G|}\Bigl(1+O\bigl(\sum_{\chi\neq 1}\chi(1)^{2-c\frac{\dim \uX}{\dim \uG} + o(1)}\bigr)\Bigr).$$
By \cite[Theorem~1.1]{LiSh3}, if \edt{$\dim \uX > \frac {2+(2/h)}c\dim \uG$}, then 
$$|\pi^{-1}(g)| = \frac{|X|}{|G|}(1+o(1)) = O(q^{\dim \uX-\dim \uG}),$$
\edt{where the $o(1)$ term goes to $0$ independently of the choices of conjugacy classes as $q\to\infty$.}
This implies the theorem for $K\cong \bar \F_q$.

For the general case, let $\cG$ denote the Chevalley scheme over $\Z$ with 
the same Dynkin diagram as $\uG$.  Fix an isomorphism between $\cG_K$
and $\uG$.
Choose representatives $x_1,\ldots,x_k$ for $S_1,\ldots,S_k$ in $\uG$.  Via the isomorphism, we can identify
all the $x_i$ as points $\cX_i$ on $\uG(A)$, where $A$ is a finitely generated $\Z$-algebra.
By \cite[Lemma~8.2]{GLT2}, there exists a dense open affine subscheme $\Spec B$ of $\Spec A$
and for each $i$ a locally closed $B$-subscheme $\cS_i$ of $\cG_B$ so that for every field $F$ and every $F$-point of $\Spec B$,
${\cS_i}_ F$ is the conjugacy class of the specialization ${\cX_i}_F$.

Now consider the multiplication morphism $\mu_B\colon \cS_1\times \cdots\times \cS_k\to \cG_B$.
By \cite[Proposition~9.5.5]{EGA4}, the set of points of $\cG_B$ over which every fiber of $\mu_B$ has 
dimension $\dim \uX-\dim \uG$ is constructible and contains every point of $\cG_B$ with finite residue field.
As $\cG_B$ is of finite type over $B$, it is of finite type over $\Z$ and therefore Jacobson \cite[Corollaire~10.4.6]{EGA4};
moreover, the closed points of $\cG_B$ are exactly the points with finite residue field
\cite[Lemme~10.4.11.1]{EGA4}.  In a Jacobson scheme, by definition, the closed points are very dense, so by 
\cite[Proposition~10.1.2]{EGA4}, the only constructible subset of $\cG_B$ containing all closed points is the whole set.
Thus, the fiber dimension condition holds for all fibers of $\mu_B$ and therefore for all fibers of $\mu_K$.
\end{proof}

We remark that, replacing the constant $C$ above by $2C+1$, we can prove that $\mu_K$ is faithfully flat.  Indeed, it suffices to prove that $\mu_K$ is surjective.
If $\dim \uX > (2C+1)\dim \uG$, then there exists $j$ such that the multiplication maps $\mu_K\colon \uX_1\times\cdots \uX_j\to \uG$ and $\nu_K\colon \uX_{j+1}\times \cdots\times \uX_k\to \uG$ are
flat and therefore dominant.  By Chevalley's theorem, the images of $\mu_K$ and $\nu_K$ are dense constructible sets.  As the intersection of two dense open subsets is non-empty, the same is true for dense
constructible subsets, and it follows that the product of two such subsets on an algebraic group covers the whole group.

\edit{A related result, in the case $\uS_1= \ldots=\uS_k$ and with the explicit constant $C=120$, was recently proved in \cite[Theorem 1]{LiSi}.}

\end{document}